\documentclass[review]{elsarticle}

\usepackage{lineno,hyperref}
\modulolinenumbers[1]

\journal{Journal of Computational Physics}

\usepackage{srcltx,graphicx}
\usepackage{amsmath, amssymb, amsthm}

\usepackage{color, xcolor}
\usepackage{array}

\usepackage{lscape}
\usepackage{algorithm}      
\usepackage[noend]{algpseudocode}
\usepackage{multirow}

\usepackage{psfrag}
\usepackage{pifont}
\usepackage{bbding}

\usepackage{hyperref}
\usepackage{makeidx}

\usepackage{caption}
\usepackage{subcaption} 

\usepackage[section]{placeins}
\usepackage{longtable}

\newtheorem{theorem}{Theorem}
\newtheorem{definition}{Definition}
{\theoremstyle{remark} \newtheorem{remark}{Remark}}
\numberwithin{figure}{subsection}
\numberwithin{equation}{section}
 
\hypersetup{linktoc = all}                
\hypersetup{hidelinks}
\hypersetup{bookmarksnumbered}
\newcommand\showfigures[1]{} 
\newcommand \be { \begin{equation}}
\newcommand \ee { \end{equation}}
\newcommand \bel {\begin{equation}\label}
\newcommand \del \partial 

\newcommand{\tb}[1]{ \boldsymbol{#1}}

\newcommand \red {\color{black}}

\newcommand \blue {\color{black}}

\newcommand \green {\color{black}}

\newcommand{\reva}[1]{{\color{black} #1}}

\newcommand{\revb}[1]{{\color{black} #1}}

\newcommand{\revc}[1]{{\color{black} #1}}

\begin{document}

\begin{frontmatter}

\title{High-order fully well-balanced numerical methods for one-dimensional blood flow with discontinuous properties{\blue, friction and gravity}}

\author[uni1]{Ernesto Pimentel-Garc\'ia}
\author[uni2]{Lucas O. Müller}
\author[uni3]{Carlos Par\'es}

\address[uni1]{Dep. Matemática Aplicada. Universidad de Málaga, CP 29010, Málaga, Spain}
\address[uni2]{Department of Mathematics, University of Trento, Trento, Italy}
\address[uni3]{Dep. Análisis Matemático, Estadística e I.O. y Matemática Aplicada. Facultad de Ciencias, Universidad de Málaga, CP 29010 Málaga, Spain}


\begin{abstract}
{\blue

We present well-balanced, high-order, semi-discrete numerical schemes for one-dimensional blood flow models with discontinuous mechanical properties and algebraic source terms representing friction and gravity. While discontinuities in model parameters are handled using the Generalized Hydrostatic Reconstruction, the presence of algebraic source terms implies that steady state solutions cannot always be computed analytically. In fact, steady states are defined by an ordinary differential equation that needs to be integrated numerically. Therefore, we resort on a numerical reconstruction operator to identify and, where appropriate, preserve steady states with an accuracy that depends on the reconstruction operator's numerical scheme. We extend our methods to deal with networks of vessels and show numerical results for single- and multiple-vessel tests, including a network of 118 vessels, demonstrating the capacity of the presented methods to outperform naive discretizations of the equations under study. 

}
\end{abstract}

\begin{keyword}
blood flow, well-balanced, friction, gravity, finite volume, networks
\end{keyword}

\end{frontmatter}

\nolinenumbers

\section{Introduction} 
\label{section--1}

{\blue

Since first put forward by Euler in 1744 \cite{Euler:1744}, one-dimensional (1D) blood flow models have been the subject of a large number of theoretical, numerical and application studies. For a comprehensive overview of works treating all these aspects the reader is referred to the recent review proposed in \cite{BlancoMuller:2025}. 

In this paper we deal with the development of numerical schemes with well-balancing properties. A scheme is said to be well-balanced if it captures exactly, or with a certain order of accuracy, certain steady state solutions of the differential model that it discretizes. While the literature on well-balanced schemes for 1D blood flow models is rich for cases in which geometrical and mechanical properties vary over space, even discontinuously \cite{delestre2013well,
ghigoLowShapiroHydrostaticReconstruction2017,
ghitti2020fully,
murillo2015roe,
liWellbalancedDiscontinuousGalerkin2018,
brittonWellbalancedDiscontinuousGalerkin2020,
spilimbergo2021one,
mullerWellbalancedHighorderSolver2013,
spilimbergo2021one,
muller2013well}, 
less was done for the case in which these models include algebraic source terms representing friction and gravity. Two exceptions are  \cite{ChuKurganov:2023} and \cite{murillo2019formulation}. The first work regarded well-balanced schemes for 1D blood flow with friction and constant mechanical properties, while the second one dealt with the construction of well-balanced schemes that preserve zero-velocity steady states for problems including gravity. In general, when friction or gravity are present, exact solutions can be computed analytically under very specific conditions, while in most general cases such solutions have to be computed numerically, using ordinary differential equations integrators. 

In previous works we proposed exactly well-balanced schemes for 1D blood flow with discontinuous geometrical and mechanical properties for the case where friction and gravity were not present \cite{muller2013well,pimentelgarcia2023high}. In these works we relied on the Generalized Hydrostatic Reconstruction (GHR) \cite{castro2007generalized}. Here, we still make use of the GHR to deal with discontinuous properties of vessels, while we adopt the methodology proposed in \cite{gomez2021collocation} for general steady state solutions, constructed by numerically integrating the Cauchy problem that defines such steady states.  Combining these two tools, we are able to construct high-order well-balanced semi-discrete path-conservative schemes for 1D blood flow in networks of vessels.

A novel aspect of this work is the extension of the proposed methodology to deal with boundary and coupling conditions needed to perform computational haemodynamic simulations. In fact, we describe how to couple multiple vessels at junction/bifurcation nodes, how to couple 1D vessels to lumped-parameter models used to describe peripheral circulation and how to prescribe boundary conditions. These features allow us to show the relevance of well-balanced schemes for the verification of a correct implementation of the effect of gravity in complex vascular network models, and to demonstrate the capacity of the proposed methods to deliver accurate results for realistic blood flow simulation setups that include friction and gravity.

The structure of the paper is as follows. In Section \ref{section--2} we present the mathematical model and define relevant quantities for the construction of well-balanced schemes. Next, Section \ref{General methodology} regards the design of the proposed numerical methods. Section \ref{section--networks} follows with the description of how our methods can be applied in real computational haemodynamic applications. Results are illustrated and discussed in Section \ref{section--results} and conclusions are drawn in Section \ref{section--conclusions}.

}

\section{Mathematical model} 
\label{section--2}

We consider the following model for 1D blood flow in thin-walled deformable elastic tubes introduced in \cite{toro2013flow}:
\begin{equation}\label{eq:model}
    \left\{
\begin{array}{l}
\displaystyle \partial_t A + \partial_x q = 0,\\
\displaystyle \partial_t q + \partial_x \left( \frac{q^2}{A} \right) + \frac{A}{\rho} \partial_x p = -\frac{f}{\rho} + {\blue g A},
\end{array}\right.
\end{equation}
with
$$
p(x,t) = K(x) \phi\left( \frac{A(x,t)}{A_0(x)}\right) + p_e(x), \quad f(x,t)=\gamma \pi \mu \frac{q}{A},
$$
where $\phi$ is the function defined by
\begin{equation}
    \phi(a)  =  a^m - a^n.
\end{equation}
 The notation is as follows:
\begin{itemize}
    \item $A(x,t)$ represents the cross-sectional area of the vessel.
    \item $q(x,t)$, the mass-flux and $u(x,t) = \dfrac{q(x,t)}{A(x,t)}$ is the averaged velocity of blood at a cross section.
    \item $\rho$, the fluid density, assumed to be constant.
    \item $A_0(x)$, the vessel cross-sectional area in an unloaded configuration.
    \item $K(x)$, the so-called stiffness coefficient, is a known positive function of the vessel wall Young modulus $E(x)$, the wall thickness $h_0(x)$ and $A_0(x)$.
    \revc{\item $p(x,t)$, the internal pressure.
    \item $p_{e}(x)$,  the external pressure acting on the vessel.}
    \item Here we assume $m>0$ and $n\in (-2,0]$. Typical values for collapsible tubes, such as veins, are $m=10$, $n=-3/2$. For arteries we have $m=1/2$, $n=0$. Moreover, the specified intervals for $m$ and $n$ guarantee the genuine nonlinearity of certain characteristic fields \cite{toro2013flow}, as explained later.
    \item $f(x,t)$ is the friction force per unit length of the tube, $\gamma$ depends on the velocity profile and $\mu$ is the fluid viscosity. Here we assume a linear profile of the velocity corresponding to $\gamma =8$. The fluid viscosity is fixed to $\mu = 0.0045$.
       \item  {\blue $g(x)$ is a smooth function representing the projection of the gravitational vector field along the direction of the $x$ axis. This parameter is allowed to vary smoothly along $x$ since vessel orientation can change significantly along this axis.}
\end{itemize}

Following \cite{pimentelgarcia2023high} we can rewrite \eqref{eq:model} as a system of balance laws in compact form: 
\begin{equation}\label{compform}
\partial_t \tb{U} + \partial_x \tb{F}(\tb{W}) + \tb{S}(\tb{W})\cdot\partial_x \tb{\sigma} = \tb{R}(\tb{W}) ,
\end{equation}
where
\begin{equation}\label{statesflux}
    \tb{U} = \left(\begin{array}{c}
         A  \\
         q 
         \end{array}\right), \quad  \tb{W} = \left(\begin{array}{c}
         A  \\
         q \\
         K \\
         A_0 \\
         p_e 
         \end{array}\right), \quad \tb{F}(\tb{W}) = \left(\begin{array}{c}
         q  \\
        \displaystyle \frac{q^2}{A} + \frac{KA_0}{\rho} \tilde \Phi\left(\frac{A}{A_0}\right)
         \end{array}\right),
\end{equation}
\begin{equation}
    \tb{S}(\tb{W}) = \left( \ \tb{S}_1(\tb{W}) \ \rvert \ \tb{S}_2(\tb{W}) \ \rvert \ \tb{S}_3(\tb{W}) \ \right), \quad \tb{\sigma} = \left( \begin{array}{c} K \\ A_0 \\ p_e \end{array} \right), \quad \tb{R}(\tb{W}) = \left( \begin{array}{c} 
    0 \\ 
    \displaystyle - \frac{\gamma \pi \mu}{\rho} \frac{q}{A} {\blue + g A}
    \end{array}\right),
\end{equation}
being
\begin{equation}\label{sources}
   \tb{S}_1(\tb{W})= \left(\begin{array}{c}
         0  \\
        \displaystyle \frac{A_0}{\rho}\Phi( a)
         \end{array}\right),
         \quad
     \tb{S}_2(\tb{W})= - \left(\begin{array}{c}
         0  \\
        \displaystyle \frac{K}{\rho} \tilde \Phi\left(a \right)
        \end{array}\right),        
        \quad
   \tb{S}_3(\tb{W})= \left(\begin{array}{c}
         0  \\
        \displaystyle  \frac{A}{\rho}  
         \end{array}\right),
\end{equation}
with
$$
\red{a = \frac{A}{A_0}}$$
and
\begin{eqnarray}
    \Phi(\revc{a}) & = & \int_{0}^\revc{a} \phi(\revc{\tau}) \,d\revc{\tau} = \left(\frac{1}{m+1}\revc{a}^{m+1} - \frac{1}{n+1}{\revc{a}}^{n+1} \right) , \\
        \tilde \Phi(\revc{a}) & = & \int_{0}^\revc{a} \revc{\tau} \phi'(\revc{\tau}) \,d\revc{\tau}  = \left(\frac{m}{m+1} \revc{a}^{m+1} - \frac{n}{n+1}\revc{a}^{n+1} \right).
\end{eqnarray}
Note that the following equality holds:
\begin{equation}\label{relacion}
 \Phi(\revc{a}) +  \tilde \Phi(\revc{a}) = \revc{a} \phi(\revc{a}). 
\end{equation}

We can rewrite $\partial_x \tb{F}(\tb{W})$ in terms of its Jacobian matrix:

\begin{equation}
    \partial_x \tb{F}(\tb{W}) = \textbf{JF}(\tb{W})\tb{W}_{x} = \tb{JF}_{\tb{U}}(\tb{W})\tb{U}_{x} + \tb{JF}_{\tb{\tb{\sigma}}}(\tb{W})\tb{\sigma}_{x},
\end{equation}
where
\begin{equation}
    \tb{JF}_{\tb{U}}(\tb{W}) = \begin{pmatrix}
0 & 1 \\
\displaystyle -u^2 + \frac{K}{\rho} a \partial_a \phi(a) & 2u
\end{pmatrix}, \quad
\tb{JF}_{\tb{\tb{\sigma}}}(\tb{W}) = \begin{pmatrix}
0 & 0 & 0\\
\displaystyle \frac{A_{0}}{\rho} \red{ \tilde \Phi(a)} & \displaystyle \frac{K}{\rho}(\tilde \Phi(a) - a^2 \partial_a \phi(a)) & 0
\end{pmatrix}.
\end{equation}
Therefore we can write the system in quasi-linear form:
\begin{equation}\label{eq:systemextended}
\partial_t \tb{W} + \tb{\tb{\mathcal{A}}}(\tb{W})\partial_x \tb{W} =  \widetilde{\tb{R}}(\tb{W}), 
\end{equation}
with:
\begin{equation}\label{eq:matrix_A}  \tb{\tb{\mathcal{A}}}(\tb{W})    = \left(  \begin{array}{ccccc} 
    0 & 1 & 0 & 0 & 0 \\
   \displaystyle -u^2 + \frac{K}{\rho} a \partial_a \phi(a)  &  2u &   \displaystyle \frac{A}{\revb{\rho}} \phi (a)&  
   \displaystyle -a^2\frac{K}{\rho}\partial_a\phi(a) & \displaystyle \frac{A}{\rho} \\
    0 & 0 & 0 & 0 & 0 \\
     0 & 0 & 0 & 0 & 0 \\
      0 & 0 & 0 & 0 & 0 
      \end{array} \right),
      \end{equation}
      \begin{equation}\label{eq:R}
\widetilde{\tb{R}}(\tb{W}) =  \left( \begin{array}{c} 0 \\ \displaystyle  - \frac{\gamma \pi \mu}{\rho} u {\blue + g A}\\  0 \\ 0 \\ 0 \end{array}   \right) . 
\end{equation}
The eigenvalues of $ \tb{\tb{\mathcal{A}}}(\tb{W})$ are:
\begin{equation}\label{eq:eigenvalues}
\lambda_1 = u - c, \quad \lambda_2 = \lambda_3 = \lambda_4 =0, \quad \lambda_5 = u + c,
\end{equation}
with
$$
c = \sqrt{ \frac{K}{\rho}a \partial_a \phi(a) }.
$$
The flow regime is said to be subcritical or subsonic if $|u| < c$, supercritical or supersonic if $|u| > c$, and critical or sonic if $|u| = c$.

\subsection{Weak solutions}
When the functions $K(x)$, $A_0(x)$ and/or $p_e(x)$ have jump discontinuities the source term  $\tb{S}(\tb{W})\cdot\partial_x \tb{\sigma} $ becomes a nonconservative product whose mathematical definition is ambiguous. Therefore, the admissible jump conditions satisfied by weak solutions have to be selected.  Following \cite{pimentelgarcia2023high} here we assume that the pairs of states 
$(\tb{W}_l, \tb{W}_r)$ that can be linked by an admissible contact discontinuity standing on a discontinuity point of  $\tb{\sigma}$ have to satisfy:
 \begin{equation}\label{dcadm}
 q_l = q_r, \quad \frac{\rho}{2}u_l^2 + K_l \phi\left(\frac{A_l}{A_{0,l}} \right) + p_{e,l}= 
  \frac{\rho}{2}u_r^2 + K_l \phi\left(\frac{A_r}{A_{0,r}} \right) + p_{e,r}.
 \end{equation}
 These jump conditions can be interpreted in terms of the choice of a family of paths within the framework of the DalMasso-LeFloch-Murat theory: see \cite{pimentelgarcia2023high} for details.

 \subsection{Nondimensional form}\label{subsec:nondiemensional} 
 
Due to the need of dealing with very small values for some variables, it is advisable to write system \eqref{eq:model} in non-dimensional form. The following non-dimensional variables are chosen here:
 $$
 x' = \frac{x}{\overline{L}}, \quad t' = \frac{t}{\overline{T}}, \quad A' = \frac{A}{\overline{A}}, \quad q'= \frac{q}{\overline{A}\, \overline{U}}, \quad {\red{u' = \frac{u}{\overline{U}}, \quad K' = \frac{K}{\overline K}, \quad A_0' = \frac{A_0}{\overline A}, \quad p_e' = \frac{p_e}{\rho \overline{U}^2},}} {\green{\quad g' = \frac{g}{\overline g}}}
 $$
 where $\overline{L}$, $\overline{T}$, $\overline{A}$, {\red{$\overline{K}$}}, {\green{$\overline{g}$}} are the characteristic length, time,  cross-sectional area, 
{ \red{stiffness coefficient}}, { \green{gravitational coefficient}} and $\overline{U} = \overline{L}/\overline{T}$. Then system \eqref{eq:model} can be written as:
 \begin{equation}\label{ndform}
 \left\{
 \begin{array}{l}
 \displaystyle \partial_{t'} A' + \partial_{x'} q' = 0, \\
 \displaystyle \partial_{t'} q' + \partial_{x'}\left( \frac{(q')^2}{A'} \right) + 
 \red{ \overline{S}_h^{-2} A' \partial_{x'} \left( K'\phi\left(\frac{A'}{A_0'} \right) \right)+ A'\partial_{x'} p_e' = - \overline{\mu} u'  } {\green{+ \overline{\beta}g'A'}}
 \end{array} 
 \right.
 \end{equation}
 where
 $$
 \overline{S}_h = \frac{\overline{U}}{\sqrt{{\overline{K}}/{\rho}}}, \quad {\red{\overline{\mu} = \frac{\gamma \pi \mu \overline{T}}{\rho\overline{A}}}}, \quad {\green{\overline{\beta} = \frac{\overline{g}\overline{T}}{\overline{U}}}},
 $$
 are dimensionless numbers.

 \subsection{Stationary solutions}\label{subsec:stationary_solutions}
 The stationary solutions of system \eqref{eq:systemextended} satisfy
\begin{equation}\label{stationary_solutions}
   \partial_x \tb{F}(\tb{W}) + \tb{S}(\tb{W})\cdot\partial_x \tb{\sigma}  = \tb{R}(\tb{W}),
\end{equation}
together with the jump conditions \eqref{dcadm} at the discontinuity points of $\tb{\sigma}$.
\eqref{stationary_solutions} is an ODE system that, if $\tb{JF}_{\tb{U}}(\tb{W})$ is regular, i.e. if not critical states are reached, can be rewritten as follows:
\begin{equation}\label{stationary_solutions_extended}
    \tb{U}_{x} = -\tb{JF}_{\tb{U}}(\tb{W})^{-1}((\tb{S}(\tb{W})+\tb{JF}_{\tb{\sigma}}(\tb{W}))\cdot \tb{\sigma}_x - \tb{R}(\tb{W})),
\end{equation}
or, equivalently:
\begin{eqnarray} \label{stationary_solutions_extended_q}
& & \partial_x q = 0, \\ \label{stationary_solutions_extended_A}
& & \partial_x A = -\frac{A \phi(a) \partial_x K - 
K  a^2 \partial_a \phi(a) \partial_x A_0 + A \partial_x p_e + \gamma \pi \mu u {\blue  - \rho g A} }{- \rho u^2 + K a\partial_a\phi(a)  },
\end{eqnarray} 
i.e. $q$ is constant and $A(x)$ solves the ODE \eqref{stationary_solutions_extended_A}. 
If there is no friction {{\green and no gravity}}, i.e. $\mu = 0$ {{\green and $g=0$}}, the solutions of \eqref{stationary_solutions} satisfy
\begin{equation} \label{ssnofriction} q = constant, \quad  \Gamma(\tb{W}) =  constant, \end{equation}
with
\begin{equation}\label{Gamma}  \Gamma(\tb{W}) = \frac{\rho}{2}u^2 + K \phi\left(\frac{A}{A_{0}} \right) + p_{e}  \end{equation}
(see \cite{pimentelgarcia2023high}) but when friction is present, there is not an easy closed form either explicit or implicit for the solutions of this ODE system.

\section{High-order fully well-balanced numerical methods}\label{General methodology}

For simplicity, we consider here uniform meshes composed by cells $I_i = [x_{i-1/2}, x_{i+1/2}]$ of length $\Delta x$ whose midpoints are represented by $x_i$. The mesh is assumed to be designed so that the discontinuity points of function $K, A_{0}$ or $p_{e}$ are placed at the interface between two computational cells.

Following \cite{pares2006numerical}, we consider semi-discrete finite-volume methods of the form:
\be\label{eq:semi_discrete}
\frac{d\tb{W}_{i}}{dt}= -\frac{1}{\Delta x}
\Big(
{\mathbb{D}}^{-}_{i+\frac{1}{2}}+ {\mathbb{D}}^{+}_{i-\frac{1}{2}} + \int_{x_{i-\frac{1}{2}}}^{x_{i+\frac{1}{2}}} \tb{\mathcal{A}}(\mathbb{P}^t_{i}(x))\frac{\partial}{\partial x}\mathbb{P}^t_{i}(x)dx - \int_{x_{i-\frac{1}{2}}}^{x_{i+\frac{1}{2}}} \widetilde{\tb{R}}(\mathbb{P}^t_{i}(x))\,dx
\Big),
\ee
where:
\begin{itemize}
	\item $\tb{W}_{i}(t) \cong \displaystyle \frac{1}{\Delta x}\int_{x_{i+\frac{1}{2}}}^{x_{i-\frac{1}{2}}} \tb{W}(x,t) \,dx$ is the approximation to the cell average of the solution at the $i$-th cell \reva{at} time $t$;
	\item $\mathbb{P}^t_{i}(x) = \mathbb{P}_i(x; \{ \tb{W}_j(t) \}_{j \in  \mathcal{S}_i})= \left[\begin{array}{c}
	     \tb{P}_{i}^{t}(x)  \\
	     \tb{\sigma}(x) 
	\end{array}\right]$ is a high-order reconstruction operator, i.e. an operator that gives a smooth high-order approximation of a function at the $i$-th cell using the values of its averages at the cells belonging to the stencil $\mathcal{S}_i$;
	\item ${\mathbb{D}}_{i+\frac{1}{2}}^{\pm} = {\mathbb{D}}^{\pm}\left(\tb{W}_{i+\frac{1}{2}}^{-} , \tb{W}_{i+\frac{1}{2}}^{+}\right) $, where:
	$$
	\tb{W}_{i+\frac{1}{2}}^{-} (t)= \mathbb{P}^t_{i}(x_{i+\frac{1}{2}}),\ \ \tb{W}_{i+\frac{1}{2}}^{+}(t) = \mathbb{P}^t_{i+1}(x_{i+\frac{1}{2}})
	$$
	are the reconstructions at the cell interface
	and 
 $$ {\mathbb{D}}^\pm(\tb{W}_{l}, \tb{W}_{r}) = \left(\begin{array}{c} {\tb{D}}^\pm(\tb{W}_{l}, \tb{W}_{r}) \\ 0 \end{array} \right)$$ are the fluctuations corresponding to a first-order path-conservative numerical method
satisfying:
	\begin{equation}\label{eq:cons}
	{\mathbb{D}}^{\pm}(\tb{W}, \tb{W}) = 0, \quad \forall \tb{W}
	\end{equation}
	and:
	\begin{eqnarray}\label{eq:D}
	    {\mathbb{D}}^{-}(\tb{W}_{l}, \tb{W}_{r}) + {\mathbb{D}}^{+}(\tb{W}_{l}, \tb{W}_{r}) & = & \int_{0}^{1} \tb{\mathcal{A}}(\tb{\Psi}(s))\frac{\partial \tb{\Psi}}{\partial s}(s) \,ds
	    \\ \nonumber
	    & = & \left[\begin{array}{c}
\displaystyle \tb{F}(\tb{W}_r) - \tb{F}(\tb{W}_l) 
+ \int_0^1 \tb{S}(\tb{\Psi}(s)) \cdot \frac{\partial \tb{\Psi_\sigma}}{\partial s} (s) \,ds \\
\\
0
\end{array}
\right], \quad \forall \tb{W}_l, \tb{W}_r,
	\end{eqnarray}
	where 
 $$
\tb{\Psi}(s; \tb{W}_l, \tb{W}_r)  = \left[ \begin{array}{c} \tb{\Psi}_{\tb{U}}(s; \tb{W}_l, \tb{W}_r) \\
\tb{\Psi}_{\tb{\sigma}}(s; \tb{W}_l, \tb{W}_r) \end{array}\right] = \left[ \begin{array}{c} \Psi_A(s; \tb{W}_l, \tb{W}_r) \\
\Psi_q(s; \tb{W}_l, \tb{W}_r) \\
\Psi_K(s; \tb{W}_l, \tb{W}_r) \\
\Psi_{A_{0}}(s; \tb{W}_l, \tb{W}_r) \\
\Psi_{p_{e}}(s; \tb{W}_l, \tb{W}_r) 
\end{array}\right]
, \quad s \in [0,1]$$
is a family of paths joining $\tb{W}_{l}$ with $\tb{W}_{r}$. 
\end{itemize}



If the trivial equations for the variables $K, A_{0}, p_{e}$ variables are dropped, the numerical method \eqref{eq:semi_discrete} can be written as follows:
\begin{eqnarray}\label{eq:sdmethod}
\frac{d\tb{U}_{i}}{dt}& = & -\frac{1}{\Delta x}
\Big(
{\tb{D}}^{-}_{i+\frac{1}{2}}+ {\tb{D}}^{+}_{i-\frac{1}{2}} + \tb{F}(\tb{W}_{i+1/2}^-) - \tb{F}(\tb{W}_{i-1/2}^+) \\
\nonumber &  & \qquad + \int_{x_{i-\frac{1}{2}}}^{x_{i+\frac{1}{2}}} \tb{S}({\mathbb{P}}_{i}(x)) \cdot \tb{\sigma}'(x)dx - \int_{x_{i-\frac{1}{2}}}^{x_{i+\frac{1}{2}}} \tb{R}({\mathbb{P}}_{i}(x)) dx
\Big).
\end{eqnarray}
In practice,  a quadrature formula 
$$
\int_{x_{i-1/2}}^{x_{i+1/2}} f(x) \,dx \approx  \Delta x \sum_{m=1}^{M} \beta_{m} f(x_{i}^{m}), $$
where $\beta_{m}$ and $x_{i}^{m}$, $m \in \{1,...,M\}$ are the weights and nodes respectively, is used to compute cell-averages and the integrals in \eqref{eq:sdmethod}.


\subsection{Well-balanced methods}\label{subsec:WB_methods}



In order to clearly state the well-balanced properties of the methods to be described, let us introduce some definitions:
\begin{definition}
    A sequence of cell values $\{ \tb{U}_i^* \}$ is said to be a discrete stationary solution if it is an equilibrium of the ODE 
    system \eqref{eq:sdmethod}, i.e. if for all $i$
    \begin{equation}\label{eq:discrete_stationary}
       {\tb{D}}^{-}_{i+\frac{1}{2}}+ {\tb{D}}^{+}_{i-\frac{1}{2}} + \tb{F}(\tb{W}_{i+1/2}^-) - \tb{F}(\tb{W}_{i-1/2}^+) 
 +  \Delta x \sum_{m=1}^{M} \beta_{m} \left[\tb{S}({\mathbb{P}}^t_{i}(x_{i}^{m})) \cdot \tb{\sigma}'(x_{i}^{m}) - \tb{R}({\mathbb{P}}^t_{i}(x_{i}^{m}))\right] = 0.
    \end{equation}
\end{definition}

\begin{definition}
    Given a stationary solution $\tb{U}^*$ of system \eqref{eq:model}, the numerical method \eqref{eq:sdmethod} is said to be exactly well-balanced (EWB) for $\tb{U}^*$ if the sequence $\{ \overline{\tb{U}}^*_i \}$ of its cell averages
    \begin{equation}\label{cell_av}
    \overline{\tb{U}}^*_i = \sum_{m=1}^{M} \beta_{m} \tb{U}^*(x_{i}^{m})
    \end{equation}
    is a discrete stationary solution. 
\end{definition}

\begin{definition}
    The numerical method \eqref{eq:sdmethod} is said to be fully exactly well-balanced (FEWB) if it is EWB for every stationary solution $\tb{U}^*$.
\end{definition}

In the absence of friction, the fact that the stationary solutions are implicitly given by \eqref{ssnofriction} allowed us to design FEWB methods in  
\cite{pimentelgarcia2023high}. Nevertheless, when friction is present, the expression of the stationary solutions is not known either in explicit or implicit form. Therefore, instead of EWB we will design here methods that are well-balanced according to the following definition:

\begin{definition}
    Given a stationary solution $\tb{U}^*$ of system \eqref{eq:model}, the numerical method \eqref{eq:sdmethod} is said to be  well-balanced (WB) for $\tb{U}^*$ if, for every $\Delta x$, it is possible to find a discrete stationary solution $\{ \widetilde{\tb{U}}^*_i \}$
    such that
    $$ 
     \overline{\tb{U}}^*_i  = \widetilde{\tb{U}}^*_i + O(\Delta x^q)
    $$
    where  $\overline{\tb{U}}^*_i$  is given by \eqref{cell_av} and $q$ is greater or equal than the order of the method.
\end{definition}

\begin{definition}
    The numerical method \eqref{eq:sdmethod} is said to be fully well-balanced (FWB) if it is WB for every stationary solution $\tb{U}^*$.
\end{definition}

Following \cite{gomez2021collocation}, in order to design FWB schemes, a numerical solver with order of accuracy $q$ greater or equal to the one of the numerical method will be used to solve the following problem:

\begin{itemize}

\item[(P)] \textit{Given an index $i_0$ and an arbitrary state $\tb{U}_{i_0}$, find numerical approximations at the interfaces and the quadrature points 
$$\tb{U}^{*, \pm}_{i + 1/2} \approx \tb{U}^{*}(x_{i+1/2}^{\pm}), \quad \tb{U}^{*,m}_{i} \approx \tb{U}^{*}(x_{i}^{m}), \ \ m \in \{1,...,M\}, \ \ \forall i,$$
of the stationary solution $\tb{U}^*$ of \eqref{eq:model} that satisfies
\begin{equation}\label{localproblem}
\sum_{m=1}^M \beta_m \tb{U}^*(x_{i_0}^m) = {\tb U}_{i_0},
\end{equation}
if it exists. }

\end{itemize}

Once the numerical solver has been applied, the cell-averages of the stationary solution satisfying \eqref{localproblem} will be approximated by
   \begin{equation}\label{num_aver}
    \widetilde{\tb U}^*_i = \sum_{m = 1}^M \beta_m \tb{U}^{*,m}_{i}. 
    \end{equation}
    The following notation will be used
$$\tb{W}^{*, \pm}_{i + 1/2} = \left( \begin{array}{c}\tb{U}^{*, \pm}_{i+1/2} \\ \tb{\sigma}(x_{i+1/2}^\pm) \end{array}\right),  \quad \tb{W}^{*,m}_{i} = \left( \begin{array}{c}\tb{U}^{*,m}_{i}  \\ \tb{\sigma}(x_i^m) \end{array}\right) \  m \in \{1,...,M\}, \quad \widetilde{\tb{W}}^{*}_{i} = \left( \begin{array}{c}\widetilde{\tb{U}}^{*}_{i} \\ \overline{\tb{\sigma}}_i \end{array}\right),
$$
where
$$
\overline{\tb{\sigma}}_i  = \sum_{m=1}^M \beta_m \tb{\sigma}(x_{i_0}^m).
$$

Two requirements are asked to the numerical solver of problem (P):  
\begin{itemize}
    \item[(R1)] Once  $\tb{U}^{*, \pm}_{i+1/2}$, $\tb{U}^{*,m}_{i}$, $ m =1,...,M$, $\widetilde{\tb U}^*_i$ have been computed,
    if the numerical solver is applied again to find approximations of the stationary solution satisfying
    $$
\sum_{m=1}^M \beta_m \tb{U}^*(x_{j_0}^m) =   \widetilde {\tb U}^*_{j_0} ,
$$
where $j_0$ is an arbitrary index, then the provided approximations are  again $\tb{U}^{*, \pm}_{i+1/2}$, $\tb{U}^{*,m}_{i}$, $ m =1,...,M$, $\widetilde{\tb U}^*_i$ for all $i$.

\item[(R2)] If $\tb{\sigma}$ is continuous at $x_{i + 1/2}$ then $\tb{W}^{*, +}_{i+ 1/2}  = \tb{W}^{*, -}_{i + 1/2} $.
Otherwise the pair of states  $ (\tb{W}^{*, -}_{i+ 1/2}  , \tb{W}^{*, +}_{i + 1/2})$ 
satisfy \eqref{dcadm}, i.e. the states can be linked by an admissible contact discontinuity.

\end{itemize}

A numerical solver of Problem (P) satisfying (R1), (R2) is described below in Section \ref{ss:numsol}. Once equipped with such a numerical solver, the design of FWB numerical methods will be based on the choice of path-conservative fluctuations and reconstruction operators that are well-balanced in the following sense: 
 \begin{definition} \label{wbpc} The path-conservative fluctuations $\tb{D}^\pm$ are said to be well-balanced for system \eqref{eq:model} if
$$
\tb{D}^\pm(\tb{W}_l, \tb{W}_r) = 0
$$
for $\tb{W}_l$ and $\tb{W}_r$ such that 
\begin{equation}\label{wb1o}
q_l = q_r, \quad \revc{\Gamma(\tb{W}_l) = \Gamma(\tb{W}_r)}.
\end{equation}
\end{definition}

\begin{definition} \label{WBrecop}
Given a numerical solver for problem (P), the reconstruction operator $\{ \tb{P}_{i}(x) \}$ is said to be well-balanced if, when applied to the sequence of cell values $\{ \widetilde{\tb U}^*_i \}$ computed by 
\eqref{num_aver} from the approximations provided by the numerical solver, one has
$$
\tb{P}_{i}(x_{i-1/2}) = \tb{U}^{*,+}_{i-1/2}, \quad \tb{P}_{i}(x_{i+1/2}) = \tb{U}^{*,-}_{i+1/2}, \quad \tb{P}_{i}(x_{i}^m) = \tb{U}^{*,m}_{i}, \ \ m=1,...,M, \quad \forall i. 
$$
 \end{definition}

 Concerning the WB path-conservative fluctuations, the Generalized Hydrostatic Reconstruction (GHDR) technique is considered here (see \cite{pimentelgarcia2023high} for details), according to which:
 \begin{eqnarray}\label{fwbD+}
{\mathbf{D}}^+(\tb{W}_l, \tb{W}_r) & =  &   \tb{F}(\tb{W}_0^+) - \mathbb{F}(\tb{W}_0^-, \tb{W}_0^+),\\\label{fwbD-}
{\mathbf{D}}^-(\tb{W}_l, \tb{W}_r) & = & \mathbb{F}(\tb{W}_0^-, \tb{W}_0^+) -   \tb{F}(\tb{W}_0^-),
\end{eqnarray}
where
$$
\tb{\sigma}_0 = \left (\begin{array}{c} K_0 \\ A_{0,0} \\  p_{e,0} \end{array} \right)
$$ 
is a vector of intermediate value between $\tb{\sigma}_l$ and $\tb{\sigma}_r$ such that $\tb{\sigma}_0 = \tb{\sigma}_l = \tb{\sigma}_r$ whenever $\tb{\sigma}_l= \tb{\sigma}_r$;  $\mathbb{F}(\cdot, \cdot)$ is a consistent numerical flux, i.e. a continuous function such that
$$
\mathbb{F}(\tb{W}, \tb{W}) = \tb{F}(\tb{W}), \quad \forall \tb{W};
$$
and $\tb{W}^{-}_0 = [A^-, q_l, K_0, A_{0,0}, p_{e,0}]^T$, $\tb{W}^{+}_0 = [A^+, q_r, K_0, A_{0,0}, p_{e,0}]^T$, where $A^\pm$ are computed by solving the equations:
    \begin{eqnarray}
    \frac{\rho}{2} \frac{q_l^2}{(A^{-})^2} +  K_0 \phi\left(\frac{A^{-}}{A_{0,0}} \right) + p_{e,0}  & = & \revc{\Gamma(\tb{W}_l)}, \label{w0-}\\
    \frac{\rho}{2} \frac{q_r^2}{(A^{+})^2} +  K_0 \phi\left(\frac{A^{+}}{A_{0,0}} \right) + p_{e,0}  & = &  \revc{\Gamma(\tb{W}_r)}.   \label{w0+}
\end{eqnarray}

In particular, in this article the HLL numerical flux is considered (see \cite{harten1983upstream}):
\begin{equation}
\mathbb{F}(\tb{W}_l, \tb{W}_r) = \begin{cases}
\tb{F}(\tb{W}_l) & \text{if $S_l \geq 0$,}\\
\displaystyle \frac{S_r \tb{F}(\tb{W}_l)  - S_l \tb{F}(\tb{W}_r)}{S_r - S_l}
+ \frac{S_l S_r}{S_r- S_l}(U_r - U_L), &\text{if $S_l < 0 < S_r$,}\\
\tb{F}(\tb{W}_r) & \text{if $S_r \leq 0$,}
\end{cases} \label{eq:hll}
\end{equation}
where 
$$
S_{l} = \min\{\lambda_{1}(\tb{W}_{l}), \lambda_{5}(\tb{W}_{r})\}, \quad S_{r} = \max\{\lambda_{1}(\tb{W}_{l}), \lambda_{5}(\tb{W}_{r})\}.
$$
The intermediate values $\tb{\sigma}_0 $ are selected so that equations \eqref{w0-}, \eqref{w0+} have always solution: see \cite{pimentelgarcia2023high}.

Concerning the reconstruction operator, a WB one $\{ \tb{P}_{i}(x) \}$ can be obtained from a standard one $\{ \tb{Q}_{i}(x) \}$ using the following reconstruction procedure: given a sequence of cell values $\{ \tb{U}_i \}$ the reconstruction at the $i$-th cell is computed as follows

\subsubsection*{Reconstruction procedure}

\begin{enumerate}

\item Apply the numerical solver to obtain, if possible,  approximations of the stationary solution $\tb{U}^*_i$ satisfying
\begin{equation}\label{firststage}
\sum_{m=1}^M \beta_m \tb{U}_i^*(x_{i}^m) = {\tb U}_{i}
\end{equation}
at the intercells and quadrature points of the cells belonging to stencil:
 $$\tb{U}^{*, \pm}_{i,j \pm 1/2}, \quad  \tb{U}^{*,m}_{i,j}, \quad m =1,...,M, \ j \in S_i.$$
 Compute 
    $$
    \widetilde{\tb U}^*_{i,j} = \sum_{m = 1}^M \beta_m \tb{U}^{*,m}_{i,j}, \quad j \in S_i. 
    $$

    \item Apply the standard reconstruction operator to the fluctuations
    $$
    \tb{V}_j =  \tb{U}_j  -  \widetilde{\tb U}^*_{i,j}, \quad j \in S_i,
    $$
    to obtain $\tb{Q}_i$.

    \item Define
 $$
\tb{P}_{i}(x_{i-1/2}) = \tb{U}^{*,+}_{i,i-1/2} + \tb{Q}_i(x_{i-1/2}), 
\quad \tb{P}_{i}(x_{i+1/2}) = \tb{U}^{*,-}_{i,i+1/2} + \tb{Q}_i(x_{i + 1/2}),$$
$$\tb{P}_{i}(x_{i}^m) = \widetilde{\tb{U}}^{*,m}_{i,i} + \tb{Q}_i(x_i^m), \quad  m=1,...,M. 
$$
    \end{enumerate}
Observe that this reconstruction procedure only defines the values of $\tb{P}_{i}$ at the intercell and quadrature points but, as it will be seen, this is enough to implement the WB numerical method.  It can be easily checked that, due to requirement (R1), this reconstruction operator is WB in the sense of Definition \ref{WBrecop}.

Once the fluctuations and reconstruction operator have been chosen, one more ingredient is necessary: a well-balanced quadrature formula to approximate the integrals in \eqref{eq:sdmethod}. It is based on the following equality satisfied by the stationary solutions $\tb{W}^* = (\tb{U}^*, \tb{\sigma})^T$ of \eqref{eq:model}:
$$
\tb{F}(\tb{W}^*(x_{i+1/2}^-)) - \tb{F}(\tb{W}^*(x_{i-1/2}^+)) = -\int_{x_{i-\frac{1}{2}}}^{x_{i+\frac{1}{2}}} \tb{S}(\tb{W}^*(x)) \cdot \tb{\sigma}'(x)\, dx + \int_{x_{i-\frac{1}{2}}}^{x_{i+\frac{1}{2}}} \tb{R}(\tb{W}^*(x)) \, dx .
$$
Therefore, the approximations of a stationary solution given by the numerical solver of problem (P) satisfy
$$
\tb{F}(\tb{W}^{*,-}_{i+1/2}) - \tb{F}(\tb{W}^{*,+}_{i-1/2}) \approx 
-\Delta x \sum_{m=1}^M \beta_m \tb{S}(\tb{W}^{*,m}_{i,i})\cdot \tb{\sigma}'(x_i^m) +
\Delta x \sum_{m=1}^M \beta_m \tb{R}(\tb{W}^{*,m}_{i,i} ),
$$
with an order of accuracy which is the minimum between that of the numerical solver and that of the quadrature formula. 
The following approximation of the integrals is then used:
\begin{eqnarray*}
 \int_{x_{i-\frac{1}{2}}}^{x_{i+\frac{1}{2}}} \tb{S}({\mathbb{P}}_{i}(x)) \cdot \tb{\sigma}'(x)dx - \int_{x_{i-\frac{1}{2}}}^{x_{i+\frac{1}{2}}} \tb{R}({\mathbb{P}}_{i}(x))\, dx  & \approx & 
- \tb{F}(\tb{W}^{*,-}_{i,i+1/2}) + \tb{F}(\tb{W}^{*,+}_{i,i-1/2}) \\
 &+&
\Delta x \sum_{m=1}^M \beta_m \left( \tb{S}({\mathbb{P}}_{i}(x_i^m)) -  \tb{S}(\tb{W}^{*,m}_{i,i}) \right)\cdot \tb{\sigma}'(x_i^m) \\
&-&
\Delta x \sum_{m=1}^M \beta_m \left(  \tb{R}({\mathbb{P}}_{i}(x_i^m)) - \tb{R}(\tb{W}^{*,m}_{i,i} ) \right),
\end{eqnarray*}
where the approximations of the stationary solution $\tb{U}^*_i$ satisfying \eqref{firststage} found at the first stage of the reconstruction operator are used. 
The expression of the numerical method is then as follows:
\begin{eqnarray} \label{eq:sdmethod2}
 \frac{d\tb{U}_{i}}{dt} &=&  -\frac{1}{\Delta x}
\Big(
{\tb{D}}^{-}_{i+\frac{1}{2}}+ {\tb{D}}^{+}_{i-\frac{1}{2}} + \tb{F}(\tb{W}_{i+1/2}^-) - \tb{F}(\tb{W}_{i-1/2}^+) - \tb{F}(\tb{W}^{*,-}_{i,i+1/2}) + \tb{F}(\tb{W}^{*,+}_{i,i-1/2}) \Big) \\ \nonumber
 &-& \qquad
\sum_{m=1}^M \beta_m \left( \tb{S}({\mathbb{P}}_{i}(x_i^m)) -  \tb{S}(\tb{W}^{*,m}_{i,i}) \right)\cdot \tb{\sigma}'(x_i^m) 
+
\sum_{m=1}^M \beta_m \left(  \tb{R}({\mathbb{P}}_{i}(x_i^m)) - \tb{R}(\tb{W}^{*,m}_{i,i} ) \right).
\end{eqnarray}
This numerical method is FWB in the following sense:
\begin{theorem}
Let us suppose that the path-conservative fluctuations and the reconstruction operator are well-balanced. Then, the numerical method \eqref{eq:sdmethod2} is well-balanced for every stationary solution of \eqref{eq:model} that is continuous everywhere but at the  discontinuity points of $\tb{\sigma}$ where it has admissible jumps satisfying \eqref{dcadm}.
\end{theorem}
\begin{proof}
Given a stationary solution $\tb{U}^*$ apply the numerical solver to problem (P) with $i_0$ an arbitrary index and 
$$ \tb{U}_{i_0} = \sum_{m=1}^M \beta_m \tb{U}^*(x_i^m)$$
to obtain approximations 
$$\tb{U}^{*, \pm}_{i+1/2} \approx \tb{U}^{*}(x_{i+1/2}^{\pm}), \quad \tb{U}^{*,m}_{i} \approx \tb{U}^{*}(x_{i}^{m}), \ \ m \in \{1,...,M\}, \ \ \forall i.$$
Then,  $\widetilde{\tb U}^*_i$ given by \eqref{num_aver} approximates the cell-averages of $\tb{U}^*$, computed with the quadrature formula, with order $q$. Let us check that
the sequence $\{\widetilde{\tb U}^*_i\} $ is a discrete stationary solution. First, due to requirement (R2) one has
$$
\mathbb{D}^\pm(\tb{W}^{*,-}_{i+1/2}, \tb{W}^{*,+}_{i+1/2}) = 0
$$
at every intercell point: this is due to \eqref{eq:cons} if $\sigma$ is continuous at $x_{i + 1/2}$ and to the WB character of the path-conservative fluctuations otherwise. 
Next, due to the well-balanced property of the  reconstruction operator one has
$$
 {\mathbb{P}}_{i}(x_i^m) = \tb{W}^{*,m}_{i,i}, \quad m=1, \dots, M, \ \forall i,
$$
and then the last two terms in \eqref{eq:sdmethod2} vanish. Finally, since
$$
\mathbb{P}_{i}(x_{i-1/2}) = \tb{W}^{*,+}_{i-1/2}, \quad \mathbb{P}_{i}(x_{i+1/2}) = \tb{W}^{*,-}_{i+1/2},
$$
one has
$$
\tb{F}(\tb{W}_{i+1/2}^-) = \tb{F}(\tb{W}^{*,-}_{i,i+1/2}), \quad  \tb{F}(\tb{W}_{i-1/2}^+) = \tb{F}(\tb{W}^{*,+}_{i,i-1/2}).
$$
Therefore, \eqref{eq:discrete_stationary} holds as we wanted to proof.
    
\end{proof}

\subsection{Numerical solver for problem (P)}\label{ss:numsol}
According to the discussion above, a numerical solver satisfying requirements (R1) and (R2) is needed to approximate the value at some points of the solution
$\tb{U}^* = (A^*, q^*)^T$ of the problem
\begin{equation}\label{problemP}
\left\{
\begin{array}{l}
   \partial_x \tb{F}(\tb{W}^*) + \tb{S}(\tb{W}^*)\cdot\partial_x \tb{\sigma}  = \tb{R}(\tb{W}^*), \\
 \displaystyle \sum_{m=1}^M \beta_m \tb{U}^*(x_{i_0}^m) = {\tb U}_{i_0},
 \end{array}
 \right.
\end{equation}
where $\tb{W}^* = (\tb{U}^*, \tb{\sigma})^T$. In this paragraph we only discuss the case in which no critical state is reached. In that case, the ODE system is equivalent to 
\eqref{stationary_solutions_extended_q}-\eqref{stationary_solutions_extended_A}, so that
$$
q^*(x) = q_{i_0}, \quad \forall x$$
and the problem reduces to find $A^*$ satisfying
\begin{equation}\label{problemP_A}
\left\{
\begin{array}{l}
   \partial_x A^* = G(x,A^*), \\
\displaystyle \sum_{m=1}^M \beta_m A^*(x_{i_0}^m) = A_{i_0},
 \end{array}
 \right.
\end{equation}
with 
\begin{equation}\label{G}
G(x, A) =  -\frac{A \phi\left(\frac{A}{A_0(x)}\right) \partial_x K(x) - 
K(x) \left(\frac{A}{A_0(x)}\right)^2 \partial_a \phi\left(\frac{A}{A_0(x)}\right) \partial_x A_0(x) + A \partial_x p_e(x) + \gamma \pi \mu \frac{q_{i_0}}{A} {\blue - \rho g(x) A}}
{- \rho \left(\frac{q_{i_0}}{A}\right)^2 + K(x) \frac{A}{A_0(x)} \partial_a\phi\left(\frac{A}{A_0(x)}\right)  }.
\end{equation}
To solve this problem, a  Gauss-Legendre  RK collocation method with Butcher tableau 
\begin{equation}\label{RK4_collocation_gen}
\begin{array}{c|ccc}
    c_{1} & a_{1,1} & \dots & a_{1,M}\\
    c_{2} & a_{2,1} & \dots & a_{2,M}\\
    \vdots & \vdots & \ddots & \vdots\\
    c_M & a_{M,1} & \dots & a_{M,M}\\
     \hline
      & \beta_{1} & \dots & \beta_{M}
\end{array}
\end{equation}
is selected in which $\beta_m$, $x_i^m = x_{i - 1/2} + c_m \Delta x$, $m =1, \dots, M$ are the weights and points of the $s$-point Gauss quadrature formula, that is the one that will be used in the design of the numerical method. The solution of \eqref{problemP_A} is then approximated at the intercell and quadrature points with the following algorithm:

\begin{enumerate}

\item For $i = i_0$:

\begin{itemize}
    \item Find $A^{*,+}_{i_0-1/2}$, $A^{*,m}_{i_0}$, $m = 1, \dots, M$ by solving the system:
\begin{equation}\label{RKaverage1}
\left\{
\begin{array}{l}
\displaystyle  A^{*,m}_{i_0} = A^{*,+}_{i_0-1/2} + \Delta x \sum_{j = 1}^M a_{m,j} G(x_{i_0}^m,  A^{*,j}_{i_0}), \quad m= 1, \dots, M, \\
\displaystyle \sum_{m=1}^M \beta_m A_{i_0}^{*,m} = A_{i_0}.
\end{array}
 \right.
\end{equation}

\item Compute
\begin{equation}\label{RKaverage2}
A^{*,-}_{i_0+1/2} = A^{*,+}_{i_0-1/2} + \Delta x \sum_{m = 1}^M \beta_m G(x_{i_0}^m,  A^{*,j}_{i_0}),
\end{equation}
    
\end{itemize}

\item For $i = i_0 +1, i_0 +2, \dots$
\begin{itemize}
\item If $\tb{\sigma}$ is continuous at $x_{i-1/2}$ then define
$$
A_{i-1/2}^{*,+} = A_{i-1/2}^{*,-}.
$$
Otherwise, find $A_{i-1/2}^{*,+}$ by solving
$$
\frac{\rho}{2} \frac{q_{i_0}^2}{(A_{i-1/2}^{*,+})^2} +  K(x_{i-1/2}^+) \phi\left(\frac{A_{i-1/2}^{*,+}}{A_{0}(x_{i-1/2}^+)} \right) + p_{e}(x_{i-1/2}^+) \\
=  \Gamma^-_{i - 1/2},
$$
with
$$
 \Gamma^-_{i - 1/2} =  \frac{\rho}{2} \frac{q_{i_0}^2}{(A_{i-1/2}^{*,-})^2} +  K(x_{i-1/2}^-) \phi\left(\frac{A_{i-1/2}^{*,-}}{A_{0}(x_{i-1/2}^-)} \right) + p_{e}(x_{i-1/2}^-) .
$$

\item Find $A^{*,m}_{i}$, $m = 1, \dots, M$ by solving the system:
\begin{equation}\label{RKforward1}
 A^{*,m}_{i} = A^{*,+}_{i-1/2} + \Delta x \sum_{j = 1}^M a_{m,j} G(x_{i}^m,  A^{*,j}_{i}), \quad m= 1, \dots, M.
\end{equation}

\item Compute
\begin{equation} \label{RKforward2}
A^{*,-}_{i+1/2} = A^{*,+}_{i-1/2} + \Delta x \sum_{m = 1}^M \beta_m G(x_{i}^m,  A^{*,j}_{i}).
\end{equation}

\end{itemize}

\item For $i = i_0 -1, i_0 -2, \dots$
\begin{itemize}
\item If $\tb{\sigma}$ is continuous at $x_{i+1/2}$ then define
$$
A_{i+1/2}^{*,-} = A_{i+1/2}^{*,+}.
$$
Otherwise, find $A_{i+1/2}^{*,-}$ by solving
$$
\frac{\rho}{2} \frac{q_{i_0}^2}{(A_{i+1/2}^{*,-})^2} +  K(x_{i+1/2}^-) \phi\left(\frac{A_{i+1/2}^{*,-}}{A_{0}(x_{i+1/2}^-)} \right) + p_{e}(x_{i+1/2}^-) \\
=  \Gamma^+_{i + 1/2},
$$
with
$$
 \Gamma^+_{i + 1/2} =  \frac{\rho}{2} \frac{q_{i_0}^2}{(A_{i+1/2}^{*,+})^2} +  K(x_{i+1/2}^+) \phi\left(\frac{A_{i+1/2}^{*,+}}{A_{0}(x_{i+1/2}^+)} \right) + p_{e}(x_{i+1/2}^+) .
$$

\item Find $A^{*,m}_{i}$, $m = 1, \dots, M$ by solving the system:
\begin{equation}\label{RKbackward1}
 A^{*,m}_{i} = A^{*,-}_{i+1/2} - \Delta x \sum_{j = 1}^M a_{M-m+1,M-j+1} G(x_{i}^m,  A^{*,j}_{i}), \quad m= 1, \dots, M.
\end{equation}

\item Compute
\begin{equation}\label{RKbackward2}
A^{*,+}_{i-1/2} = A^{*,-}_{i+1/2} - \Delta x \sum_{m = 1}^M \beta_m G(x_{i}^m,  A^{*,j}_{i}).
\end{equation}

\end{itemize}

\end{enumerate}

\begin{remark}
Please note that, while \eqref{RKforward1}-\eqref{RKforward2} and \eqref{RKbackward1}-\eqref{RKbackward2} are standard forward and backward applications of the RK collocation methods to solve \eqref{problemP_A} in the cells of the stencil, this is not the case for \eqref{RKaverage1}-\eqref{RKaverage2}, where the method is adapted  to take into account that a condition on the average is given at the $i$th cell. The expression \eqref{RKbackward1}-\eqref{RKbackward2} of the backward application of the method is based on the symmetry of the Gauss quadrature points with respect to the midpoint of the interval.
\end{remark}

Requirement (R2) is satisfied because of the computations of the left and right limits of $A$ at the intercells. Requirement (R1) is satisfied because of the following reversibility or symmetry property satisfied by RK collocation methods: if, starting from $ A^{*,+}_{i-1/2}$ at $x_{i-1/2}$, one step of the forward RK collocation method 
\eqref{RKforward1}-\eqref{RKbackward2} gives 
$A^{*,-}_{i + 1/2}$ at $x_{i + 1/2}$, then one step of the backward method \eqref{RKbackward1}-\eqref{RKbackward2} starting from $A^{*,-}_{i + 1/2}$ at $x_{i + 1/2}$ gives $ A^{*,+}_{i-1/2}$ at $x_{i-1/2}$. No explicit method has this property (see \cite{gomez2021high}).

\subsection{First-order fully well-balanced path-conservative method}\label{subsec:first-order}
The mid-point rule 
$$
\int_{x_{i-1/2}}^{x_{i+1/2}} f(x) \, dx \approx \Delta x f(x_i)
$$
will be considered and the second-order 1-stage RK collocation method whose Butcher tableau is 
\begin{equation}\label{RK2_collocation}
    \begin{tabular}{c|c}
    1/2 & 1/2 \\
    \hline
     & 1
\end{tabular}
\end{equation}
will be used for solving \eqref{problemP_A} that, in this case, reduces to 
\begin{equation}\label{problemP_A_fso}
\left\{
\begin{array}{l}
   \partial_x A^* = G(x,A^*), \\
\displaystyle A^*(x_{i_0}) = A_{i_0}.
 \end{array}
 \right.
\end{equation}
The trivial piecewise constant reconstruction operator is considered, i.e.
$$
\tb{Q}_i(x) =  \tb{U}_i, \quad \forall x \in [x_{i-1/2}, x_{i+1/2}], \quad \forall i.$$
Numerical method \eqref{eq:sdmethod2} reduces in this case to:
\begin{equation}
\label{eq:semi_discrete_WB1}
\frac{d\tb{U}_{i}}{dt} = -\frac{1}{\Delta x}
\Big(
\tb{D}^{-}_{i+\frac{1}{2}}+\tb{D}^{+}_{i-\frac{1}{2}}
\Big),
\end{equation}
where $\tb{D}^{\pm}_{i+\frac{1}{2}} = \tb{D}^{\pm}\left(\tb{W}^{*,-}_{i-1,i+1/2}, \tb{W}^{*,-}_{i,i+1/2} \right)$. 
The explicit Euler method is used for the time discretization, so that the fully discrete method is then as follows:
\begin{equation}
\label{eq:fully_discrete_WB1}
\tb{U}_{i}^{n+1} = \tb{U}_{i}^{n}-\frac{\Delta t}{\Delta x}
\Big(
\tb{D}^{-}_{i+\frac{1}{2}}+\tb{D}^{+}_{i-\frac{1}{2}}
\Big).
\end{equation}

\subsection{Second-order fully well-balanced method}\label{subsec:second-order}

The mid-point rule and the second-order 1-stage collocation RK method whose Butcher tableau is \eqref{RK2_collocation} are again considered.  The second-order MUSCL reconstruction operator with $minmod$ limiter is chosen  (see \cite{van1974towards}): given a sequence of cell-values $\{ \tb{V}_i \}$, the $k$-component of $\tb{Q}_i$ is given by
    \begin{equation*}
    \tb{Q}_{i,k}(x)  =   V_{i,k}+minmod\left(\displaystyle \frac{V_{i,k}-V_{i-1,k}}{\Delta x}, \frac{V_{i+1,k}-V_{i,k}}{\Delta x}\right)(x-x_{i}),
    \end{equation*}
    where
    $$
    minmod(a,b) = \begin{cases}
    \min\{a,b\} & \text{if} \ \ a,b>0, \\
	\max\{a,b\} & \text{if} \ \ a,b <0, \\
    0 & \text{otherwise.}
    \end{cases}
    $$
Since, in this case, one has
$$ \tb{P}_i (x_i) = \tb{U}_i^{*,1}, $$
the numerical method \eqref{eq:sdmethod2} reduces in this case to
\begin{equation}
\label{eq:semi_discrete_WB2}
\frac{d\tb{U}_{i}}{dt}  = -\frac{1}{\Delta x}
\Big(
\tb{D}^{-}_{i+\frac{1}{2}}+\tb{D}^{+}_{i-\frac{1}{2}} + \tb{F}\left(\mathbb{P}_{i}(x_{i+ \frac{1}{2}})\right)
-\tb{F}\left(\tb{W}_{i,i+ \frac{1}{2}}^{*,-}\right)+\tb{F}\left(\tb{W}_{i,i- \frac{1}{2}}^{*,+}\right)-\tb{F}\left(\mathbb{P}_{i}(x_{i- \frac{1}{2}})\right)
\Big),
\end{equation}
with  $\tb{D}^{\pm}_{i+\frac{1}{2}} = \tb{D}^{\pm}\left(\mathbb{P}_{i}(x_{i+\frac{1}{2}}), \mathbb{P}_{i+1}(x_{i+\frac{1}{2}})\right).$ Finally, the discretization in time is performed with the following second-order TVD RK method:
\begin{eqnarray*}
 & \tb{W}_{i}^{(1)} = \tb{W}_{i}^{n} + \Delta t L(\tb{W}_{i}^{n}),\\
 & \tb{W}_{i}^{n+1} = \frac{1}{2}\tb{W}_{i}^{n} + \frac{1}{2}\tb{W}_{i}^{(1)} + \frac{1}{2}L(\tb{W}_{i}^{(1)}),
\end{eqnarray*}
where $L(\tb{W}_{i}) $ represents the right-hand side  of \eqref{eq:semi_discrete_WB2}: see \cite{gottlieb1998total}.

\subsection{Third-order fully well-balanced method}\label{subsec:third-order}

The two-point Gauss quadrature 
$$
\int_{x_{i-1/2}}^{x_{i+1/2}} f(x) \, dx \approx \frac{\Delta x}{2} (f(x_{i,1}) + f(x_{i,2})),
$$
where
$$
x_{i,1} = x_{i-\frac{1}{2}} + \frac{\Delta x}{2}\left(-\sqrt{\frac{1}{3}}+1\right),\ \
x_{i,2} = x_{i-\frac{1}{2}} + \frac{\Delta x}{2}\left(\sqrt{\frac{1}{3}}+1\right),
$$
is now used
and the fourth-order 2-stage Gauss-Legendre collocation RK method whose Butcher tableau is 
\begin{equation}\label{RK4_collocation}
    \begin{tabular}{c|cc}
    $\frac{1}{2}-\frac{\sqrt{3}}{6}$ & $\frac{1}{4}$ & $\frac{1}{4}-\frac{\sqrt{3}}{6}$ \\
    $\frac{1}{2}+\frac{\sqrt{3}}{6}$
     & $\frac{1}{4}+\frac{\sqrt{3}}{6}$ & $\frac{1}{4}$ \\
     \hline
      & $\frac{1}{2}$ & $\frac{1}{2}$
\end{tabular}.
\end{equation}
The third-order CWENO reconstruction operator is used as the standard one  (see \cite{cravero2016accuracy}) and
the discretization in time is performed with the following third-order TVD RK method:
\begin{eqnarray*}
 & \tb{W}_{i}^{(1)} = \tb{W}_{i}^{n} + \Delta t L(\tb{W}_{i}^{n}),\\
 & \tb{W}_{i}^{(2)} = \frac{3}{4}\tb{W}_{i}^{n} + \frac{1}{4} \tb{W}_{i}^{(1)} + \frac{1}{4}\Delta t L(\tb{W}_{i}^{(1)}), \\
 & \tb{W}_{i}^{n+1} = \frac{1}{3}\tb{W}_{i}^{n} + \frac{2}{3}\tb{W}_{i}^{(2)} + \frac{2}{3}L(\tb{W}_{i}^{(2)}),
\end{eqnarray*}
where $L(\tb{W}_{i}) $ represents the right-hand side  of \eqref{eq:sdmethod2}: see \cite{gottlieb1998total}.

\subsection{Critical case}\label{subsec:critical}
Let us suppose that $\tb{U}^* = (A^*(x), q^*)^T$, with $q^*$ constant, is a smooth stationary solution that reaches a critical state at $x_c$. Then,
$$-u^2 + a\frac{K}{\rho} \partial_a \phi(a) = 0 $$ 
at $x = x_c$ and the denominator in function \eqref{G} vanishes. This makes it difficult to approximate $A^*$ by solving numerically equation 
\eqref{problemP_A}. The strategy developed in \cite{gomez2021collocation} to approximate transcritical stationary solutions for the shallow water equations is adopted here.
The strategy is based on the following reasoning:  if the critical state in $x_c$ is isolated and $A^*$ is smooth then one has
$$
\partial_x A^*(x_c) = \lim_{x \to x_c} \partial_x A^*(x) = \lim_{x \to x_c} G(x, A^*(x)),
$$
and therefore, the function $G(x, A^*(x))$ has to have a limit in $x_c$, which implies that the numerator of $G$ has to vanish in $x_c$:
\begin{equation}\label{eq:critical_relation}
    \frac{A}{\rho}\phi(a) \partial_{x}K -a^2\frac{K}{\rho}\partial_a\phi(a) \partial_x A_{0} + \frac{A}{\rho}\partial_{x}p_{e} +  \frac{f}{\rho} {\blue - g A} = 0.
\end{equation}
Therefore, the derivative of $A^*$ at $x_c$ can be computed by applying L'Hôpital's rule to  an indeterminate limit of the form $\left(\frac{0}{0}\right)$:
\begin{equation}
    A_{x}(x_{c}) = \lim_{x \rightarrow x_{c}} -\frac{\frac{A}{\rho}\phi(a) \partial_{x}K -a^2\frac{K}{\rho}\partial_a\phi(a) \partial_x A_{0} + \frac{A}{\rho}\partial_{x}p_{e} + \frac{f}{\rho}{\blue - g A}}{-u^2 + a\frac{K}{\rho} \partial_a \phi(a)}.
\end{equation}

 After some straightforward but cumbersome computations, the following  formula is obtained:
\begin{equation}
    A_{x}(x_{c}) = \frac{A_{x}(x_{c})M_{1} + E}{A_{x}(x_{c})\alpha + M_{2}},
\end{equation}
where \eqref{eq:critical_relation} has been used and $\alpha$, $M_{1}$, $M_{2}$ and $E$ are given by:
$$\alpha = -\frac{K}{A_{0}}\left(3\partial_a \phi(a) + a\partial_{a}^{2}\phi(a)\right),$$
$$M_{1} = a\partial_{a}\phi(a) \partial_x K - \frac{K}{A_{0}}\left(a\partial_a \phi(a) + a^2\partial_{a}^{2}\phi(a)\right)\partial_x A_{0} - 2\frac{f}{A}$$
$$M_{2} = -a\partial_{a}\phi(a) \partial_x K + \frac{K}{A_{0}}\left(a\partial_a \phi(a) + a^2\partial_{a}^{2}\phi(a)\right)\partial_x A_{0}$$
$$E = A\phi(a)\partial^{2}_{x}K - Ka^2\partial_{a}\phi(a)\partial_x^2 A_{0} - 2a^{2}\partial_{a}\phi(a)(\partial_{x}K) (\partial_{x}A_{0}) + \frac{K}{A_{0}}a^2\left(a\partial^{2}_{a}\phi(a) + 2\partial_{a}\phi(a)\right)(\partial_{x}A_{0})^{2} + A\partial_{x}^{2}p_{e}-\rho\,A\,\partial_xg.$$
Therefore, $A^*(x_c)$ solves the second-degree equation
\begin{equation}
    A_{x}(x_{c})^{2} \alpha + A_{x}(x_{c})(M_{2}-M_{1}) - E = 0,
\end{equation}
that leads to
\begin{equation}\label{eq:critical_slope}
    A_{x}(x_{c}) = \frac{M_{2}-M_{1} \pm \sqrt{(M_{2}-M_{1})^{2} + 4\alpha E}}{2\alpha}.
\end{equation}

This analytical expression for the derivatives of smooth stationary solutions at critical points is used in the well-balanced reconstruction procedure as follows: when solving
\eqref{problemP_A}, if the state $\tb{U}_{i_0}$ is close to be critical, what is measured as follows
\begin{equation}\label{eq:critical_state}
\displaystyle \Big|\frac{|u_{i_0}|}{{c}_{i_0}} - 1\Big| < \epsilon,
\end{equation}
where
$$
u_{i_0} = \frac{q_{i_0}}{A_{i_0}}, \quad a_{i_0} = \frac{A_{i_0}}{A_0(x_{i_0})}, \quad c_{i_0} = \sqrt{\frac{K(x_{i_0})}{\rho}a_{i_0}\partial_{a}\phi({a_{i_0}})},
$$
and  $\epsilon$ is a selected threshold ($\epsilon= 10^{-8}$ has been considered here), then 
\begin{itemize}
    \item If 
    $$
\left| \frac{A_{i_0}}{\rho}\phi(a_{i_0}) \partial_{x} K(x_{i_0}) -a_{i_0}^2
\frac{K(x_{i_0})}{\rho}\partial_a\phi(a_{i_0}) \partial_x A_{0}(x_{i_0}) + \frac{A_{i_0}}{\rho}\partial_{x}p_{e}(x_{i_0}) +  \frac{f}{\rho}  {\blue - g(x_{i_0}) A_{i_0}}\right| > \epsilon,
$$
so that
    \eqref{eq:critical_relation} is not verified, it is assumed that \eqref{problemP_A} has no solution so that the well-balanced reconstruction is not necessary and the standard one is used. 
    \item Otherwise, it is assumed that the solution of \eqref{problemP_A} is transcritical and its derivative at $x_{i_0}$ is computed by
    $$ D = \begin{cases} \displaystyle\frac{M_{2}-M_{1} + \sqrt{(M_{2}-M_{1})^{2} + 4\alpha E}}{2\alpha} & \text{if $A$ is decreasing close to $x_{i_0}$;}\\
    \\
                         \displaystyle \frac{M_{2}-M_{1} - \sqrt{(M_{2}-M_{1})^{2} + 4\alpha E}}{2\alpha} & \text{otherwise}.      
    \end{cases}
 $$
 and the solution of \eqref{problemP_A} is approximated by
 $$
 A^*_{i_0}(x) = \tb{U}_{i_0} + D(x- x_{i_0})
 $$
 at the ${i_0}$-th cell. $A$ is decided to be increasing close to $i_0$ if $A_{i_0 + 1} > A_{i_0 -1}$. 
\end{itemize}

{\blue

\section{Application to networks} \label{section--networks}

In this section we describe the application of the numerical methods developed in the previous section to networks of vessels. Let us consider a network with $N_v$ vessels, $N_n$ bifurcation/junction vertexes, $N_b$ vertexes where one component of the state vector is to be prescribed and $N_t$ terminal vessels, i.e. vessels coupled to lumped-parameter models. Such models are commonly used to describe peripheral circulation in a simple and efficient manner. Each vessel constitutes a 1D domain, with blood flow governed by 
\begin{equation}
            \partial_t \tb{W}^{k} + \tb{\tb{\mathcal{A}}}(\tb{W}^{k})\partial_{x} \tb{W}^{k} = 0,
\end{equation}
as given by \eqref{eq:systemextended} and $x \in [0, L_k]$, with $L_k$ the length of the $k$-th vessel. For each 1D domain, we apply the numerical methodology developed in the previous sections. In the following we describe how we treat each one of the above mentioned features needed to model blood flow in a network of vessels.

\subsection{Coupling 1D domains}

Without loss of generality, we consider a network with a single shared vertex $V_p$, and $N_p$ vessels, all sharing $V_p$, and proceed by solving a Riemann problem at this vertex following the methodology described in \cite{muller2015high}. Briefly, at each time level $\hat{t}$, we solve the following Riemann problem:
    \begin{equation}\label{eq:Riemann_junction}
        \left\{\begin{array}{l}
            \partial_t \tb{W}^{k} + \tb{\tb{\mathcal{A}}}(\tb{W}^{k})\partial_{x} \tb{W}^{k} = 0, \\
            \tb{W}^{k}(x,\hat{t}) = \tb{W}_{1D}^{k,n},
        \end{array}\right.
    \end{equation}
with $k=1,\ldots,N_p$ and where $\tb{W}_{1D}^{k,n}$ are the reconstructed states $\mathbb{P}_{i}^{n,k}$ defined in Subsections \ref{subsec:first-order}, \ref{subsec:second-order}, \ref{subsec:third-order}, at time $\hat{t}$ and evaluated at node $V_p$, for each 1D vessel.
    In order to solve this Riemann problem, which implies computing $\tb{W}^{k}_*$, $k=1, \ldots, N_p$ intermediate states, we use an all-rarefaction approximate Riemann solver approach. In particular, we use wave relations proposed in \cite{toro2013flow}. Application of such wave relations, which include Riemann invariants for genuinely nonlinear characteristic fields and wave relations for the stationary contact discontinuity in correspondence of the jump in geometrical and mechanical properties at the bifurcation/junction point, results in the following nonlinear system:
\begin{itemize}
            \item Conservation of mass: 
            
            \begin{equation}
            \sum_{k=1}^{N_{p}} g_{p}^{k}q_{*}^{k} = 0,  \ \ g_{p}^{k} = \begin{cases}
                1, & x_{p}^{k} = L_{k},\\-1, & x_{p}^{k} = 0.
            \end{cases}
            \end{equation}
            \item Total pressure:
            \begin{equation}
            p(A_{*}^{1}) +\frac{1}{2}\rho\left(\frac{q_{*}^{1}}{A_{*}^{1}}\right) - p(A_{*}^{k}) +\frac{1}{2}\rho\left(\frac{q_{*}^{k}}{A_{*}^{k}}\right) = 0, k =2,...,N_{p}.
            \end{equation}
        \item Riemann invariants:
        \begin{equation}
        u_{*}^{k}-u_{1D}^{k} + g_{p}^{k}\beta^{k}=0, \ \ k=1,...,N_{p},  \label{eqa:riemanninv}
       \end{equation}
        with  $\beta^{k} = \int^{A^{k}_*}_{A^k_{1D}} c(\tau)/\tau d\tau$.
\end{itemize}
The obtained states $\tb{W}^{k}_*$, with $k=1, \ldots, N_p$, are then used to compute Godunov-type numerical fluxes at boundary cell faces for each vessel sharing vertex $V_p$. For the $k$-th vessel with $N^k_c$ computational cells, we set:
\begin{equation}
            \mathbb{F}^k_j = \tb{F}^{k}_j (\tb{W}^{k}_*)\;, j = \begin{cases}
                N^k_c + \frac{1}{2}, & x_{p}^{k} = L^{k},\\\frac{1}{2}, & x_{p}^{k} = 0\,,
            \end{cases} \label{eq:numfluxjunc}
 \end{equation}
that will replace the HLL flux \eqref{eq:hll} used in internal cell interfaces. Notably, in a general network, this procedure has to be performed for all $N_n$ bifurcation/junction vertexes, for each time stage involved in the time-marching algorithm defined by the Butcher tableau introduced in \eqref{RK4_collocation_gen}. 

\subsection{Prescribing boundary conditions}

The hyperbolic nature of the 1D blood flow model under study, together with the fact that flow is subcritical under physiological flow conditions, implies that only one state variable, pressure or flow, can be prescribed at the inlet/outlet of vessels. Consider a single vessel $k$ containing one of the $N_b$ \emph{boundary} vertexes. In general, we want to prescribe either flow rate or area/pressure over time at one of the boundary cell faces of the vessel. In this case, we solve \eqref{eqa:riemanninv} for the remaining component of state vector $\tb{W}^{k}_{*,b}$ and then compute the numerical flux at the cell face of interest as specified in \eqref{eq:numfluxjunc}. As for the case of coupling multiple 1D domains, also in this case $\tb{W}^{k}_{1D}$ are provided by the reconstructed polynomials $\mathbb{P}_{i}^{n, k}$, defined in Subsections \ref{subsec:first-order}, \ref{subsec:second-order}, \ref{subsec:third-order}.

\subsection{Coupling of 1D domains and lumped-parameter models}

In the application shown in Section \ref{section--results}, as well as in many works on computational haemodynamics, it is necessary to use lumped-parameter models to describe the interaction of 1D vessels with peripheral circulation. Conceptually, the model of interest to this work is a lumped-parameter model consisting of a proximal resistance $R^k_\mathrm{prox}$, connected to the $k$-th terminal vessel, a capacitor $C^k$ and a distal resistance $R^k_\mathrm{dist}$, connected to the capacitor and to a distal venous compartment. Such models are described by the following ordinary differential equation:
\begin{equation}
\frac{d P^k}{dt} = \frac{1}{C^k} \left( q^k_\mathrm{in} - (P^k - P_\mathrm{ven}) / R^k_\mathrm{dist}\right)\;.
\end{equation}
Here, $P_\mathrm{ven}$ is the distal venous pressure.

For any given time level $\hat{t}$ and at the boundary cell face connected to the lumped-parameter model, we have to compute the state vector $\tb{W}^{k}_{*,t}$, along with $q^k_\mathrm{in} $. This is achieved by first noting that:
$$
q^k_\mathrm{in} = q^{k}_{*}\;
$$
and then by enforcing pressure continuity:
$$
p(A^k_{*}) = P^k + q^{k}_{*}  R^k_\mathrm{prox} 
$$
and constancy of the appropriate Riemann invariant, as in \eqref{eqa:riemanninv}. Also in this case we have that $\tb{W}^{k}_{1D}$ are computed from the reconstructed polynomials $\mathbb{P}_{i}^{n, k}$. As for previous cases, we use $\tb{W}^{k}_{*,t}$ in \eqref{eq:numfluxjunc} to compute the numerical flux at the cell border of interest.

}

\section{Numerical tests} \label{section--results}

In this section, we examine various numerical tests to assess the effectiveness of the well-balanced numerical methods presented in preceding sections. The following numerical methods will be employed to discretize system \eqref{compform}:

\begin{itemize}
\item O1\_WB\_GHR\_HLL: first-order fully well-balanced method using the HLL numerical  flux;
\item O2\_WB\_GHR\_HLL:  second-order fully well-balanced extension;
\item O3\_WB\_GHR\_HLL:  third-order fully well-balanced extension;
\item O1\_noWB\_GHR\_HLL: first-order non well-balanced method using the HLL numerical  flux;
\item O2\_noWB\_GHR\_HLL:  second-order non well-balanced extension;
\item O3\_noWB\_GHR\_HLL:  third-order non well-balanced extension;
\end{itemize}

\begin{remark}
In all cases, to prevent the reconstruction of quantities with magnitudes similar to rounding errors, we employ the non-dimensional form \eqref{ndform}. This is especially crucial in the context of third-order scenarios. In practice we choose characteristic values using the following approach:
     $$
     \overline{L} = x_{f}-x_{0}, \quad \overline{A} = \frac{\sum_{i=1}^{N} A_{0}(x_{i})}{N}, \quad \overline{U} = \frac{\sum_{i=1}^{N}\left| \frac{q(x_{i}, 0)}{A_{0}(x_{i})}\right|}{N}, \quad \overline{K} = \frac{\sum_{i=1}^{N} K(x_{i})}{N}, \quad \overline{p}_{e} = \frac{\sum_{i=1}^{N}  p_{e}(x_{i})}{N}, \quad \overline{g} = \frac{\sum_{i=1}^{N}  g(x_{i})}{N}
     $$
     $$
     \overline{T} = \frac{\overline{L}}{\overline{U}}, \quad \overline{\Delta x} = \frac{\Delta x}{\overline{L}}, \quad \overline{t}_{\text{end}} = t_{\text{end}}{\overline{T}},
     $$
     where $[x_{0}, x_{f}]$ is the space interval; $\mathbf{W}_{0}(x)=[A(x,0), q(x,0), K(x), A_{0}(x), p_{e}(x)]^{T}$ is the initial condition; $g$ is the gravity function; $N$ is the number of points of the uniform mesh; $t_\text{end}$ is the final time of the simulation.
     
 Although all the methods are stable under the usual restriction
$$
\Delta t = CFL \frac{\Delta x}{\max_i\{ |\lambda_j(\tb{W}_i^n)|, \quad j =1,5. \}},
$$
with $CFL \in (0, 1]$, all numerical simulations have been conducted with $CFL = 0.5$, as this value guarantees the positivity of the HLL solver. For further details see \cite{bouchut2004nonlinear}.

\end{remark}

\subsection{Single vessel tests}

In this subsection we consider tests concerning single arteries or veins. In all test cases the friction term $f(x,t)=\gamma \pi \mu \frac{q}{A}$ is taken using $\gamma = 8$ and $\mu = 0.0045\,Pa\, s$.

\begin{table}[h]
    \small
  	\centering
  	\begin{tabular}{|c|c|c|c|c|c|c|c|c|c|c|}
  		\hline 
  		\textbf{Parameters} & $m$ & $n$ & $\rho$ & $L$ & $x_g$ & $K_{\text{ref}}$ & $A_{0,\text{ref}}$ & $p_{e,\text{ref}}$ & $t_{\text{end}}$ & $N$\\
  		\hline
  		Well-balanced tests & & & & & & & & & & \\
  		\hline 
  		Test 1 & 10 & -1.5 & 1050 & 0.015 & 0.5L & 100 & $0.00015^{2}\pi$ & 0 & 1 & 100 \\ 
  		\hline
  		Test 2 & 10 & -1.5 & 1050 & 0.015 & 0.5L & 100 & $0.00015^{2}\pi$ & 0 & 0.5 & 100 \\
        \hline
  		Test 3 & 10 & -1.5 & 1050 & 0.015 & - & 1000 & $0.00015^{2}\pi$ & 1000 & 1 & 100 \\
  		\hline
        Test 4 & 10 & -1.5 & 1050 & 0.015 & - & 1000 & $0.00015^{2}\pi$ & 1000 & 1 & 100 \\
  		\hline
  		Test 5 & 10 & -1.5 & 1050 & 0.015 & - & 100 & $0.00015^{2}\pi$ & 0 & 0.3 & 100 \\ 
  		\hline 
  		Accuracy test & & & & & & & & & &\\
  		\hline
  		Test 6 & 0.5 & 0 & 1050 & 5 & - & - & - & - & 0.4 & - \\ 
  		\hline 
  		Riemann problems & & & & & & & & & &\\
  		\hline 
        Test 7 & 10 & -1.5 & 1050 & 0.5 & 0.5L & 10 & 1.0E-4 & 66.661 & 0.02 & 1000 \\ 
  		\hline
  	\end{tabular} 
	  	\caption{Parameters of numerical tests: $m [-]$, $n[-]$, $\rho[{Kg}/{m}^{3}]$, $L[{m}]$, $x_{g}[{m}]$, $K_{{ref}}[{Pa}]$, $A_{0,{ref}}[{m}^{2}]$, $p_{e,{ref}}[{Pa}]$ $t_{{end}}[{s}]$, $N[-]$.}

  	\label{tab:TestsParameters}
\end{table}

\begin{table}[!h]
    \scriptsize
  	\centering
  	\begin{tabular}{|c|c|c|c|c|c|c|}
  		\hline 
  		\textbf{Left values} & $A_{l}$ & $q_l$ & $K_l$ & $A_{0,l}$ & $p_{e,l}$\\
  		\hline
  		Well-balanced tests & & & & &\\
  		\hline 
  		Test 1 & $A_{0,\text{ref}}$ & 4E-10 & $K_{\text{ref}}$ & $A_{0,\text{ref}}$ & 0  \\ 
  		\hline 
  		Test 2 & $A_{0,\text{ref}}$ & 4E-6 & $K_{\text{ref}}$ & $A_{0,\text{ref}}$ & 0  \\
        \hline
        Test 3 & $A_{0,\text{ref}}$ & 4E-10 & - & - & -  \\
  		\hline
        Test 4 & $A_{0,\text{ref}}$ & 4E-10 & - & - & -  \\
  		\hline
  		Test 5 & - & - & - & - & -  \\
  		\hline
  		Accuracy test & & & & & \\
  		\hline
  		Test 6 & 2.0456$A_{0,l}$ & 0 & 58725 & 5E-4 & 10000
  		\\
        \hline
  		Riemann problems & & & & & \\
  		\hline
  		Test 7 & 1.5$A_{0,l}$ & 1.8 & $5K_{\text{ref}}$ & $1.1A_{0,\text{ref}}$ & $p_{e,\text{ref}}$
        \\
  		\hline 
  		\hline 
  		\textbf{Right values} & $A_{r}$ & $q_r$ & $K_r$ & $A_{0,r}$ & $p_{e,r}$\\
  		\hline
  		Well-balanced tests & & & & &\\
  		\hline 
  		Test 1 & - & - & $0.98K_{\text{ref}}$ & $0.98A_{0,\text{ref}}$ & 0  \\ 
  		\hline 
  		Test 2 & - & - & $0.98K_{\text{ref}}$ & $0.98A_{0,\text{ref}}$ & 0  \\
  		\hline
        Test 3 & - & - & - & - & -  \\
  		\hline 
        Test 4 & - & - & - & - & -  \\
  		\hline 
  		Test 5 & - & - & - & - & -  \\
  		\hline
  		Accuracy test & & & & & \\
  		\hline
  		Test 6 & - & - & - & - & -
  		\\
        \hline
        Riemann problems & & & & & \\
  		\hline
  		Test 7 & 1.1$A_{0,r}$ & 2.0 & $50K_{\text{ref}}$ & $1.2A_{0,\text{ref}}$ & $0.1p_{e,\text{ref}}$
        \\
  		\hline
  	\end{tabular} 
	  	\caption{Left-right initial values of Riemann problems: $A_{l,r} [{m}^{2}]$, $u_{l,r}[{m/s}]$, $K_{l,r}[{Pa}]$, $A_{0,l,r}[{m}^2]$, $p_{e,l,r}[{Pa}]$.}

  	\label{tab:TestInitialValues}
\end{table}

\subsubsection{Well-balanced property}

The aim of this subsection is to verify the well-balanced property of the provided numerical schemes. We consider an $N$-point uniform mesh for the spatial interval $[0, L]$ where we impose the given stationary solution as boundary conditions. For Tests 1-2, we will investigate discontinuous stationary solutions characterized by the following discontinuous coefficients:

\begin{equation}\label{eq:K_A0_pe_discontinuous}
K(x) = \begin{cases}
    K_{l} & \text{if $x<x_{g}$},\\
    K_{r} & \text{if $x\geq x_{g}$},
\end{cases} 
\ \ A_{0}(x) = \begin{cases}
    A_{0,l} & \text{if $x<x_{g}$},\\
    A_{0,r} & \text{if $x\geq x_{g}$},
    \end{cases} 
    \ \ p_{e,0}(x) = \begin{cases}
    p_{e,l} & \text{if $x<x_{g}$},\\
p_{e,r} & \text{if $x\geq x_{g}$},
\end{cases}
\end{equation}
where the left and the right values are given in Table \ref{tab:TestInitialValues}. The value of $m$, $n$, $\rho$, $L$, $x_{g}$, $K_{\text{ref}}$, $A_{0,\text{ref}}$, $p_{e,\text{ref}}$, $t_{\text{end}}$ and $N$, are given in Table \ref{tab:TestsParameters}. In all well-balanced test cases we consider the following gravity function:
$$g(x) = g_{\text{ref}} - \frac{x}{2\,L}g_{\text{ref}},$$
being $g_{\text{ref}} = 9.81 (m/s^2)$.

\subsubsection*{Test 1: Discontinuous subcritical stationary solution}

Our initial condition is given by a subcritical stationary solution $\mathbf{W}_{0}(x)= (A(x, 0), q(x, 0), K(x), A_{0}(x), p_{e}(x))$ such that:
$$A(0, 0) = A_{l}, \ \ q(0, 0) = q_{l},$$
where $A_{l}$ and $q_{l}$ are given in Table \ref{tab:TestInitialValues} and $K(x)$, $A_{0}(x)$, $p_{e}(x)$ are given by \eqref{eq:K_A0_pe_discontinuous}. This \textit{discrete} stationary solution is constructed with the same collocation methods used for each method, i.e., the second-order 1-stage RK collocation method \eqref{RK2_collocation} for the first- and second-order ones and the fourth-order 2-stage Gauss-Legendre collocation RK method \eqref{RK4_collocation} for the third-order one. In Table \ref{tab:Test1Errors} we observe that the fully well-balanced methods are the only ones that preserve the stationary solution.

\begin{table}[ht]
    \small
  	\centering
  	\begin{tabular}{|c|c|c|}
  		\hline 
  		Scheme & $||\Delta A/A_0||_1$& $||\Delta u||_1$\\
  		\hline 
  		O1\_WB\_GHR\_HLL & 4.66E-19 & 3.32E-18\\
  		O2\_WB\_GHR\_HLL& 1.05E-18 & 2.51E-18 \\
  		O3\_WB\_GHR\_HLL& 1.00E-17 & 8.95E-18 \\
  		O1\_noWB\_GHR\_HLL & 1.23E-05 & 3.12E-06\\
  		O2\_noWB\_GHR\_HLL& 5.68E-06 & 1.49E-06 \\
  		O3\_noWB\_GHR\_HLL& 4.70E-06 & 1.21E-06 \\
  		
  		\hline 
  	\end{tabular} 
	  	\caption{Test 1:  $L^{1}$ errors at  time $t = 1\,s$. }

  	\label{tab:Test1Errors}
\end{table}

\subsubsection*{Test 2: Discontinuous supercritical stationary solution}

The initial condition is given in this case by a supercritical stationary solution $$\mathbf{W}_{0}(x)= (A(x, 0), q(x, 0), K(x), A_{0}(x), p_{e}(x))$$ such that:
$$A(0, 0) = A_{l}, \ \ q(0, 0) = q_{l},$$
where $A_{l}$ and $q_{l}$ are given in Table \ref{tab:TestInitialValues} and $K(x)$, $A_{0}(x)$, $p_{e}(x)$ are given by \eqref{eq:K_A0_pe_discontinuous}. The only difference with respect the previous test is the value of $q_{l}$. This \textit{discrete} stationary solution is constructed again with the same collocation methods used for each method. As for the previous test, in Table \ref{tab:Test2Errors} we observe that the fully well-balanced methods are the only ones that preserve the stationary solution.

\begin{table}[ht]
    \small
  	\centering
  	\begin{tabular}{|c|c|c|}
  		\hline 
  		Scheme & $||\Delta A/A_0||_1$& $||\Delta u||_1$\\
  		\hline 
  		O1\_WB\_GHR\_HLL & 1.75E-16 & 9.81E-16\\
  		O2\_WB\_GHR\_HLL& 1.61E-16 & 9.81E-16 \\
  		O3\_WB\_GHR\_HLL& 4.55E-17 & 6.54E-15 \\
  		O1\_noWB\_GHR\_HLL & 2.98E-05 & 1.15E-03\\
  		O2\_noWB\_GHR\_HLL& 1.10E-06 & 4.30E-05 \\
  		O3\_noWB\_GHR\_HLL& 1.19E-06 & 4.38E-05 \\
  		
  		\hline 
  	\end{tabular} 
	  	\caption{Test 2:  $L^{1}$ errors at  time $t = 0.5\,s$. }

  	\label{tab:Test2Errors}
\end{table}

\subsubsection*{Test 3: Smooth subcritical stationary solution}

In this test we consider as initial condition a subcritical stationary solution $$\mathbf{W}_{0}(x)= (A(x, 0), q(x, 0), K(x), A_{0}(x), p_{e}(x))$$ such that:
$$A(0, 0) = A_{s}, \ \ q(0, 0) = q_{s},$$
where 
$$ A_{s}=A_{0,\text{ref}}, \ \ q_{s} = 4\cdot 10^{-10}\,m^3/s\,,$$ and $K(x)$, $A_{0}(x)$, $p_{e}(x)$ are given by 

$$K(x)=K_{\text{ref}} + 0.01 x\frac{ K_{\text{ref}}}{L}, \ \  A_{0}(x)=A_{0,\text{ref}} + 0.01 x\frac{ A_{0,\text{ref}}}{L},\ \ p_{e}(x)=p_{e,\text{ref}} + 0.01 x\frac{ p_{e,\text{ref}}}{L},$$
where $L$, $K_{\text{ref}}$, $A_{0,\text{ref}}$ and $p_{e,\text{ref}}$ can be found in Table \ref{tab:TestsParameters}. A \textit{discrete} stationary solution is constructed again with the same collocation methods used for each method. We observe in Table \ref{tab:Test3Errors} that only the well-balanced methods are able to preserve the stationary solution.

\begin{table}[ht]
    \small
  	\centering
  	\begin{tabular}{|c|c|c|}
  		\hline 
  		Scheme & $||\Delta A/A_0||_1$& $||\Delta u||_1$\\
  		\hline 
  		O1\_WB\_GHR\_HLL & 1.10E-18 & 1.82E-17\\
  		O2\_WB\_GHR\_HLL& 1.85E-18 & 9.72E-17 \\
  		O3\_WB\_GHR\_HLL& 2.08E-16 & 1.49E-16 \\
  		O1\_noWB\_GHR\_HLL & 3.55E-07  & 5.07E-07\\
  		O2\_noWB\_GHR\_HLL& 3.55E-06 & 2.58E-05 \\
  		O3\_noWB\_GHR\_HLL& 1.89E-08 & 6.21E-08 \\
  		
  		\hline 
  	\end{tabular} 
	  	\caption{Test 3:  $L^{1}$ errors at  time $t = 1\,s$. }

  	\label{tab:Test3Errors}
\end{table}

\subsubsection*{Test 4: Perturbed smooth subcritical stationary solution}

The initial condition is given by:
\begin{equation}
\tb{W}_{0}^{P}(x) = \tb{W}_{0}(x) + \tb{\delta}(x),
\end{equation}
where $\tb{W}_{0}(x)$ is the stationary solution defined in Test 3, $K(x)$, $A_{0}(x)$ and $ p_e(x)$ are also given as in Test 3 and
$$
\tb{\delta}(x) = [10^{-10}e^{-2000000(x-0.0075)^{2}},0,0,0,0]^T.
$$
This test is devoted to show the behaviour of the different methods when a perturbation is posed into a stationary solution. In Figures \ref{fig:Test4_WB_vs_noWB_A/A0}, \ref{fig:Test4_WB_vs_noWB_u} we show the numerical solution of first-, second- and third-order well-balanced and non well-balanced methods at different times together with the computed stationary solution given in Test 3. We observe that only the well-balanced methods are able to recover the unperturbed stationary solution. 

\begin{figure}[h]
	\begin{subfigure}[h]{0.49\textwidth}
		\centering
		\includegraphics[width=1\linewidth]{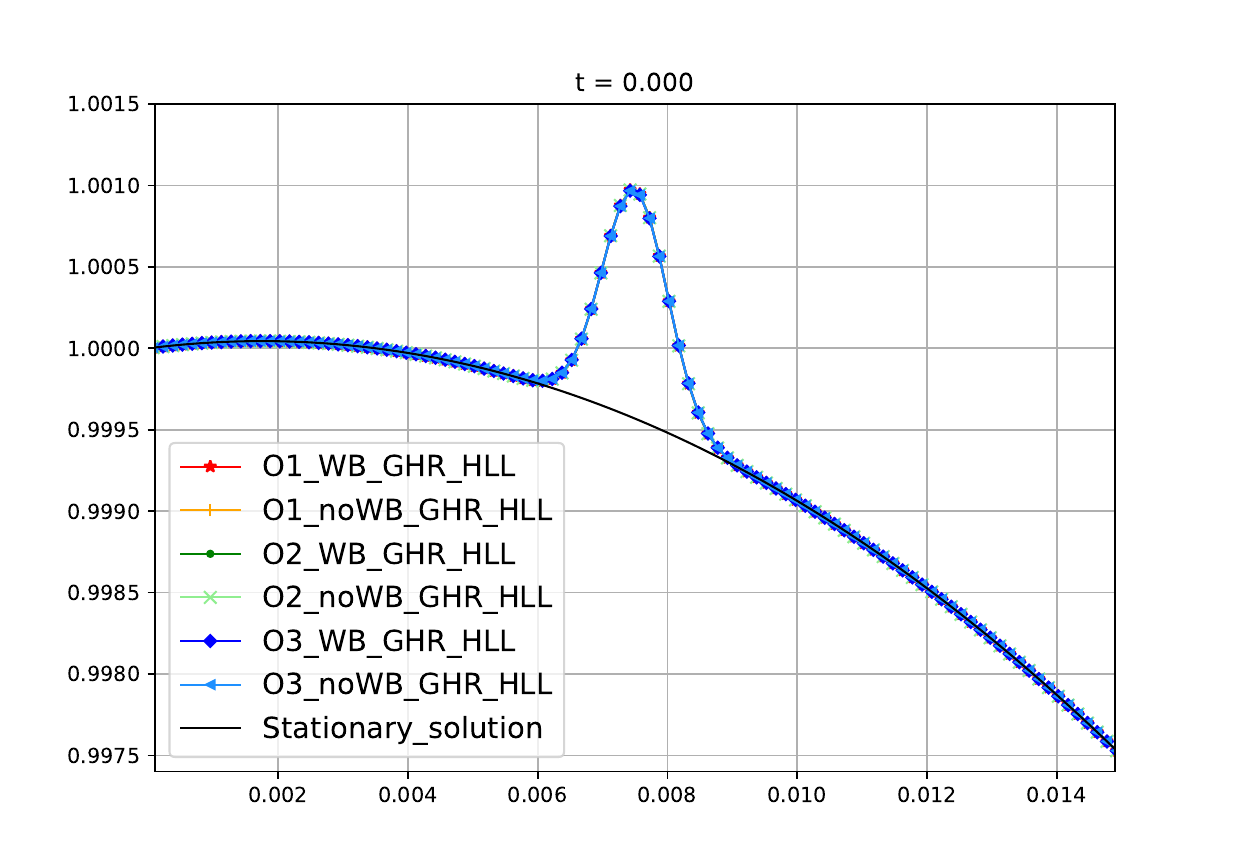}
		\caption*{Variable $A/A_{0}$}
		\label{fig:Test4_WB_vs_noWB_t_0_AA0}
	\end{subfigure}
	\begin{subfigure}[h]{0.49\textwidth}
		\centering
		\includegraphics[width=1\linewidth]{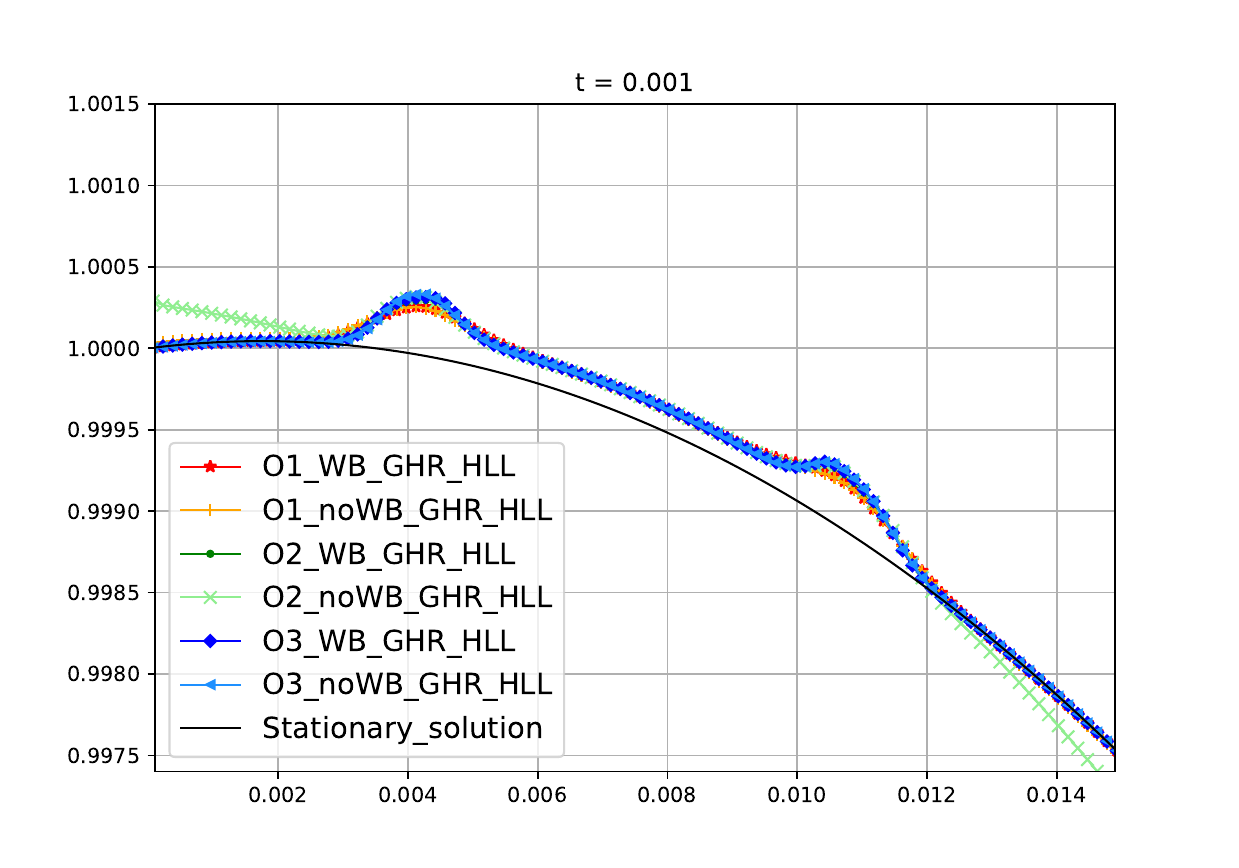}
		\caption*{Variable $A/A_{0}$}
		\label{fig:Test4_WB_vs_noWB_t_0001_AA0}
	\end{subfigure}
    \begin{subfigure}[h]{0.49\textwidth}
		\centering
		\includegraphics[width=1\linewidth]{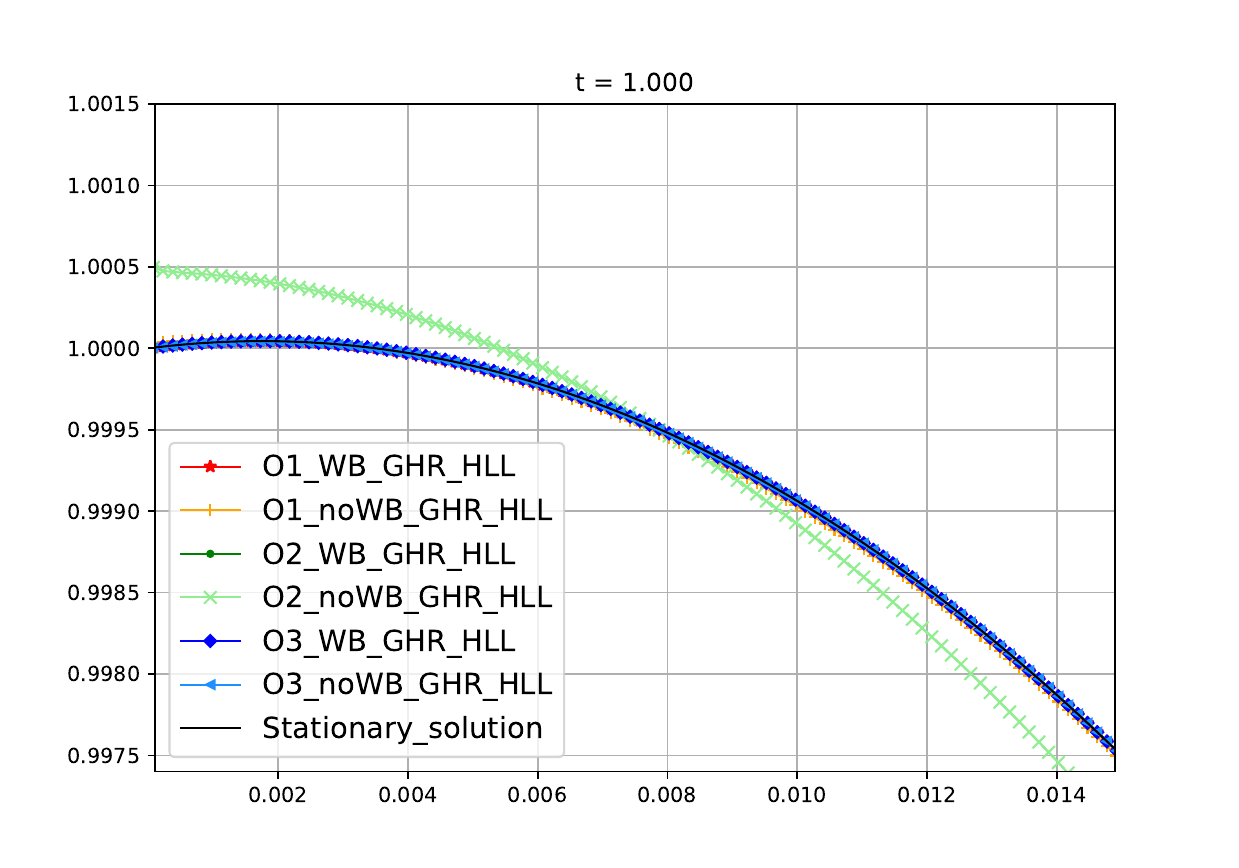}
		\caption*{Variable $A/A_{0}$}
		\label{fig:Test4_WB_vs_noWB_t_1_AA0}
	\end{subfigure}
	\begin{subfigure}[h]{0.49\textwidth}
		\centering
		\includegraphics[width=1\linewidth]{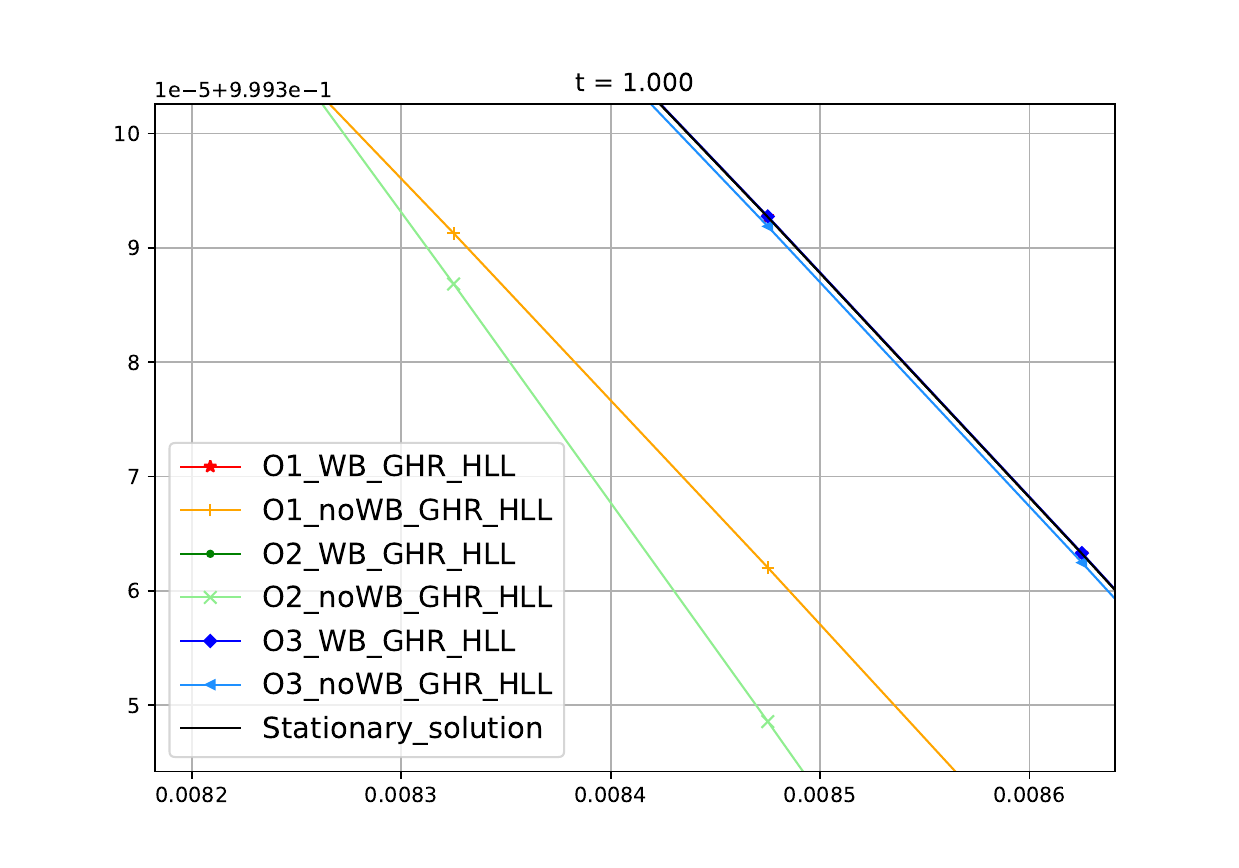}
		\caption*{Variable $A/A_{0}$. Zoom}
		\label{fig:Test4_WB_vs_noWB_t_1_AA0_zoom}
	\end{subfigure}
	\caption{Test 4: numerical solution of first-, second-, third-order well-balanced and non well-balanced methods at different times. Variable $A/A_0$.}
	\label{fig:Test4_WB_vs_noWB_A/A0}
\end{figure}

\begin{figure}[h]
	\begin{subfigure}[h]{0.49\textwidth}
		\centering
		\includegraphics[width=1\linewidth]{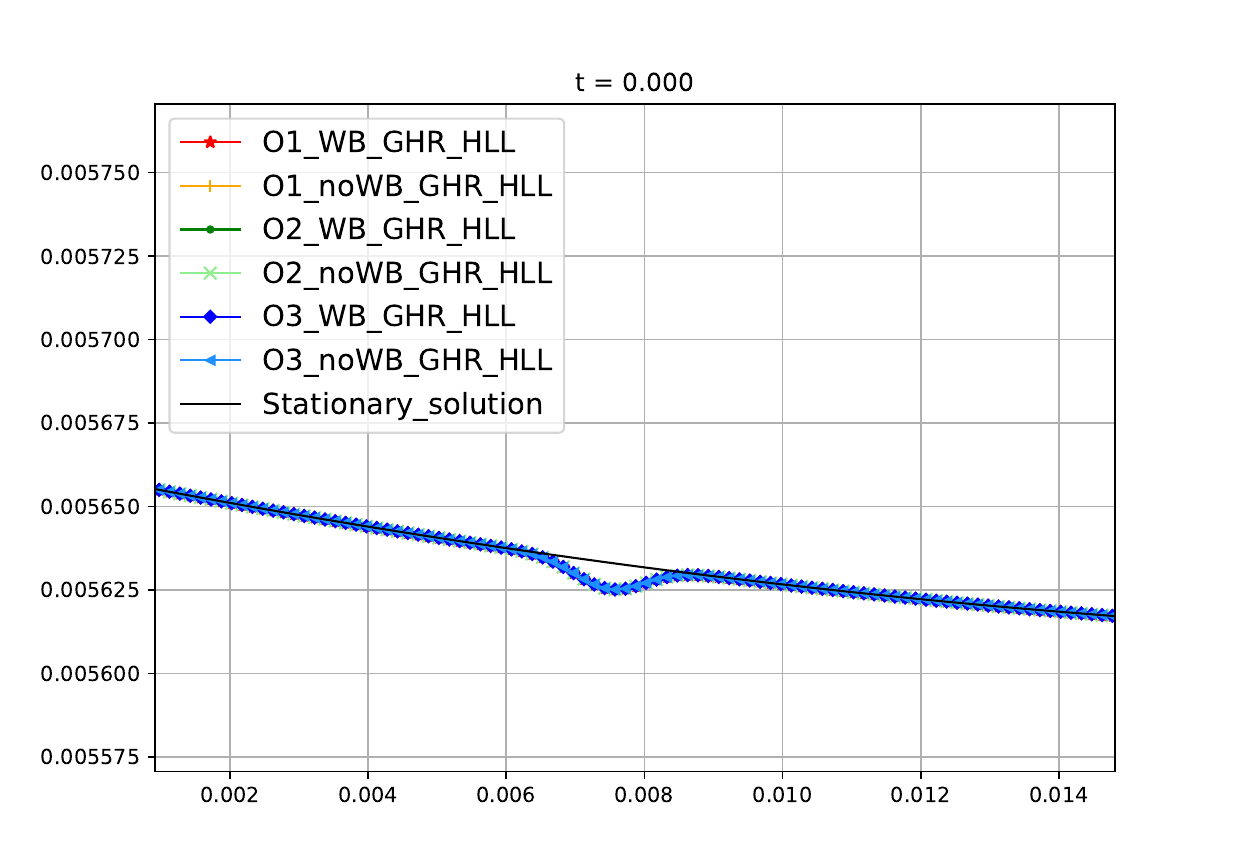}
		\caption*{Variable $u$}
		\label{fig:Test4_WB_vs_noWB_t_0_u}
	\end{subfigure}
	\begin{subfigure}[h]{0.49\textwidth}
		\centering
		\includegraphics[width=1\linewidth]{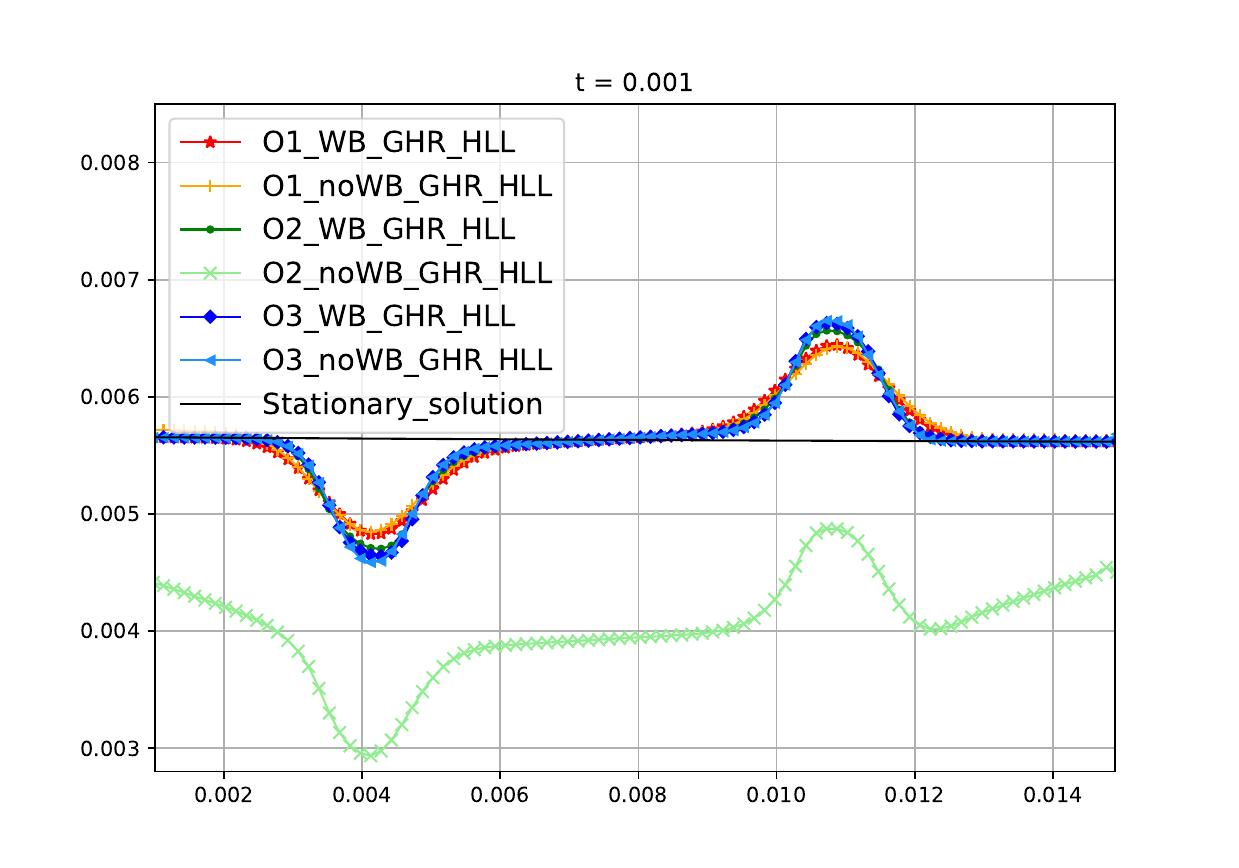}
		\caption*{Variable $u$}
		\label{fig:Test4_WB_vs_noWB_t_0001_u}
	\end{subfigure}
    \begin{subfigure}[h]{0.49\textwidth}
		\centering
		\includegraphics[width=1\linewidth]{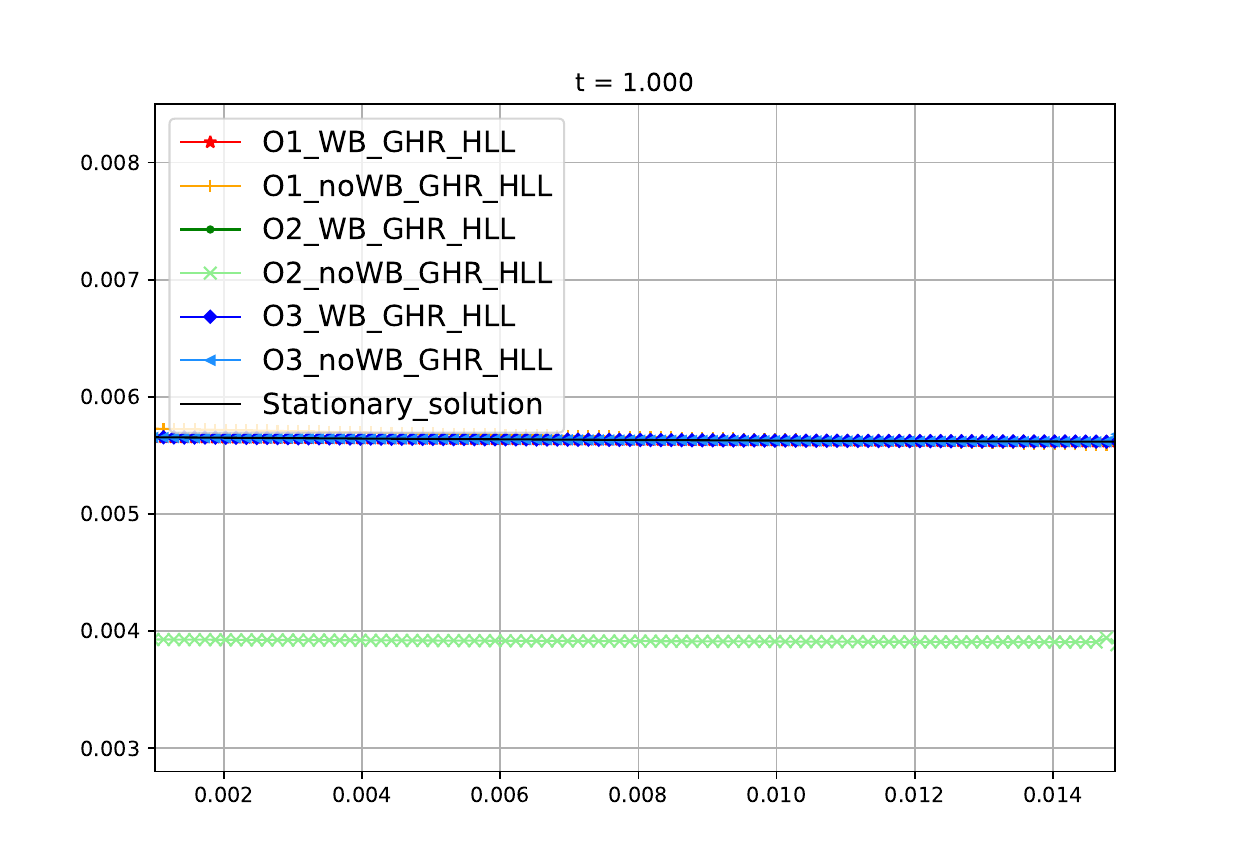}
		\caption*{Variable $u$}
		\label{fig:Test4_WB_vs_noWB_t_1_u}
	\end{subfigure}
	\begin{subfigure}[h]{0.49\textwidth}
		\centering
		\includegraphics[width=1\linewidth]{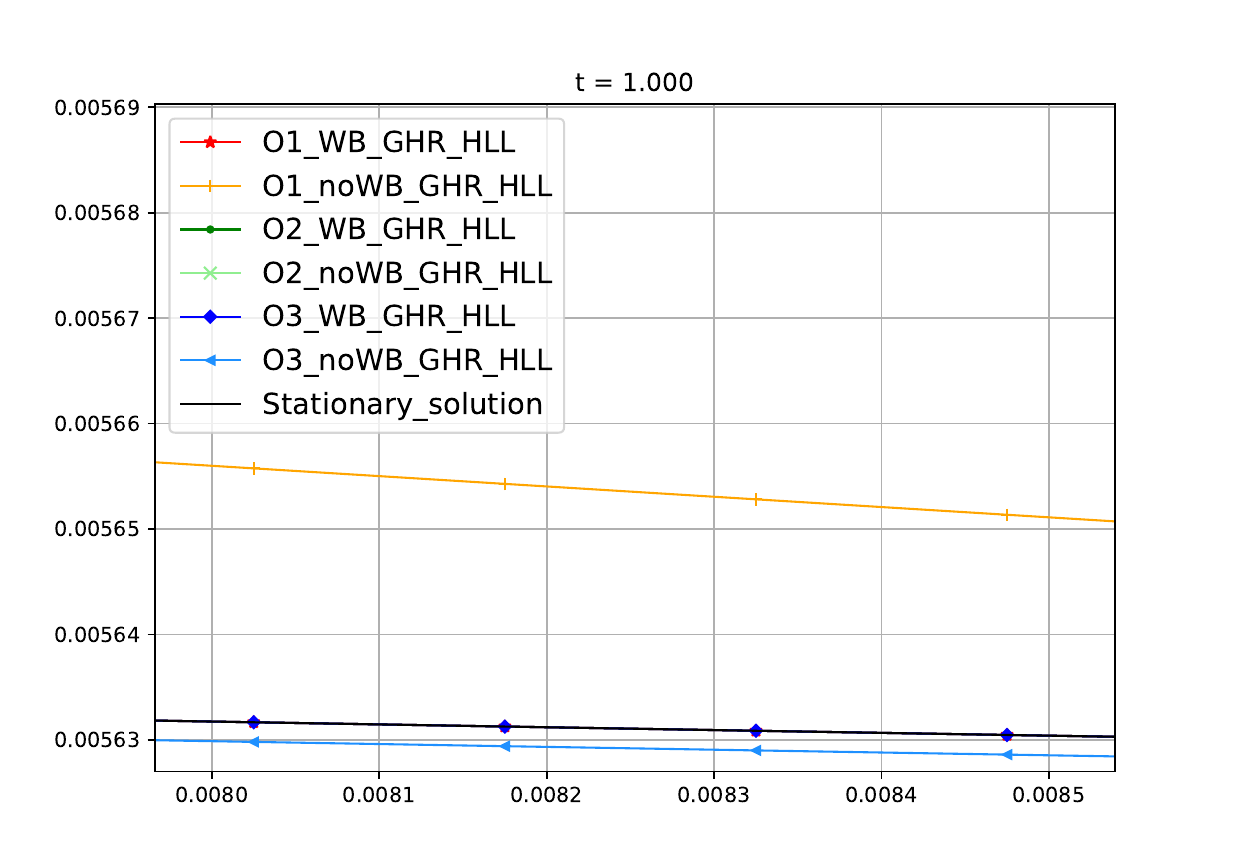}
		\caption*{Variable $u$. Zoom}
		\label{fig:Test4_WB_vs_noWB_t_1_u_zoom}
	\end{subfigure}
	\caption{Test 4: numerical solution of first-, second-, third-order well-balanced and non well-balanced methods at different times. Variable $u$.}
	\label{fig:Test4_WB_vs_noWB_u}
\end{figure}

\subsubsection*{Test 5: Smooth transcritical stationary solution}

In this test we consider as initial condition a transcritical stationary solution $$\mathbf{W}_{0}(x)= (A(x, 0), q(x, 0), K(x), A_{0}(x), p_{e}(x))$$ such that:
$$A(x_{c}, 0) = A_{c}, \ \ q(x_{c}, 0) = q_{c},$$
where 
$$x_{c} = L/2 + dx/2, \ \ A_{c}=1.1A_{0}(x_{c}), \ \ q_{c} = 1.25232 \cdot 10^{-7}\,m^3/s,$$ and $K(x)$, $A_{0}(x)$, $p_{e}(x)$ are given by 

$$K(x)=K_{\text{ref}}e^{-4log(1/1.1)(x-x_{c})^{2}}, \ \  A_{0}(x)=A_{0,\text{ref}}e^{-4log(1/1.1)(x-x_{c})^{2}},\ \ p_{e}(x)=-x\left(\gamma\pi\mu \frac{q_{c}}{A_{c}^{2}}-g_cA_c\right),$$
where $K_{\text{ref}}$ and $A_{0,\text{ref}}$ can be found in Table \ref{tab:TestsParameters} and $g_c = g(x_c)= 7.332975$. It can be proven that $W_{0}(x_{c})$ is a critical state and verifies \eqref{eq:critical_relation}. From this critical state we construct the transcritical stationary solution in the domain $[0,L]$ using again the collocation methods seen in the previous sections. In this case we only show the results of the first- and second-order methods. We observe in Figure \ref{fig:Test5_WB_vs_noWB} that only the well-balanced methods are able to preserve the transcritical stationary solution while the non well-balanced schemes converge to a different solution.

\begin{figure}[h]
	\begin{subfigure}[h]{0.49\textwidth}
		\centering
		\includegraphics[width=1\linewidth]{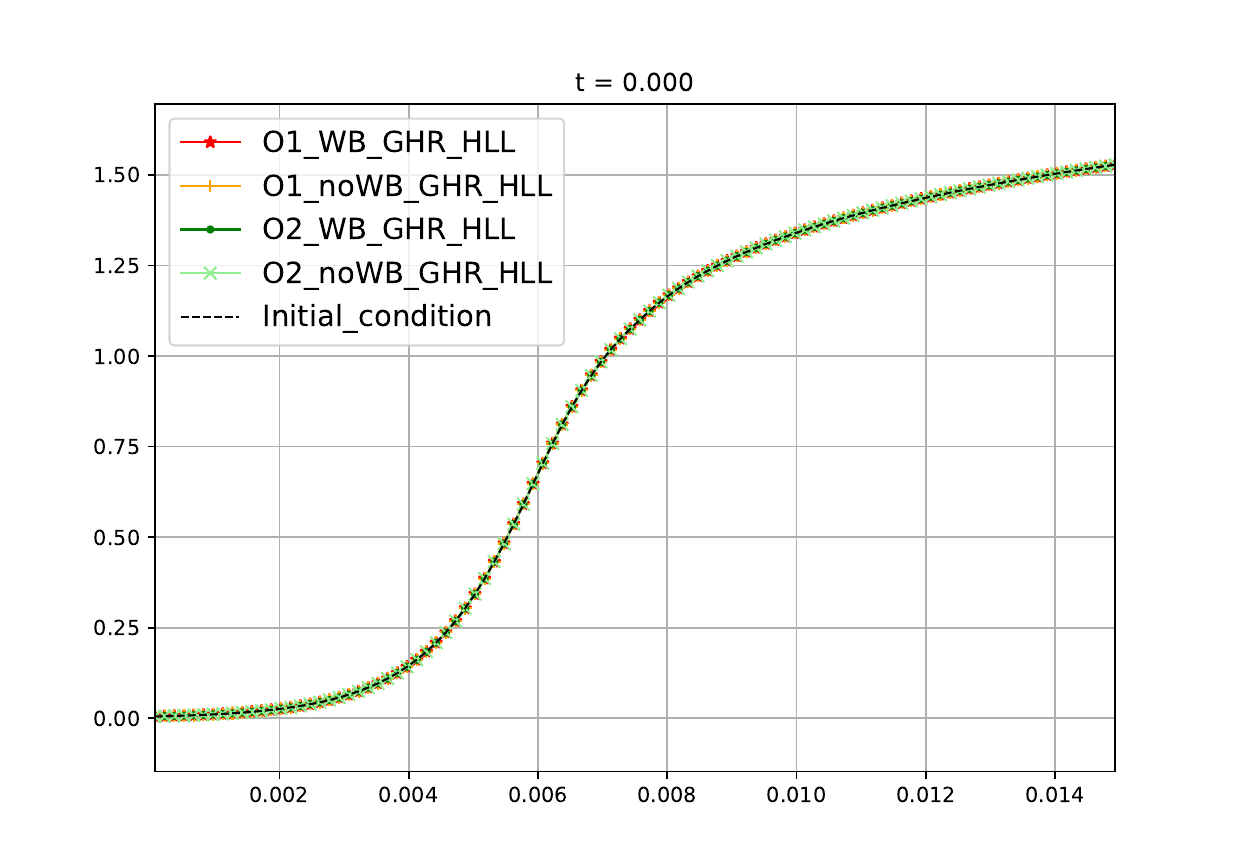}
		\caption*{Variable $A/A_{0}$}
		\label{fig:Test5_WB_vs_noWB_t_0_AA0}
	\end{subfigure}
	\begin{subfigure}[h]{0.49\textwidth}
		\centering
		\includegraphics[width=1\linewidth]{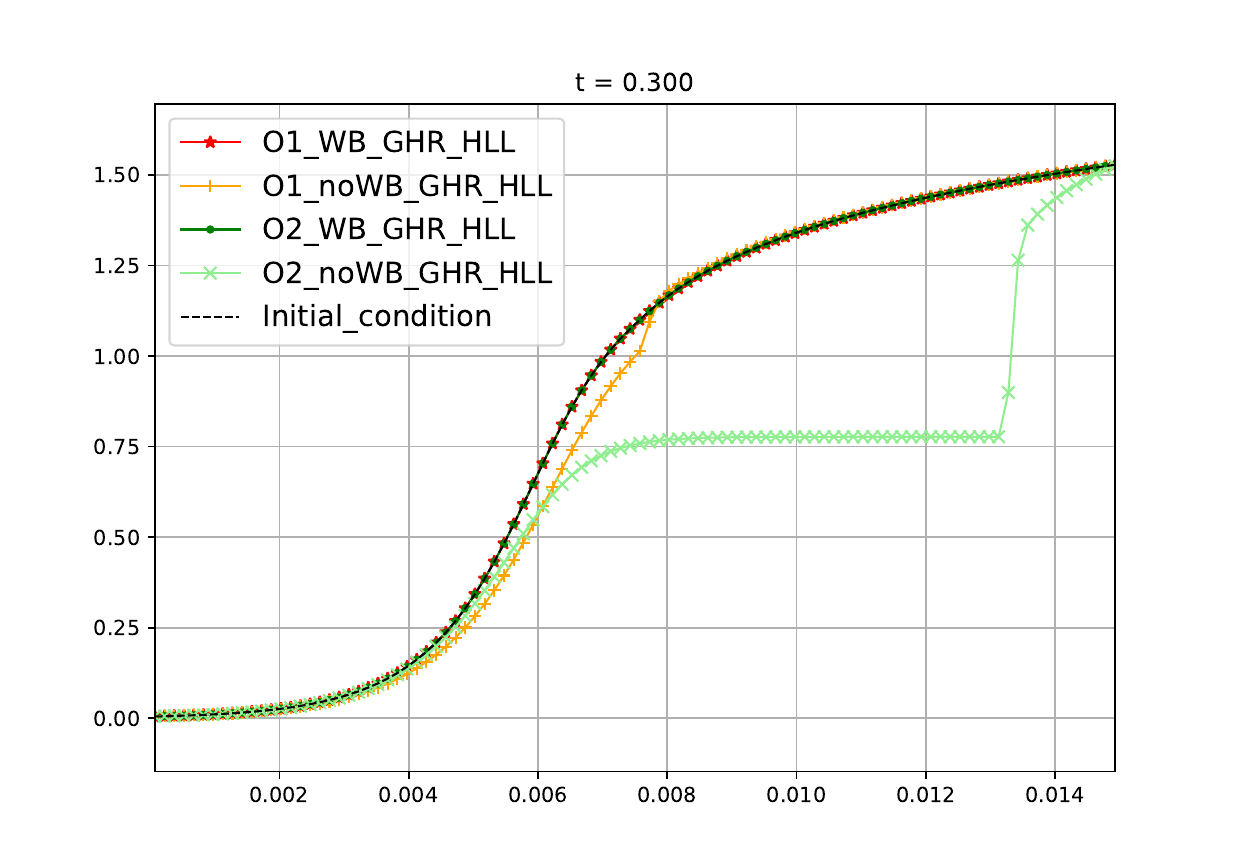}
		\caption*{Variable $A/A_{0}$}
		\label{fig:Test5_WB_vs_noWB_t_003_AA0}
	\end{subfigure}
	\begin{subfigure}[h]{0.32\textwidth}
		\centering
		\includegraphics[width=1.1\linewidth]{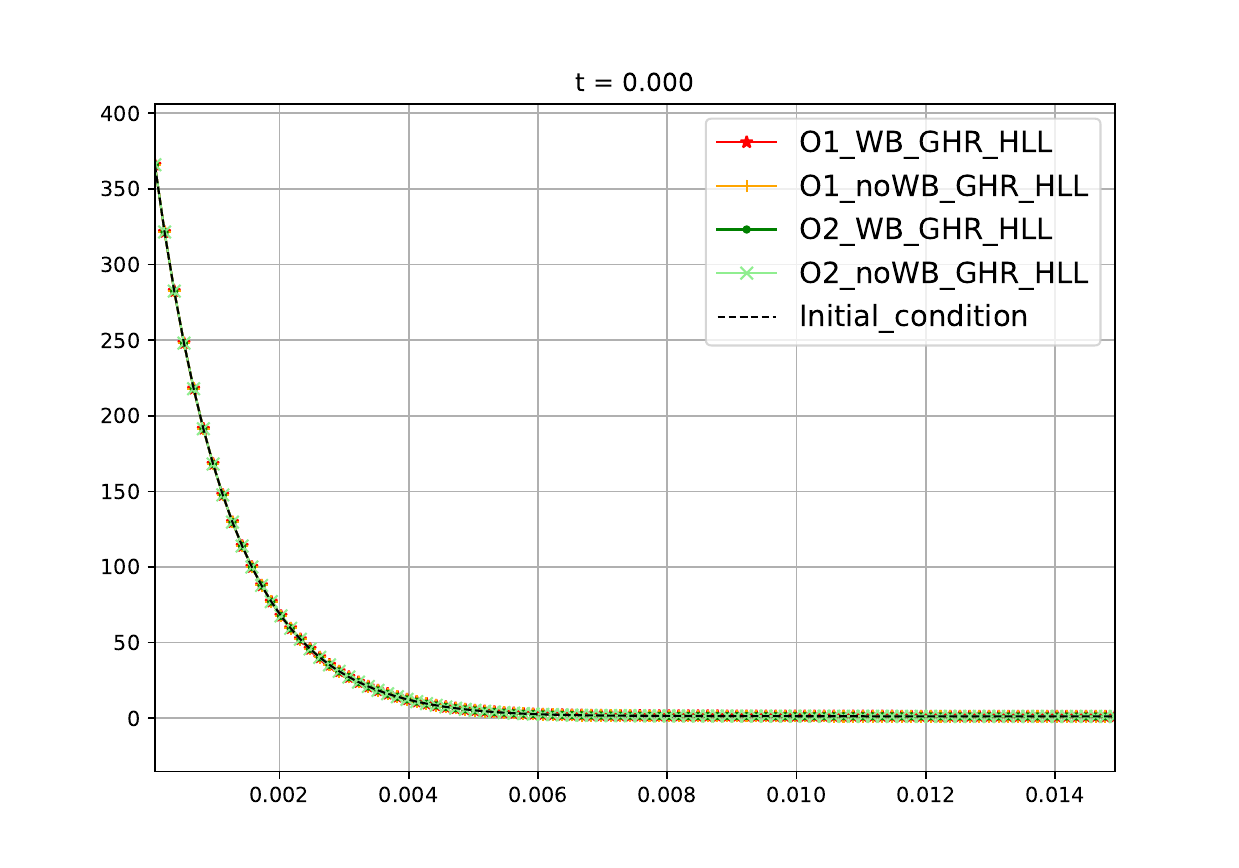}
		\caption*{Variable $u$.}
		\label{fig:Test5_WB_vs_noWB_t_0_u}
	\end{subfigure}
    \begin{subfigure}[h]{0.32\textwidth}
		\centering
		\includegraphics[width=1.1\linewidth]{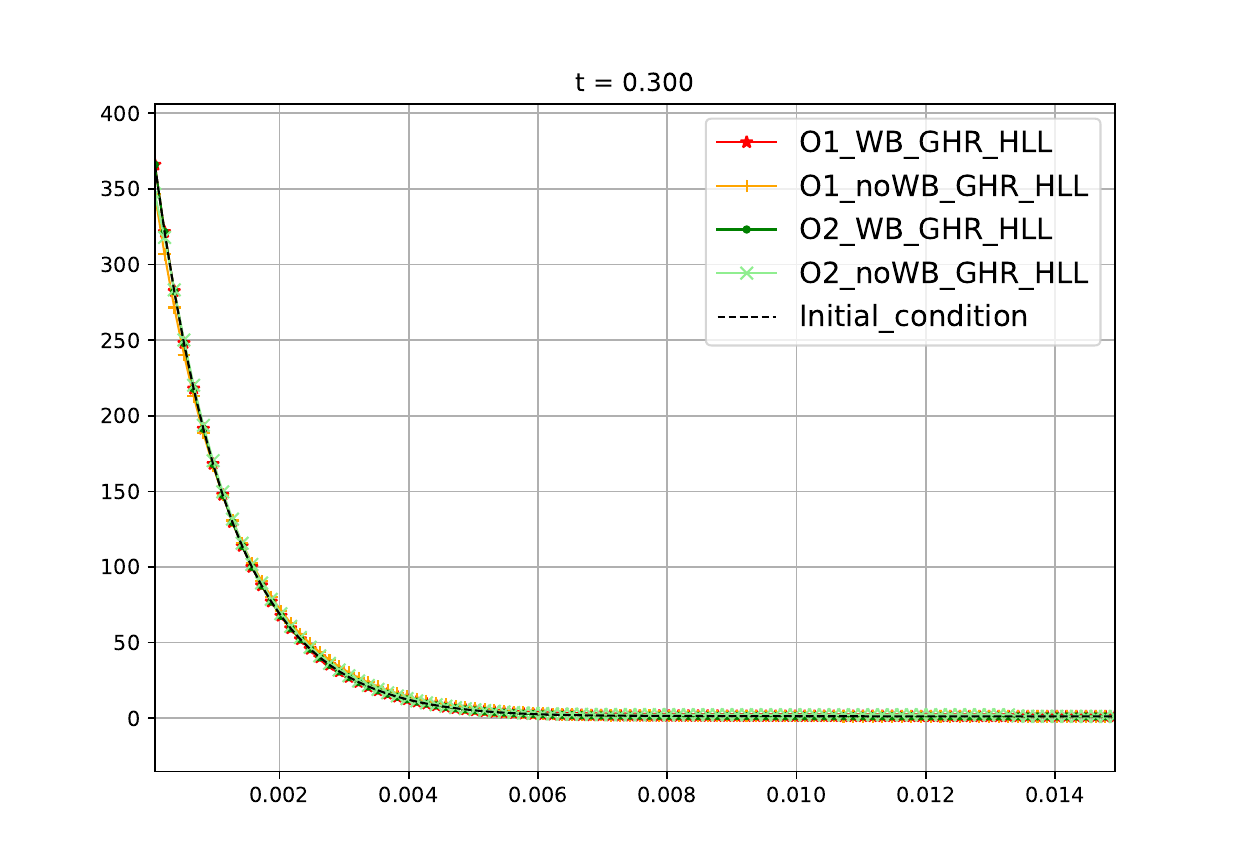}
		\caption*{Variable $u$.}
		\label{fig:Test5_WB_vs_noWB_t_003_u}
	\end{subfigure}
	\begin{subfigure}[h]{0.32\textwidth}
		\centering
		\includegraphics[width=1.1\linewidth]{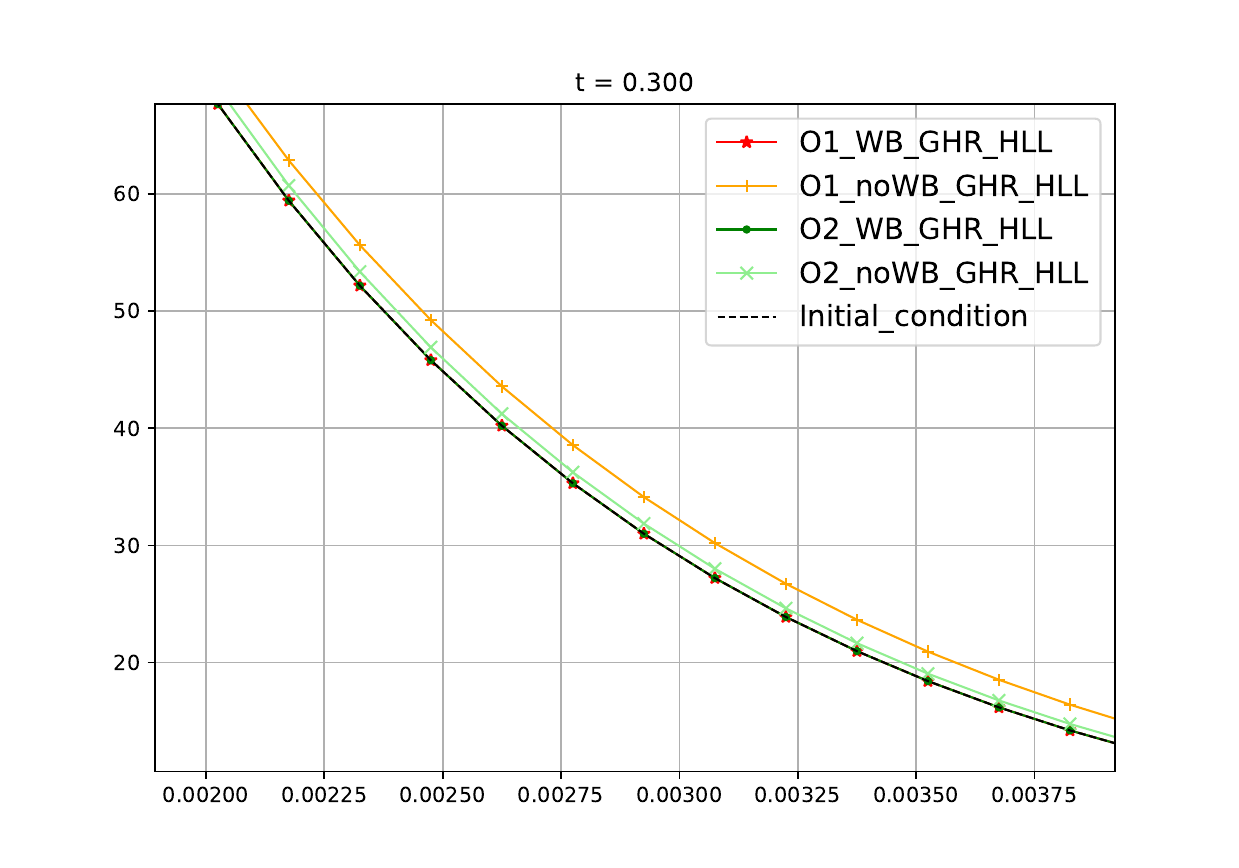}
		\caption*{Variable $u$. Zoom}
		\label{fig:Test5_WB_vs_noWB_t_003_u_zoom}
	\end{subfigure}
	\caption{Test 5: numerical solution of first-, second-order well-balanced and non well-balanced methods at different times. Top: Variable $A/A_{0}$. Bottom: Variable $u$.}
	\label{fig:Test5_WB_vs_noWB}
\end{figure}

\subsubsection{Accuracy test}

This subsection is devoted to show that the order of accuracy of the methods is reached.

\paragraph{Test 6: Perturbed smooth subcritical stationary solution}

This test is based on Test 6 in \cite{pimentelgarcia2023high}, with the inclusion of the friction term. The initial condition is given by:
\begin{equation}
\tb{W}_{0}^{P}(x) = \tb{W}_{0}(x) + \tb{\delta}(x),
\end{equation}
where $\tb{W}_{0}(x)$ is the stationary solution defined by 
$$A^{*}(0) = A_{l}, \quad q^{*}= 0,$$
and
$$K(x) = K_{l} + 100e^{-10(x-2.5)^2}, \ \ A_{0}(x) = A_{0,l} + 0.0001e^{-10(x-2.5)^2}, p_{e}(x) = p_{e,l} + 100e^{-10(x-2.5)^2},$$
where $A_{l}, K_{l}, A_{0,l}, p_{e,l}$ are given in Tables \ref{tab:TestsParameters}, \ref{tab:TestInitialValues}, and
$$
\tb{\delta}(x) = [10^{-6}e^{-40(x-1)^{2}},0,0,0,0]^T.
$$
In order to compute the errors and check the order of the first-, second- and third-order fully well-balanced schemes we use 200-, 400-, 800- and 1600-point uniform meshes. The third-order fully well-balanced method was used to calculate the reference solution, utilizing a 3200-point uniform grid. Table \ref{tab:Error_TestOrder} presents the obtained results, from which we deduce that the expected order of accuracy is achieved in all three cases.

\begin{table}[!h]
  	\centering
  	\begin{tabular}{|c|c|c|c|c|}
  	    \hline
         \multicolumn{5}{|c|}{First-order} \\
 \hline
  		Number of cells & $||\Delta A||_1$ & Order & $||\Delta u||_1$ & Order \\  
  		\hline 
   100 & 3.42e-07 & - & 7.21e-04 & - \\
   \hline
   200 & 1.39e-07 & 1.30 & 5.18e-04 & 0.48 \\
   \hline
   400 & 6.80e-08 & 1.03 & 3.40e-04 & 0.61 \\
   \hline
   800 & 3.62e-08 & 0.91 & 2.05e-04 & 0.73 \\
   \hline
  1600 & 1.93e-08 & 0.91 & 1.15e-04 & 0.83 \\
  		\hline 
  		\hline
         \multicolumn{5}{|c|}{Second-order} \\
 \hline
  		Number of cells & $||\Delta A||_1$ & Order & $||\Delta u||_1$ & Order \\  
  		\hline 
   100 & 3.07e-07 & - & 4.85e-04 & - \\
   \hline
   200 & 8.68e-08 & 1.82 & 1.87e-04 & 1.38 \\
   \hline
   400 & 2.14e-08 & 2.02 & 4.47e-05 & 2.06 \\
   \hline
   800 & 5.40e-09 & 1.99 & 1.15e-05 & 1.97 \\
   \hline
  1600 & 1.36e-09 & 1.99 & 2.95e-06 & 1.96 \\
  		\hline
  		\hline
         \multicolumn{5}{|c|}{Third-order} \\
 \hline
  		Number of cells & $||\Delta A||_1$ & Order & $||\Delta u||_1$ & Order \\  
  		\hline 
   100 & 3.52e-08 & - & 2.16e-04 & - \\
   \hline
   200 & 8.76e-09 & 2.01 & 5.41e-05 & 2.00 \\
   \hline
   400 & 1.05e-09 & 3.06 & 6.47e-06 & 3.07 \\
   \hline
   800 & 1.34e-10 & 2.97 & 8.25e-07 & 2.97 \\
   \hline
  1600 & 1.53e-11 & 3.13 & 9.42e-08 & 3.13 \\
  		\hline
  	\end{tabular} 
	  	\caption{Test 6: order of accuracy for the first-, second- and third-order fully well-balanced scheme: $L^{1}$ errors $||\Delta \cdot||_1$ at time $t=0.4$.}

  	\label{tab:Error_TestOrder}
\end{table}

\subsubsection{Riemann problem}

Let us consider the following initial condition:

\begin{equation}
\tb{W}_{0}(x) =
\begin{cases}
    [A_{l}, A_{l}u_{l}, K_{l}, A_{0,l}, p_{e,l} ]^{T} & \text{if} \ \ x<x_{g}, \\
    [A_{r}, A_{r}u_{r}, K_{r}, A_{0,r}, p_{e,r} ]^{T} & \text{if} \ \  x\geq x_{g}, 
\end{cases}
\end{equation}
where the left and the right values are given in Table \ref{tab:TestInitialValues}. The values of $m$, $n$, $\rho$, $L$, $x_{g}$, $K_{\text{ref}}$, $A_{0,\text{ref}}$, $p_{e,\text{ref}}$, $t_{\text{end}}$ and $N$ are given in Table \ref{tab:TestsParameters}. We consider the constant gravity function $g(x)=g_{\text{ref}} = 9.81\,m/s^2$. In this test we will compare the numerical solutions obtained with the different methods with a reference solution computed using the third-order well-balanced scheme with a 5000-point mesh. In this case we impose boundary conditions so that the left boundary state is given by $A \equiv A_l$ and solving this ODE $q_{t} = -\frac{\gamma \pi \mu}{\rho}\frac{q}{A_{l}}+g_{\text{ref}}\,A_l$ and the right boundary state is given by $A \equiv A_r$ and solving this ODE $q_{t} = -\frac{\gamma \pi \mu}{\rho}\frac{q}{A_{r}}+g_{\text{ref}}\,A_r$.

\paragraph{Test 7: Subcritical test case} The solution of this Riemann problem consists on a left-moving rarefaction wave, a stationary contact discontinuity and a right-moving shock wave. We observe in Figures \ref{fig:Test7_WB_vs_noWB_t_AA0}, \ref{fig:Test7_WB_vs_noWB_t_u} the numerical solution obtained with all methods compared with the reference solution at time $t=0.03\,s$. As expected, the higher the order of the method, the better the rarefaction is captured. The inset of the contact discontinuities shows the differences between well-balanced and non well-balanced methods. The only ones that are able to preserve the reference solution are the well-balanced ones. In order to justify that the well-balanced methods converge to the \textit{right} solution we observe Figure \ref{fig:Test7_WB_vs_noWB_t_Gamma} where we plot the energy, i.e., the function $\Gamma$ given by \eqref{Gamma}. We see that the well-balanced methods are the only ones able to preserve the energy in the contact discontinuity.

\begin{figure}[h]
	\begin{subfigure}[h]{0.32\textwidth}
		\centering
		\includegraphics[width=1\linewidth]{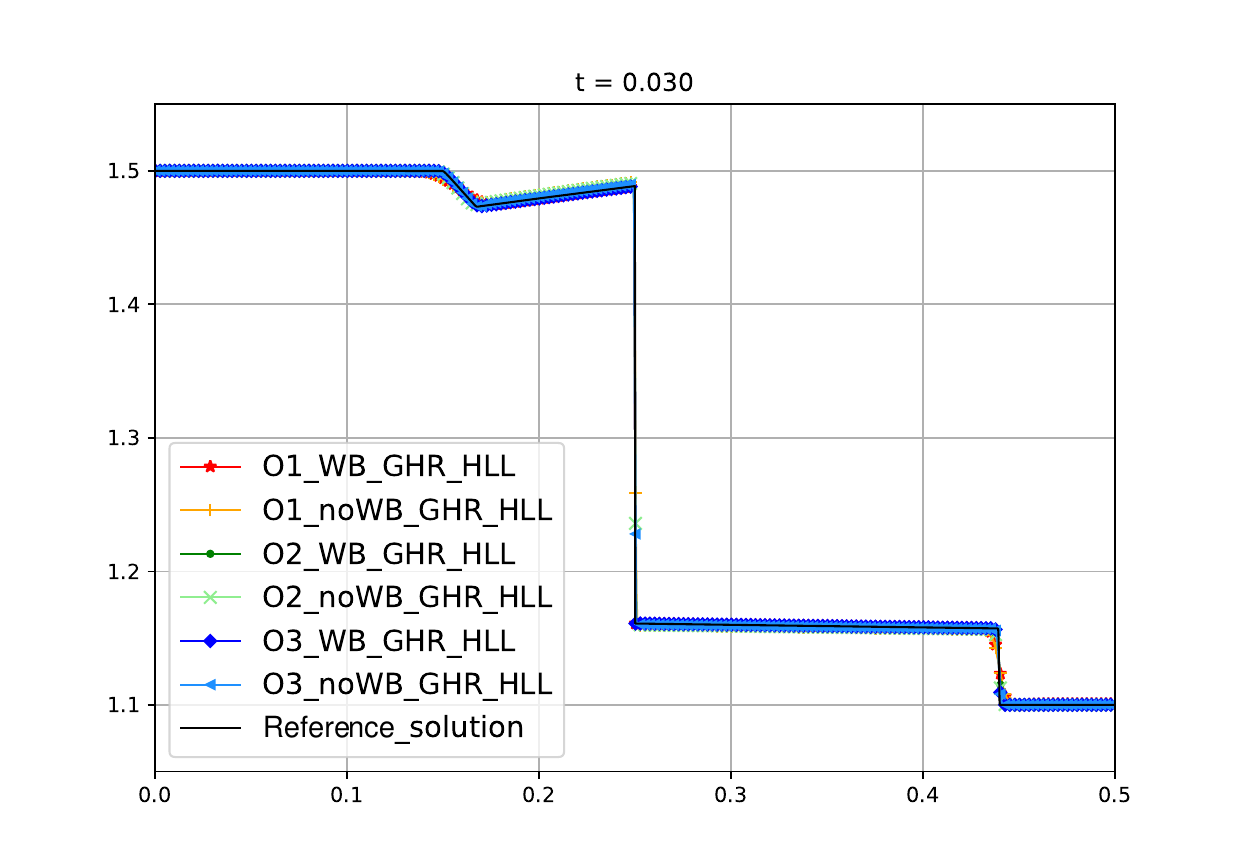}
		\caption*{Variable $A/A_{0}$}
		\label{fig:Test7_WB_vs_noWB_t_003_AA0}
	\end{subfigure}
	\begin{subfigure}[h]{0.32\textwidth}
		\centering
		\includegraphics[width=1\linewidth]{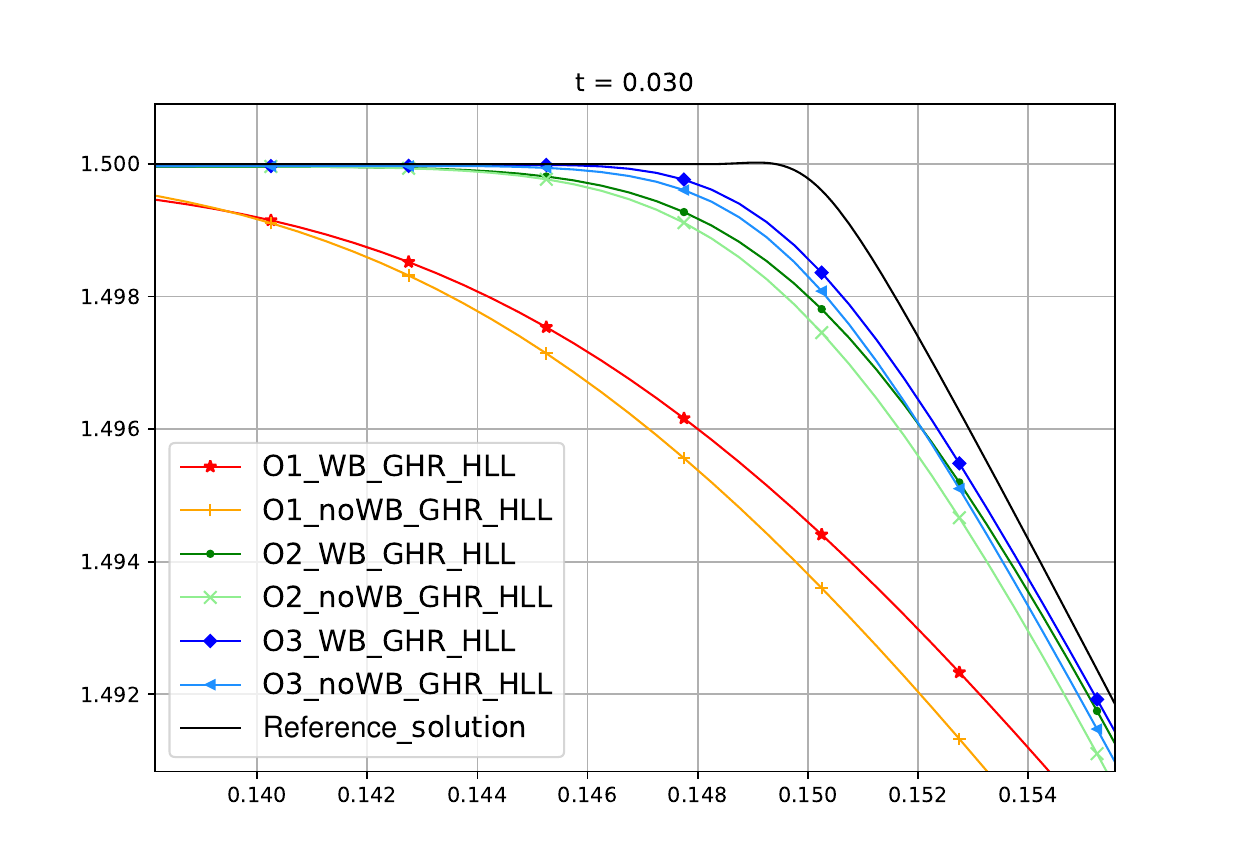}
		\caption*{Variable $A/A_{0}$. Zoom rarefaction}
		\label{fig:Test7_WB_vs_noWB_t_003_AA0_zoom_rarefaction_head}
	\end{subfigure}
	\begin{subfigure}[h]{0.32\textwidth}
		\centering
		\includegraphics[width=1\linewidth]{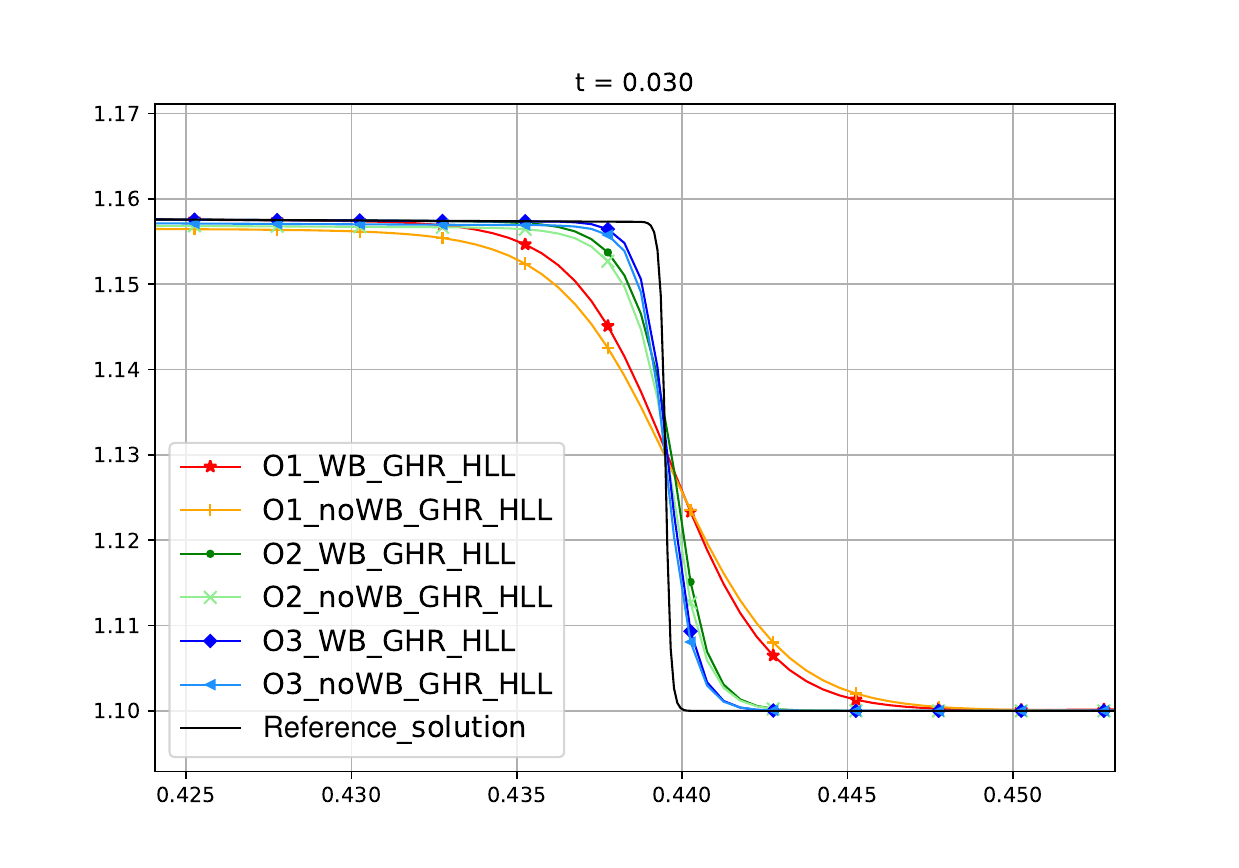}
		\caption*{Variable $A/A_{0}$. Zoom shock}
		\label{fig:Test7_WB_vs_noWB_t_003_AA0_zoom_shock}
	\end{subfigure}
    \begin{subfigure}[h]{0.49\textwidth}
		\centering
		\includegraphics[width=0.8\linewidth]{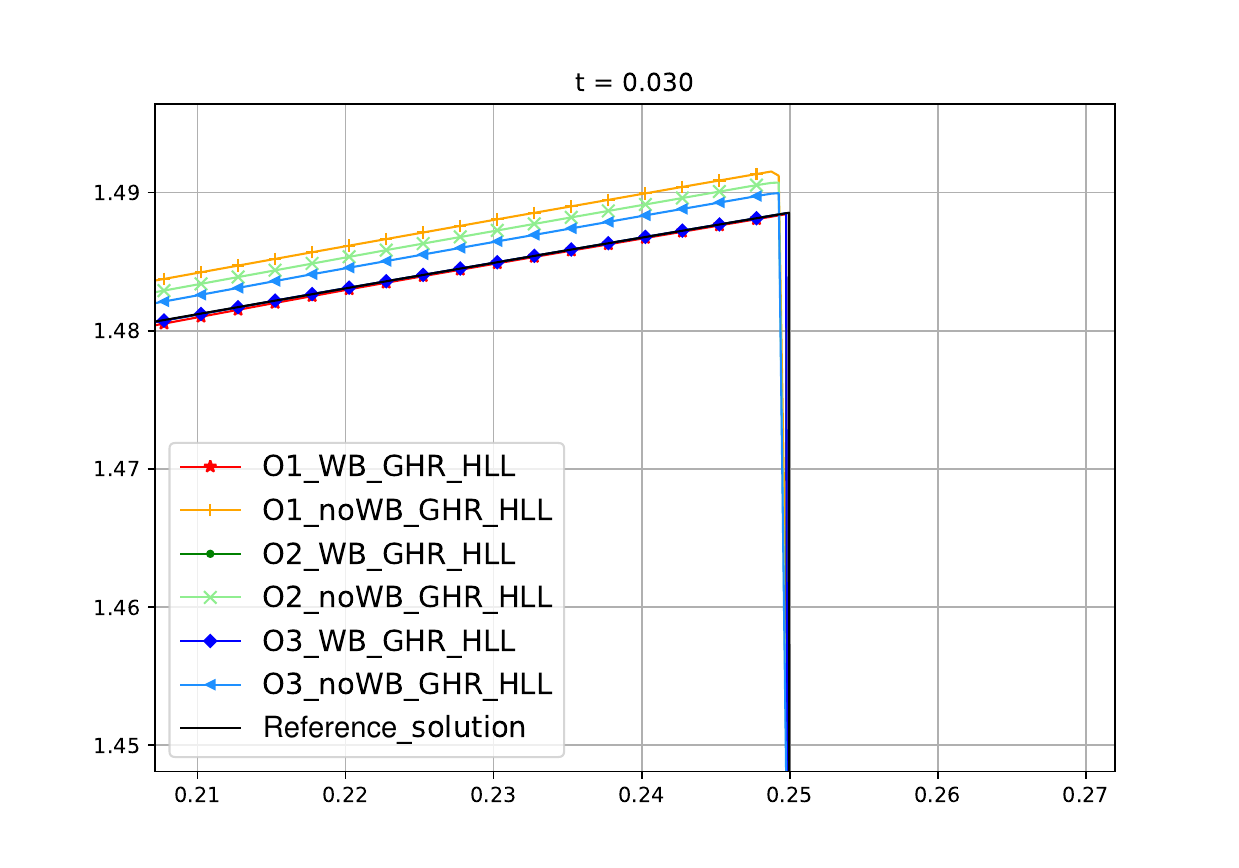}
		\caption*{Variable $A/A_{0}$. Zoom contact discontinuity: left-side}
		\label{fig:Test7_WB_vs_noWB_t_003_AA0_zoom_contact_left}
	\end{subfigure}
	\begin{subfigure}[h]{0.49\textwidth}
		\centering
		\includegraphics[width=0.8\linewidth]{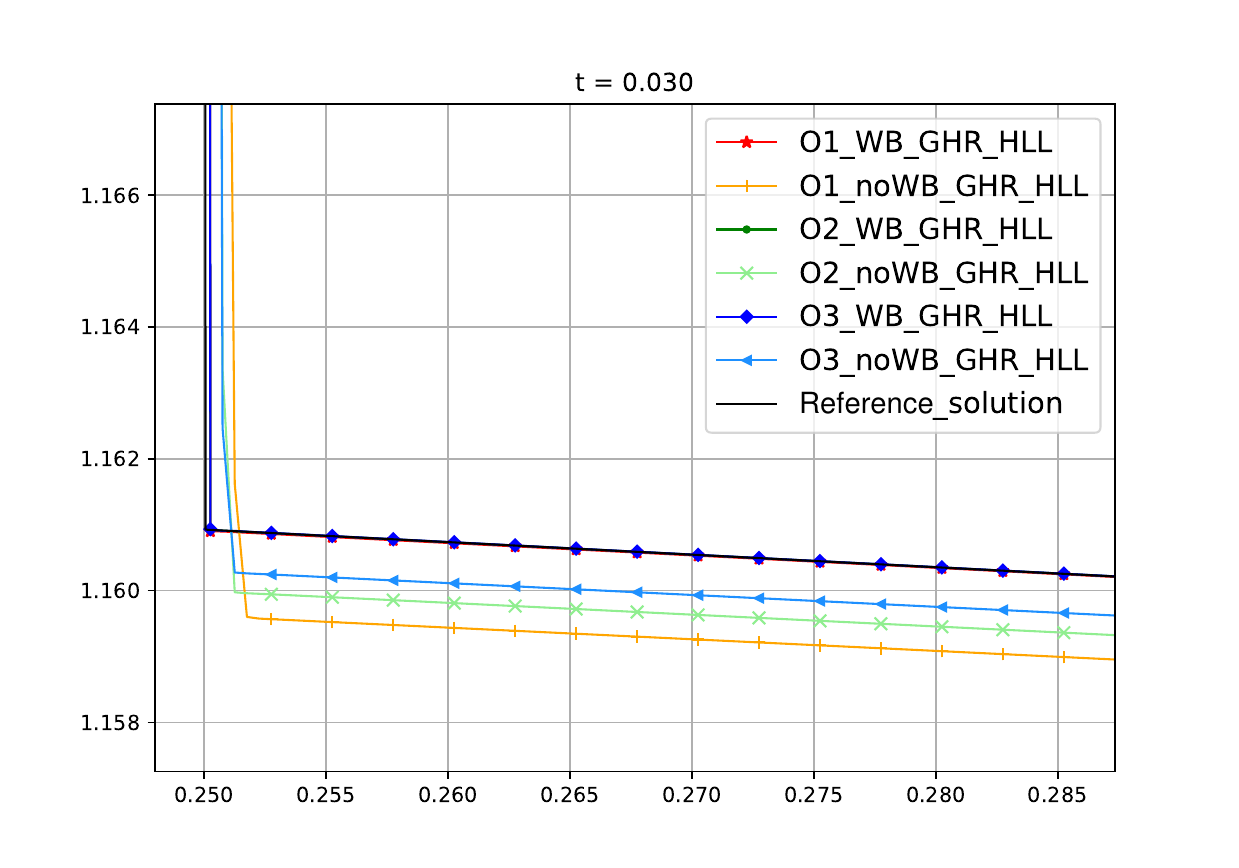}
		\caption*{Variable $A/A_{0}$. Zoom contact discontinuity: right-side}
		\label{fig:Test7_WB_vs_noWB_t_003_AA0_zoom_contact_right}
	\end{subfigure}
	\caption{Test 7: numerical solution at time $t=0.03\,s$ of first-, second-, third-order well-balanced and non well-balanced methods. Variable $A/A_{0}$. Top: complete domain (left), zoom at rarefaction (center), zoom at shock (right). Bottom: zoom left-side contact discontinuity (left), zoom right-side contact discontinuity (right).}
	\label{fig:Test7_WB_vs_noWB_t_AA0}
\end{figure}

\begin{figure}[h]
	\begin{subfigure}[h]{0.32\textwidth}
		\centering
		\includegraphics[width=1\linewidth]{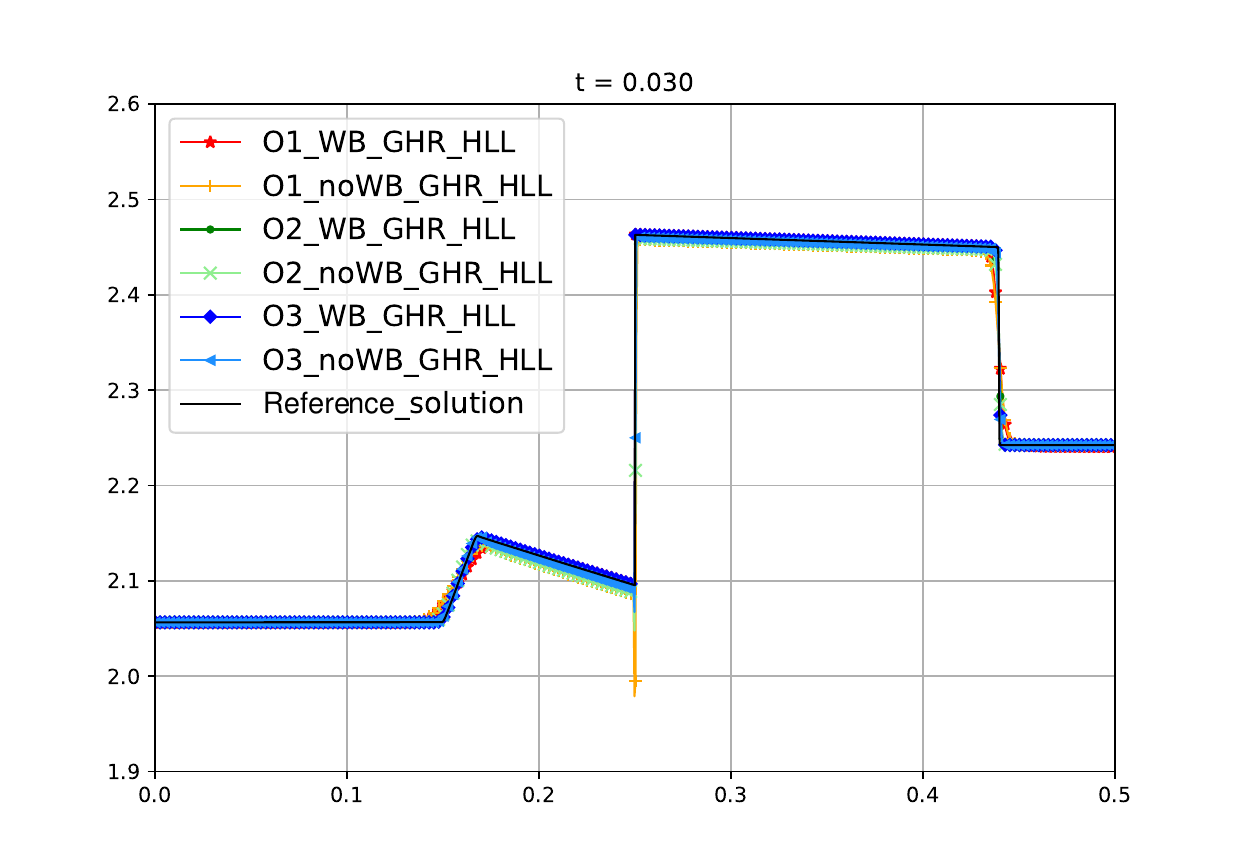}
		\caption*{Variable $u$}
		\label{fig:Test7_WB_vs_noWB_t_003_u}
	\end{subfigure}
	\begin{subfigure}[h]{0.32\textwidth}
		\centering
		\includegraphics[width=1\linewidth]{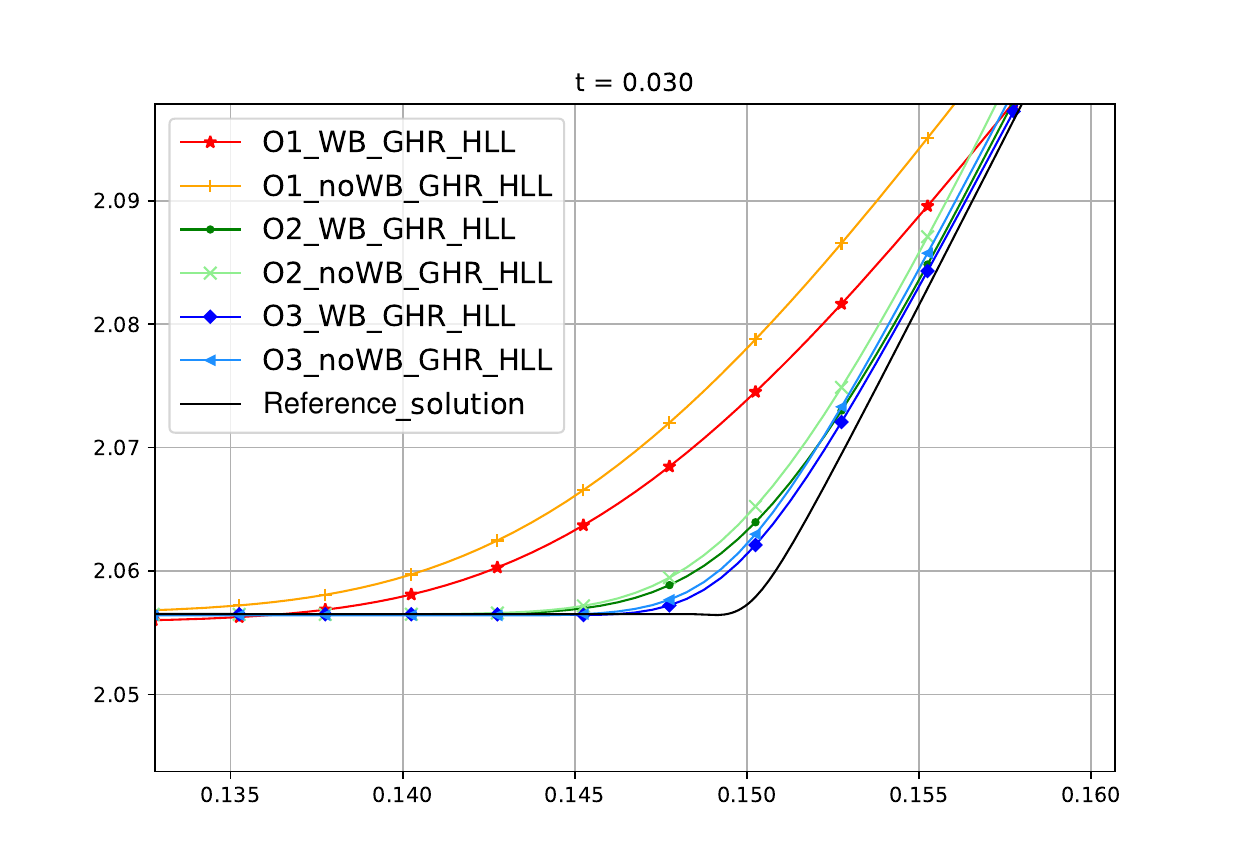}
		\caption*{Variable $u$. Inset at rarefaction}
		\label{fig:Test7_WB_vs_noWB_t_003_u_zoom_rarefaction}
	\end{subfigure}
	\begin{subfigure}[h]{0.32\textwidth}
		\centering
		\includegraphics[width=1\linewidth]{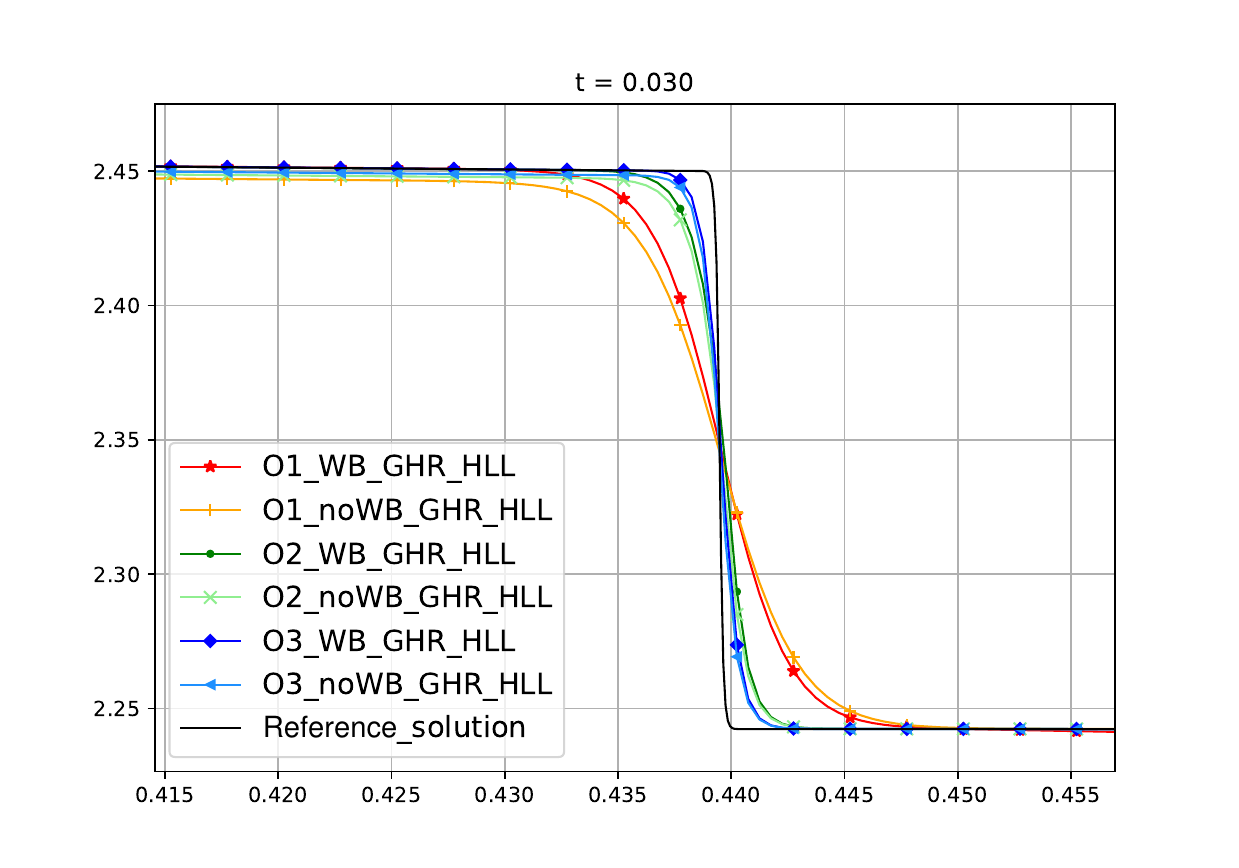}
		\caption*{Variable $u$. Inset at shock}
		\label{fig:Test7_WB_vs_noWB_t_003_u_zoom_shock}
	\end{subfigure}
    \begin{subfigure}[h]{0.49\textwidth}
		\centering
		\includegraphics[width=0.8\linewidth]{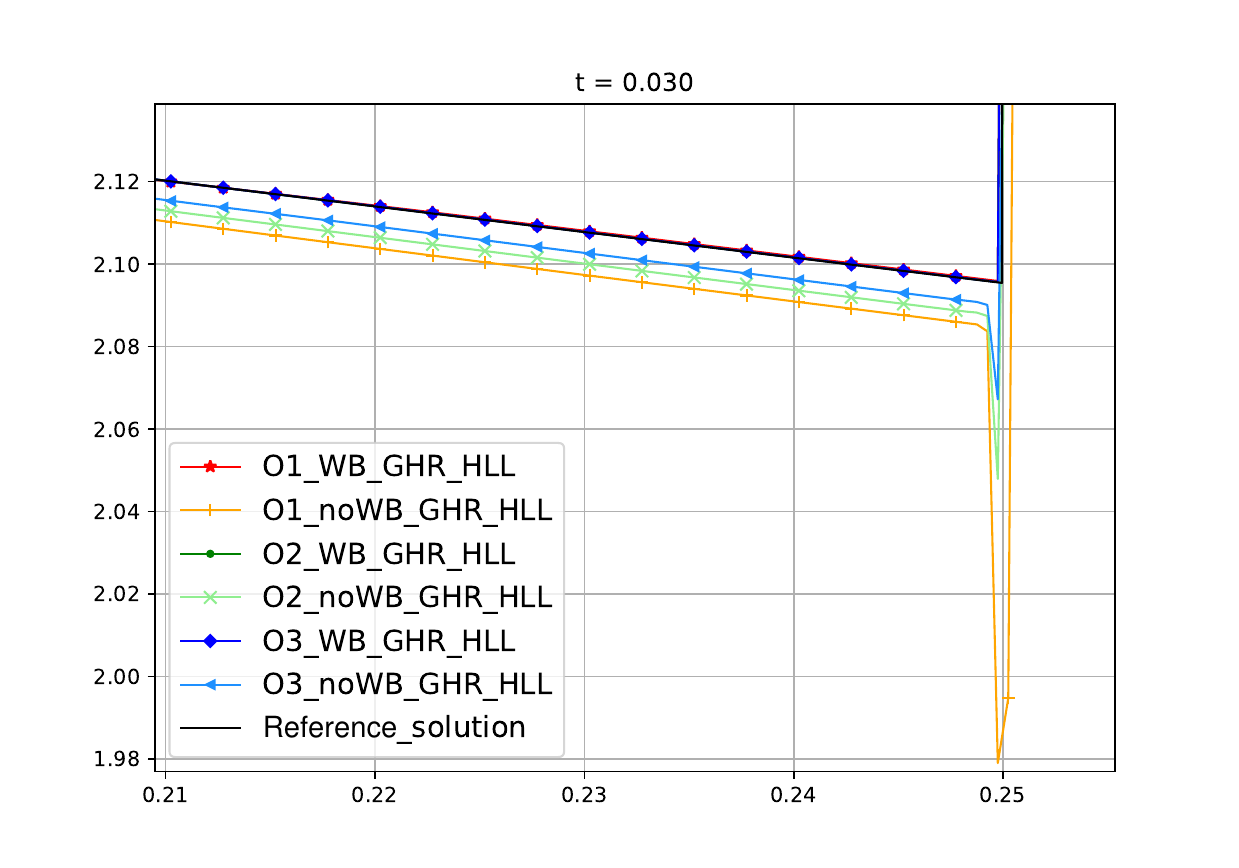}
		\caption*{Variable $u$. Inset at contact discontinuity: left-side}
		\label{fig:Test7_WB_vs_noWB_t_003_u_zoom_contact_left}
	\end{subfigure}
	\begin{subfigure}[h]{0.49\textwidth}
		\centering
		\includegraphics[width=0.8\linewidth]{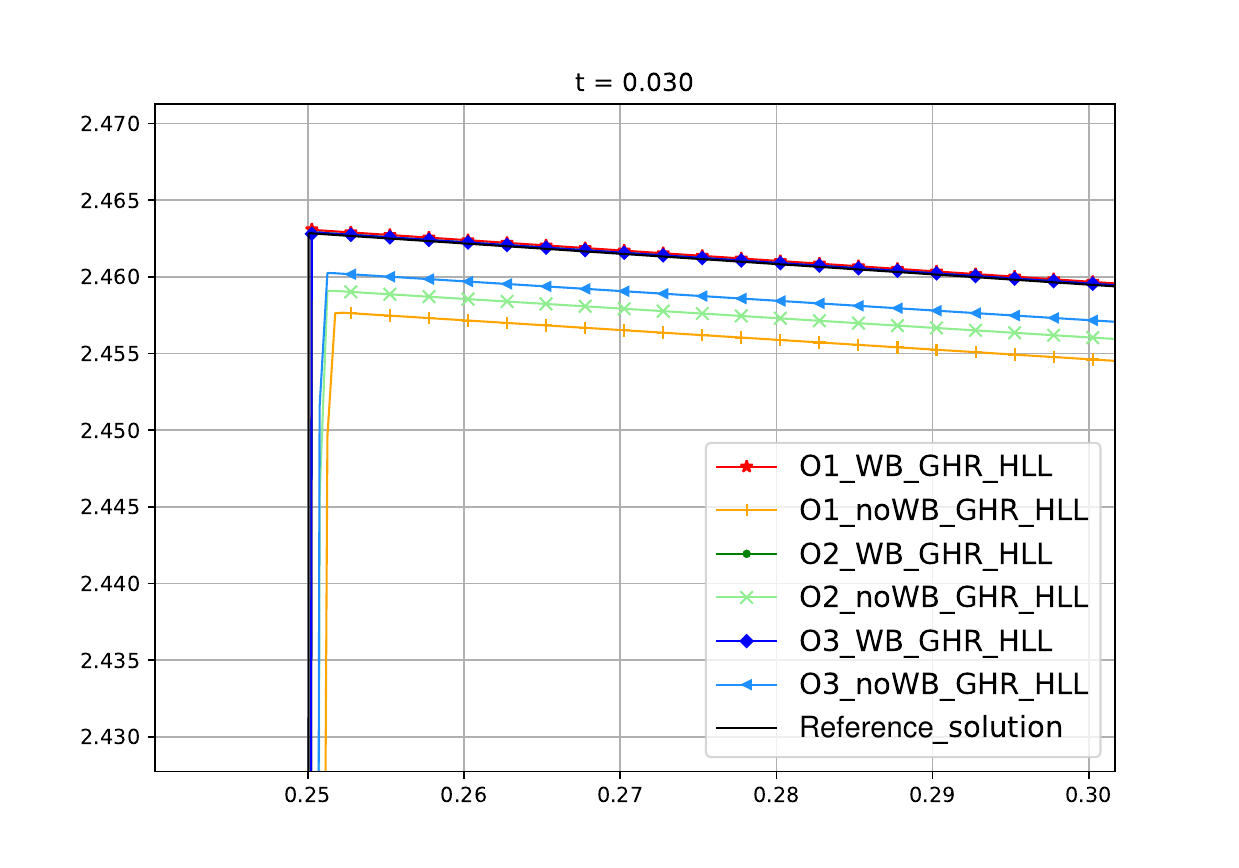}
		\caption*{Variable $u$. Inset at contact discontinuity: right-side}
		\label{fig:Test7_WB_vs_noWB_t_003_u_zoom_contact_right}
	\end{subfigure}
	\caption{Test 7: numerical solution at time $t=0.03\,s$ of first-, second-, third-order well-balanced and non well-balanced methods. Variable $u$. Top: complete domain (left), inset at rarefaction (center), inset at shock (right). Bottom: inset at left-side contact discontinuity (left), inset at right-side contact discontinuity (right).}
	\label{fig:Test7_WB_vs_noWB_t_u}
\end{figure}

\begin{figure}[h]
    \begin{subfigure}[h]{0.49\textwidth}
		\centering
		\includegraphics[width=0.8\linewidth]{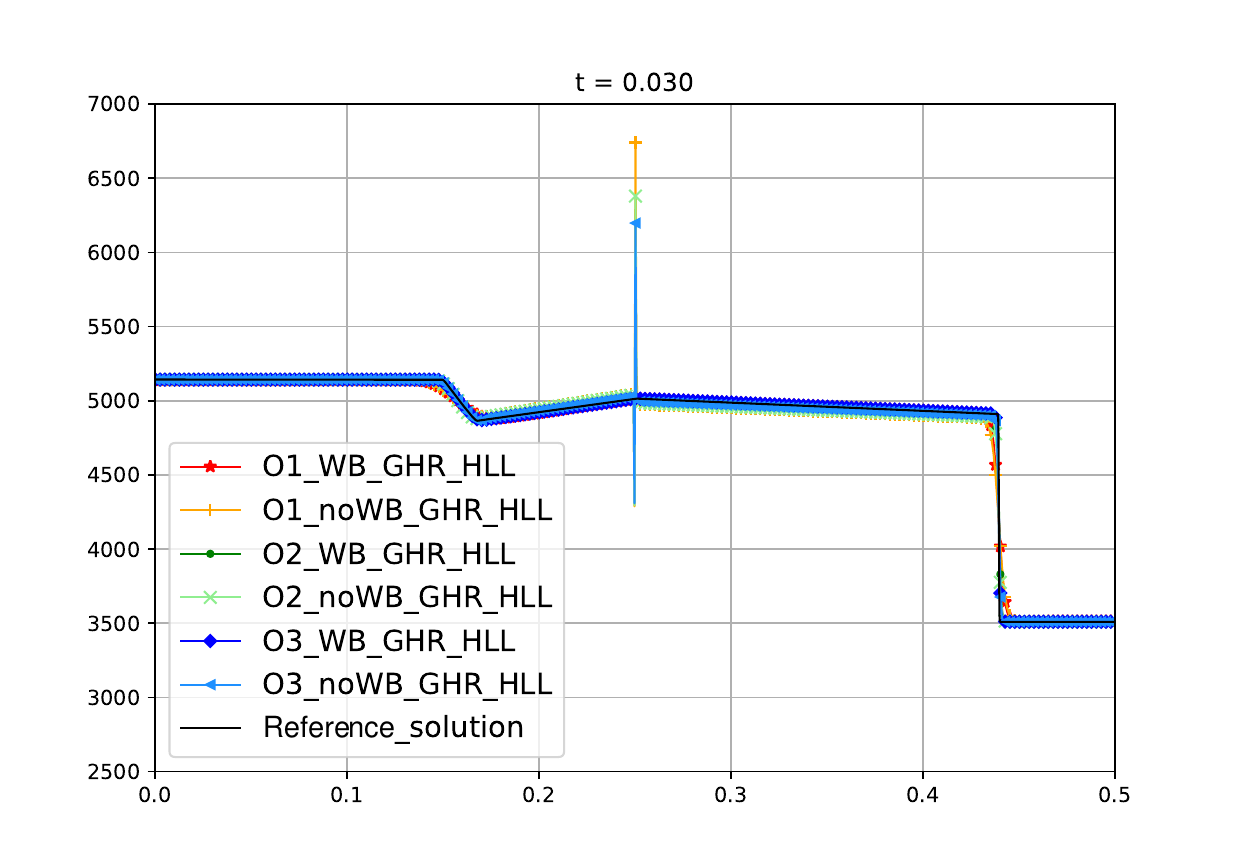}
		\caption*{Energy $\Gamma$. Complete domain}
		\label{fig:Test7_WB_vs_noWB_t_003_Gamma}
	\end{subfigure}
	\begin{subfigure}[h]{0.49\textwidth}
		\centering
		\includegraphics[width=0.8\linewidth]{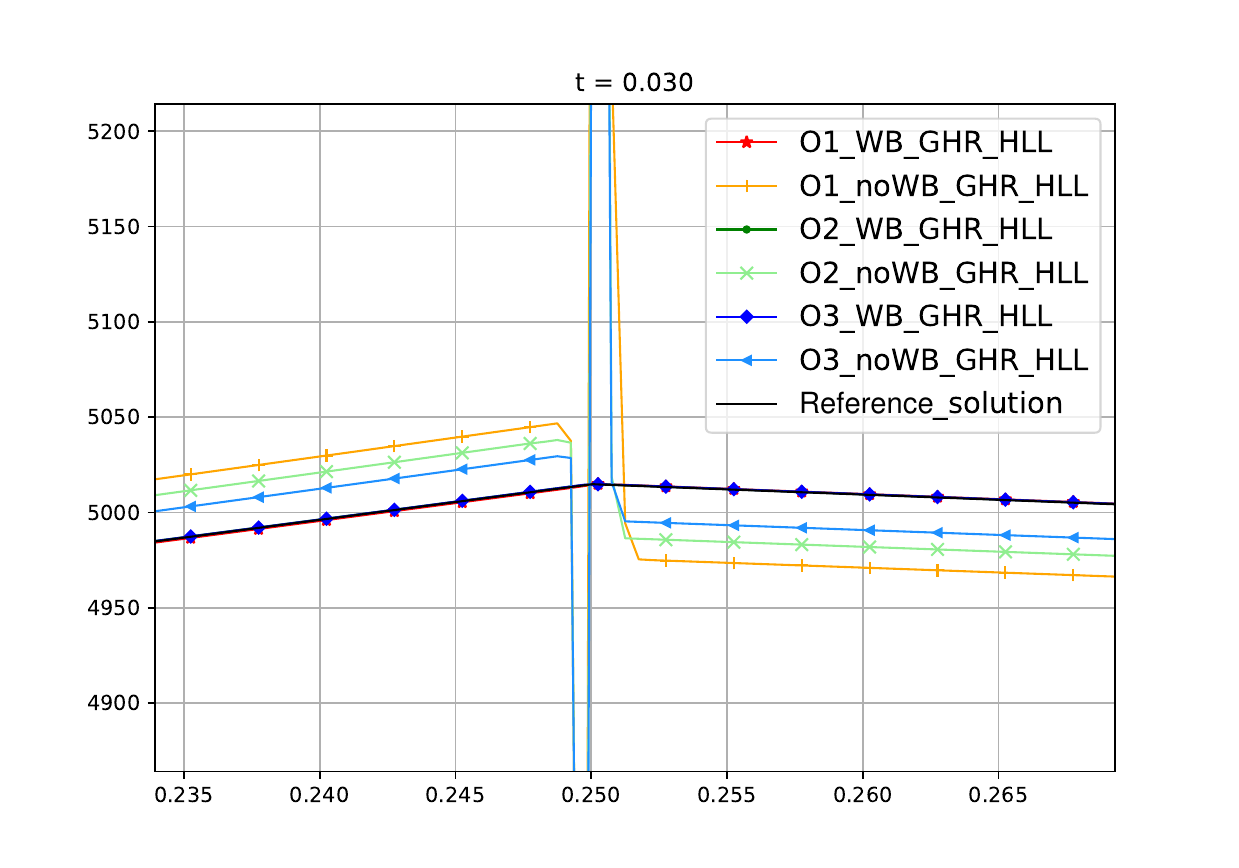}
		\caption*{Energy $\Gamma$. Zoom contact discontinuity}
		\label{fig:Test7_WB_vs_noWB_t_003_Gamma_zoom_contact}
	\end{subfigure}
	\caption{Test 7: numerical solution at time $t=0.03\,s$ of first-, second-, third-order well-balanced and non well-balanced methods. Energy $\Gamma$. Complete domain (left), contact discontinuity (right).}
	\label{fig:Test7_WB_vs_noWB_t_Gamma}
\end{figure}

{\blue 

\subsection{Networks tests}

In this section we consider two tests, one concerning a vessel with prescribed inlet flow and coupled to a lumped-parameter model, and a second one regarding a network of 118 arteries with prescribed inlet pressure at the aortic root and terminal vessels coupled to lumped-parameter models.

\subsubsection{Vessel - lumped-parameter model test}


We consider a vein with the following characteristics:
\begin{itemize}
\item vessel length: $L = 1.4\; cm$ and reference radius $R_0=0.015\; cm$,
\item tube law parameters: $m=10$, $n=-3/2$, $K = 10000\;dyn/cm^2$ and $p_0= 0\; dyn/cm^2$,
\item lumped-parameter model coefficients: $R_\mathrm{prox} = 750\; dyn \; s / cm^5$, $R_\mathrm{dist} = 4250\; dyn \;s / cm^5$, $C= 3 \times 10^{-9}\; cm^5 / dyn$, $P_\mathrm{ven} = 0$,
\item fixed inlet flow: $q_\mathrm{in} = q(0,t) = 0.0004\; mL/s$.
\end{itemize}

This problem has a steady state solution with $q(x,t) = q_\mathrm{in}$, and area that varies according to the solution of the Cauchy problem defined by \eqref{stationary_solutions_extended_A}, together with initial condition $p(L) = q_\mathrm{in} (R^\mathrm{prox} + R^\mathrm{dist})$. The simulation is initialized with $q(x,0)=0$ and approaches the steady state solution after a short transient. The computational domain is discretized using 10 cells and $CFL=0.5$. Final simulation time is $t_\mathrm{end}= 10 \;s$.

Figure \ref{fig:singlevessel} shows computational results for second order WB and NWB schemes. In can be observed how the solution obtained with the NWB scheme results in a steady state solution that violates mass conservation if one considers the 1D domain as control volume (recall that the prescribed flow at the vessel's inlet is $0.0004\;mL/s$). On the other hand, the WB scheme delivers the correct solution for flow rate, as well as a very good approximation of the pressure. Interestingly, errors in the NWB scheme for flow rate result in errors in the approximation of pressure, since the smaller flow rate computed with this scheme results in reduced viscous dissipation and consequently smaller pressure values with respect to the reference solution.

\begin{figure}
\begin{center}
\includegraphics[width=0.35\textwidth]{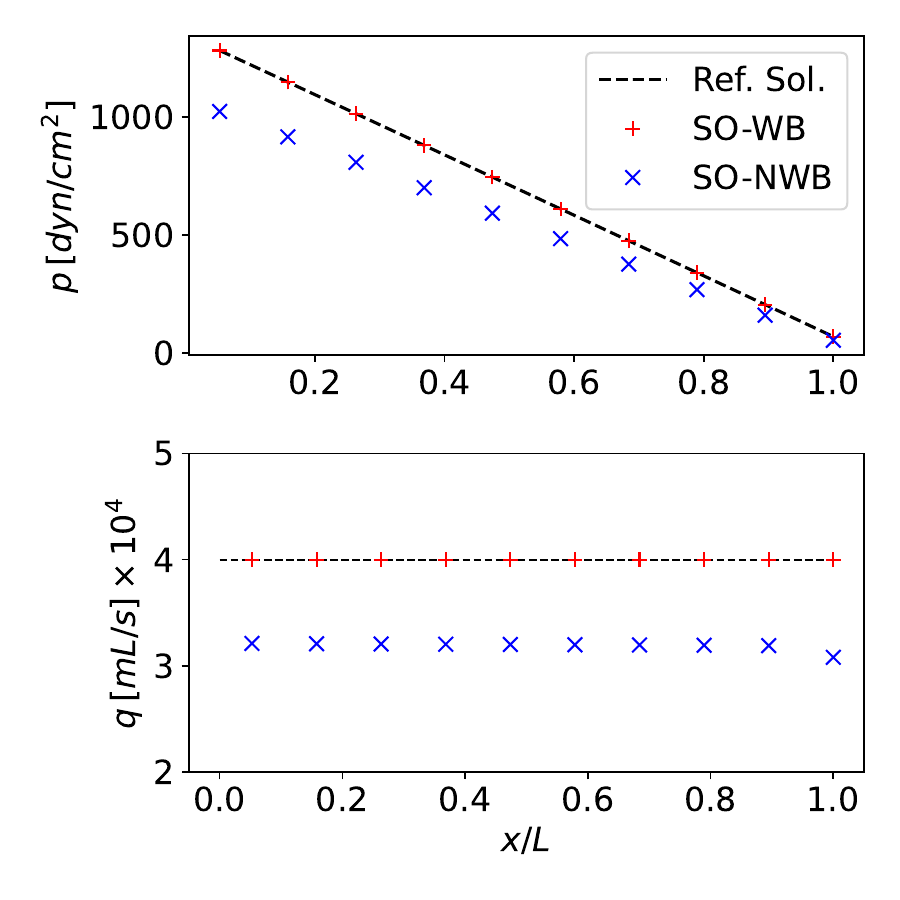}
\caption{Single vessel test with fixed inflow and terminal RCR model. Steady state for pressure and flow rate, with reference solution for pressure and exact (trivial) solution for flow rate. Reference solution obtained with second order WB scheme and 1000 computational cells.} \label{fig:singlevessel}
\end{center}
\end{figure}

\subsubsection{ADAN86 vessel network}

ADAN86 is a reduced version of the anatomically detailed arterial network model ADAN \cite{blancoAnatomicallyDetailedArterial2015}, published in \cite{safaeiRoadmapCardiovascularCirculation2016}. The model consists of 86 named arteries of the systemic circulation. Since some of these arteries are split in order to represent departing branches, the total number of 1D domains is 118. Complete information regarding vessel connectivity, length and reference radio $R_0$, as well as coefficients of terminal $RCR$ lumped parameter models coupled to terminal vessels are provided in Table \ref{tab:adannetwork} of the Appendix. The network configuration can be seen in Figure \ref{fig:adan86-overview}, panel (a). Tube law parameters are $m=1/2$, $n=0$, $P_0=10^5\;dyn/cm^2$ and $K$ is computed as:
$$
K = \frac{E\,h_0}{(1 - \nu^2) R_0}\;,
$$
with $E= 2 \times 10^6\, dyn/cm^2$ . The wall thickness $h_0$ computed as a function of the vessel reference radius
$$
 \frac{h_0}{R_0} = a \exp^{b R_0} + c \exp^{d R_0}\,,
$$
with $a = 0.2802$, $b = -5.053\,cm^{-1}$, $c = 0.1324$ and $d = -0.1114\,cm^{-1}$, according to \cite{mullerAnatomicallyDetailedArterialvenous2023}. 

A special feature of ADAN86 is that the spatial location of the points belonging to the polyline that defines each vessel is known. This allows for a very detailed characterization of the gravitational field impact on vessels, since we can project gravity on each one of these points according to the tangent to the polyline.  Given the set of points by which a polyline is defined, for each computational cell we compute $g(x)$ by evaluating $g(x)$ along the polyline at the collocation nodes used by the numerical scheme. As a result, the gravity term $g(x)$ can vary significantly along vessels. In Figure \ref{fig:adan86-overview}, panels (b) and (c) show the course of the left vertebral artery and the term $g(x)$ along this vessel, as an example of how variable $g(x)$ can be. 

\begin{figure}[htbp]
  \centering

  \begin{minipage}[c]{0.48\textwidth}
  \begin{center}
    \includegraphics[width=\linewidth,height=0.9\textheight,keepaspectratio]{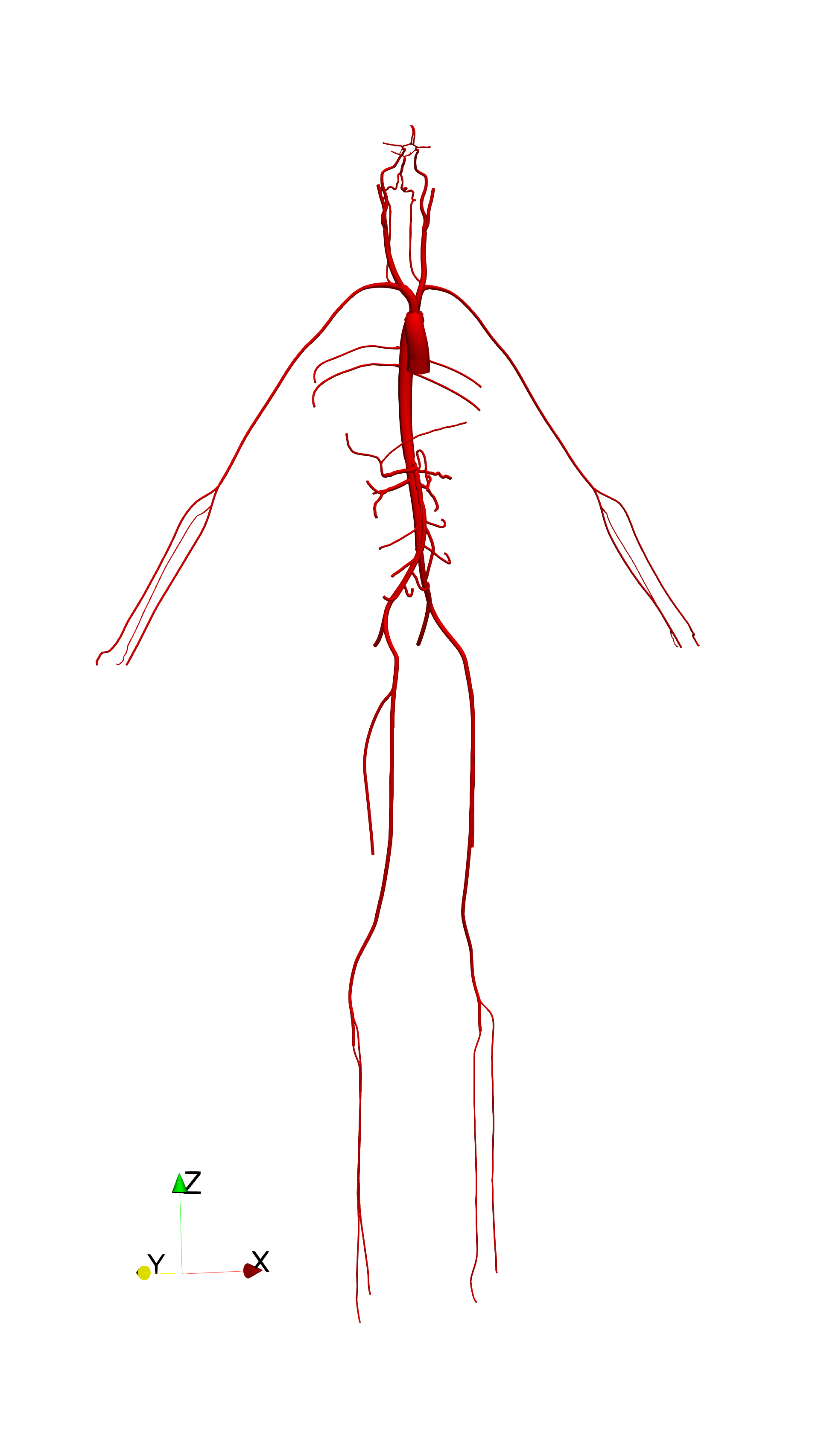}
    \end{center}
    \small\textbf{(a)}
    
  \end{minipage}%
  \hfill
  \begin{minipage}[c]{0.48\textwidth}
    \vspace{0pt}
    \begin{minipage}[t]{\linewidth}
    \begin{center}
      \includegraphics[width=\linewidth,height=0.25\textheight,keepaspectratio]{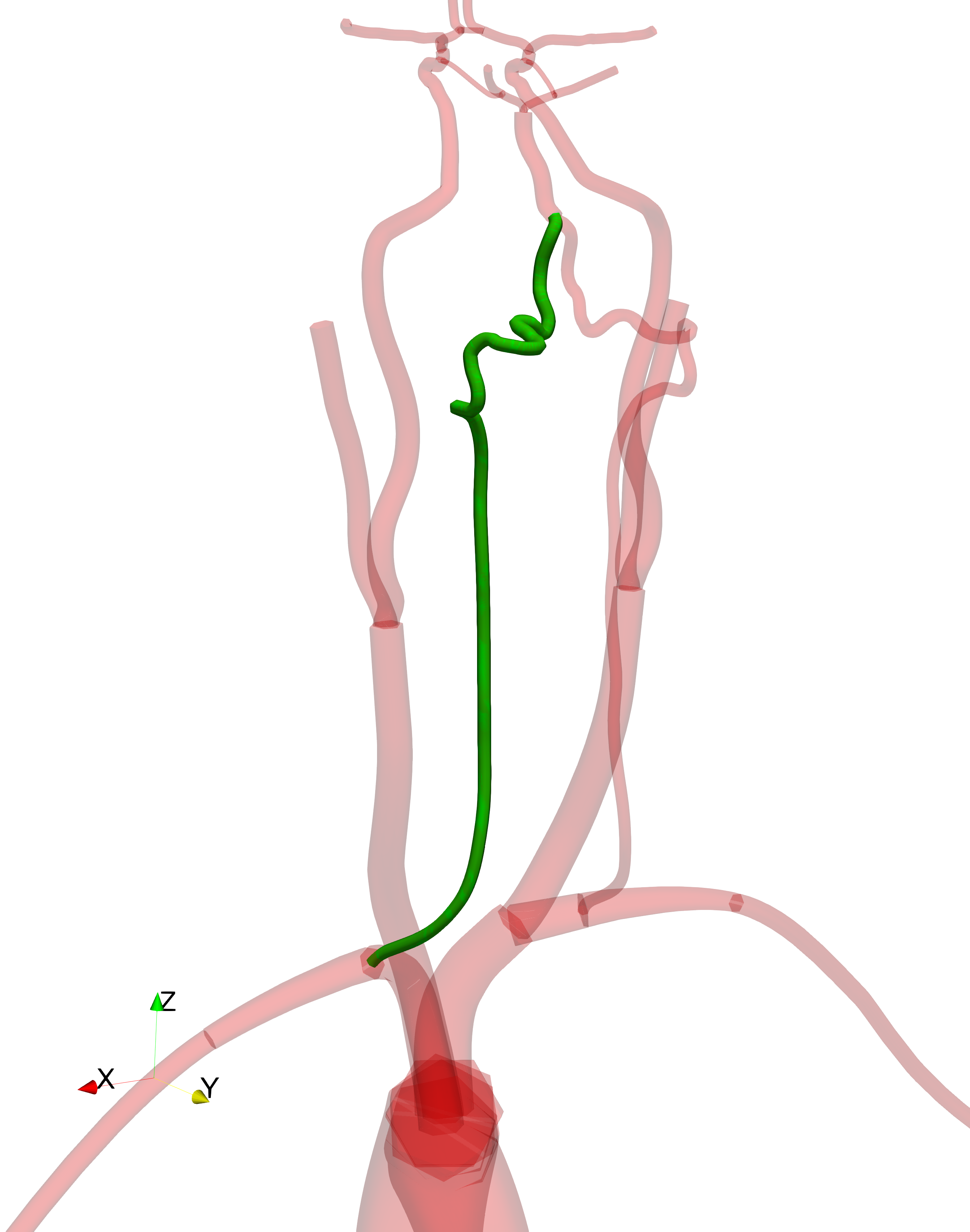}
      \end{center}
      \small\textbf{(b)}

    \end{minipage}

    \vspace{1em}

    \begin{minipage}[t]{\linewidth}
    \begin{center}
      \includegraphics[width=\linewidth,height=0.4\textheight,keepaspectratio]{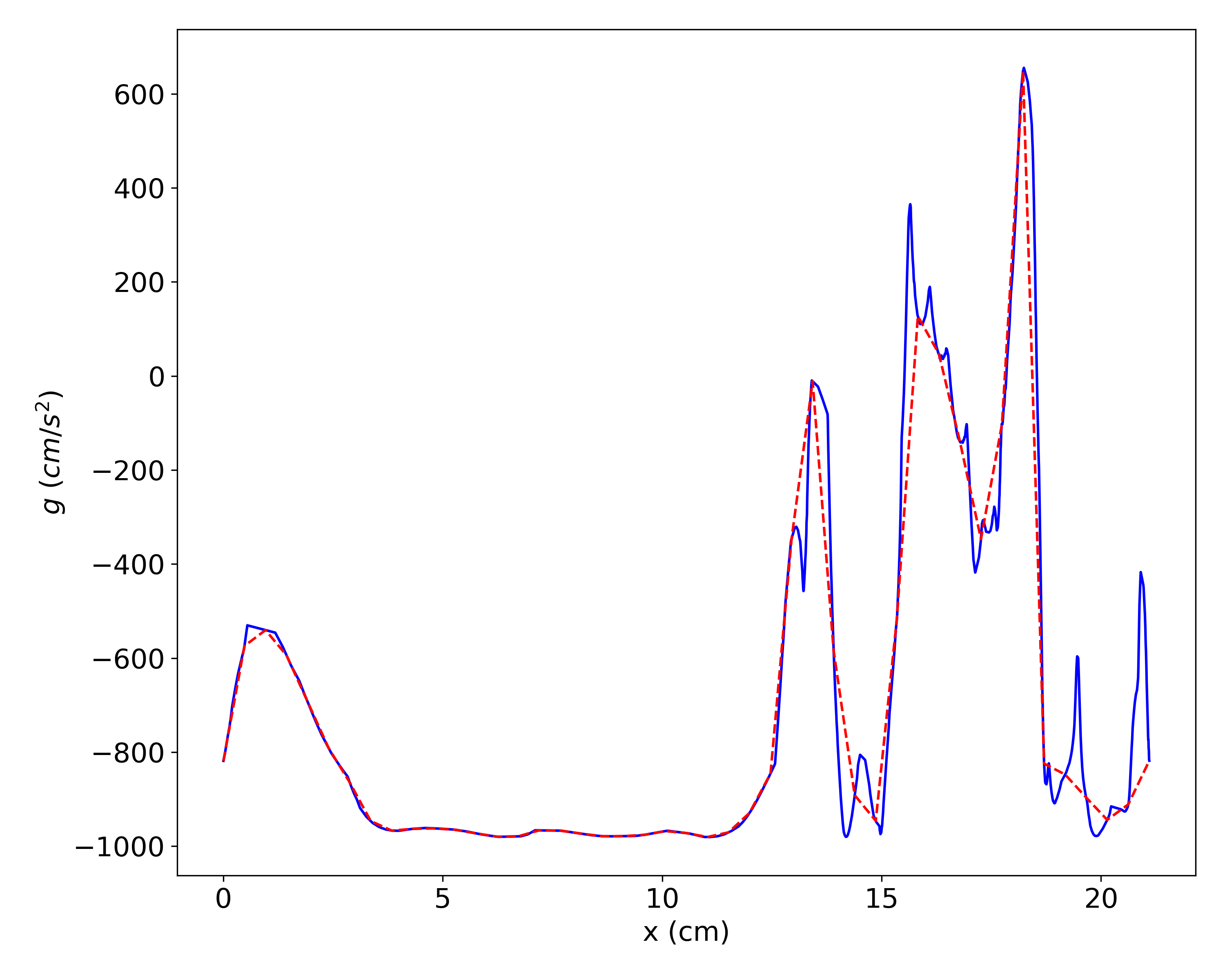}
      \end{center}
      \small\textbf{(c)}
    \end{minipage}
  \end{minipage}
    \caption{ADAN86 model. Network overview (a), posterior view of the left vertebral artery, evidenced in green (b), and $g(x)$ for this vessel (c). The gravitational field acts in opposite direction to versor $z$ displayed in the panels. In panel (c), the blue line represents true $g$ while the dashed line regards the discretization of term $g(x)$ for second order schemes used here, which results in continuous piecewise first order polynomials.}

\label{fig:adan86-overview}
\end{figure}

\paragraph{Deadman test}

Introducing gravity into cardiovascular models is an essential ingredient to properly describe the impact of orthostatic stress on haemodynamics \cite{Howard:1977}. In doing so, a straightforward way to verify the correct implementation of the mathematical description and numerical treatment of gravity is to perform tests for which a reference solution is available. In this context, a natural test to be performed is that of reproducing \emph{hydrostatic} pressure solutions for no flow conditions. One way of achieving this is by prescribing a constant pressure at the root of the aorta and imposing zero-flow conditions at terminal points of the network. In this way, the steady state solution is that of zero flow rate over the entire network and hydrostatic pressure, with a fixed and known value at the root of the aorta.

Figure \ref{fig:deadman} displays computational results for second order versions of both, WB and NWB schemes, using a computational mesh with characteristic length $\Delta x_\mathrm{char}=1\,cm$ and prescribed pressure at the root of the aorta $P_\mathrm{root}=10^5\,dyn/cm^2$. The top row shows the computed pressure field, the middle row regards deviations of computed pressure field with respect to hydrostatic pressure, and the bottom row shows flow rate over the network. Qualitatively, results obtained with both schemes seem to reasonably reproduce the hydrostatic pressure field, with minimum values for pressure in the cerebral region and maximal ones in the lower limbs. However, a closer inspection in which errors of the computed pressure field with respect to the hydrostatic one are shown, reveals that while errors are very small for the solution obtained with the WB scheme, even for this coarse mesh, the same is not true for the NWB scheme. In fact, this scheme produces deviations from the reference solution, which turn to be even more evident if one inspects flow rate over the network. Computed solutions for this variable evidence the poor approximation obtained with the NWB scheme. On the other hand, the WB scheme produces a solution where flow is orders of magnitude smaller than that obtained with the NWB scheme. Noteworthy, the anomalous flow field observed in the solution computed with the NWB scheme violates mass conservation. This is a very delicate aspect, especially if one would then move towards closed-loop models of the circulation, where mass conservation is of paramount importance due to the high sensitivity of variables of interest such as cardiac output or mean arterial pressure to total blood volume \cite{mullerAnatomicallyDetailedArterialvenous2023}. 

In order to provide a quantitative comparison of solutions obtained with the WB and NWB second order schemes, Figure \ref{fig:deadmanxy} shows computed pressure and flow rate over space for selected vessels. This figure does not show the actual pressure field. In turn, it shows pressure in each vessel shifted by the inlet pressure value: $p_\mathrm{fig}=p - p(0, t_\mathrm{end})$. The same is true for depicted \emph{Hyd. Pressure}, which is the hydrostatic pressure computed from the vessel inlet. In this way one can straightforwardly compare the agreement of the expected course of the pressure field with respect to the reference solution. These plots evidence that even for a rather coarse mesh, the WB scheme can accurately capture the effect of a, some times, highly variable gravitational term. Furthermore, we note that the predicted flow rate with the WB schemes is always orders of magnitude closer to the exact solution compared to that computed with the NWB scheme.

\begin{figure}
\begin{minipage}[c]{0.9\textwidth}
\begin{center}
\includegraphics[width=0.7\textwidth]{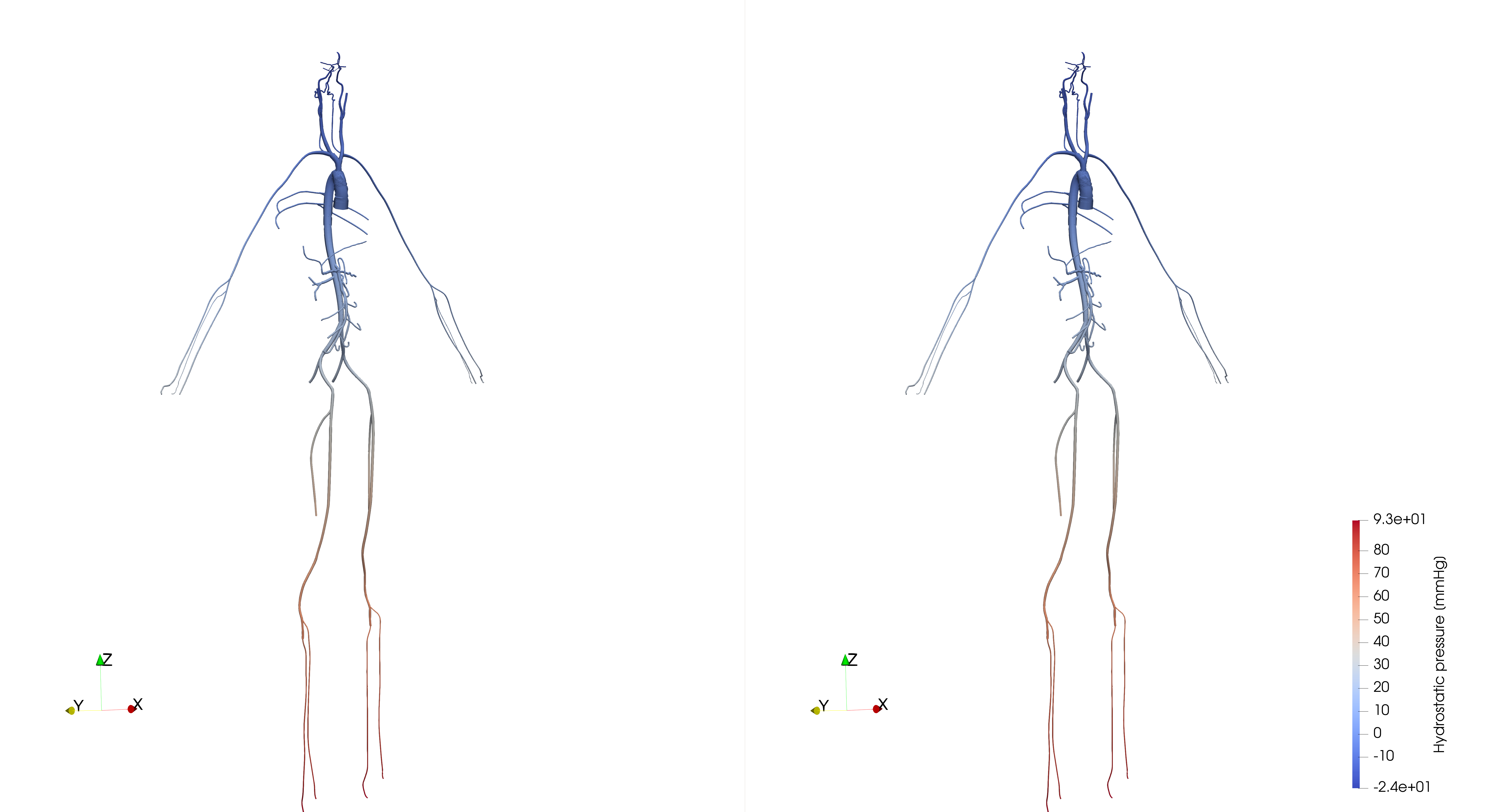} \\
\end{center}
\small\textbf{(a)}
\end{minipage}
\begin{minipage}[c]{0.9\textwidth}
\begin{center}
\includegraphics[width=0.65\textwidth]{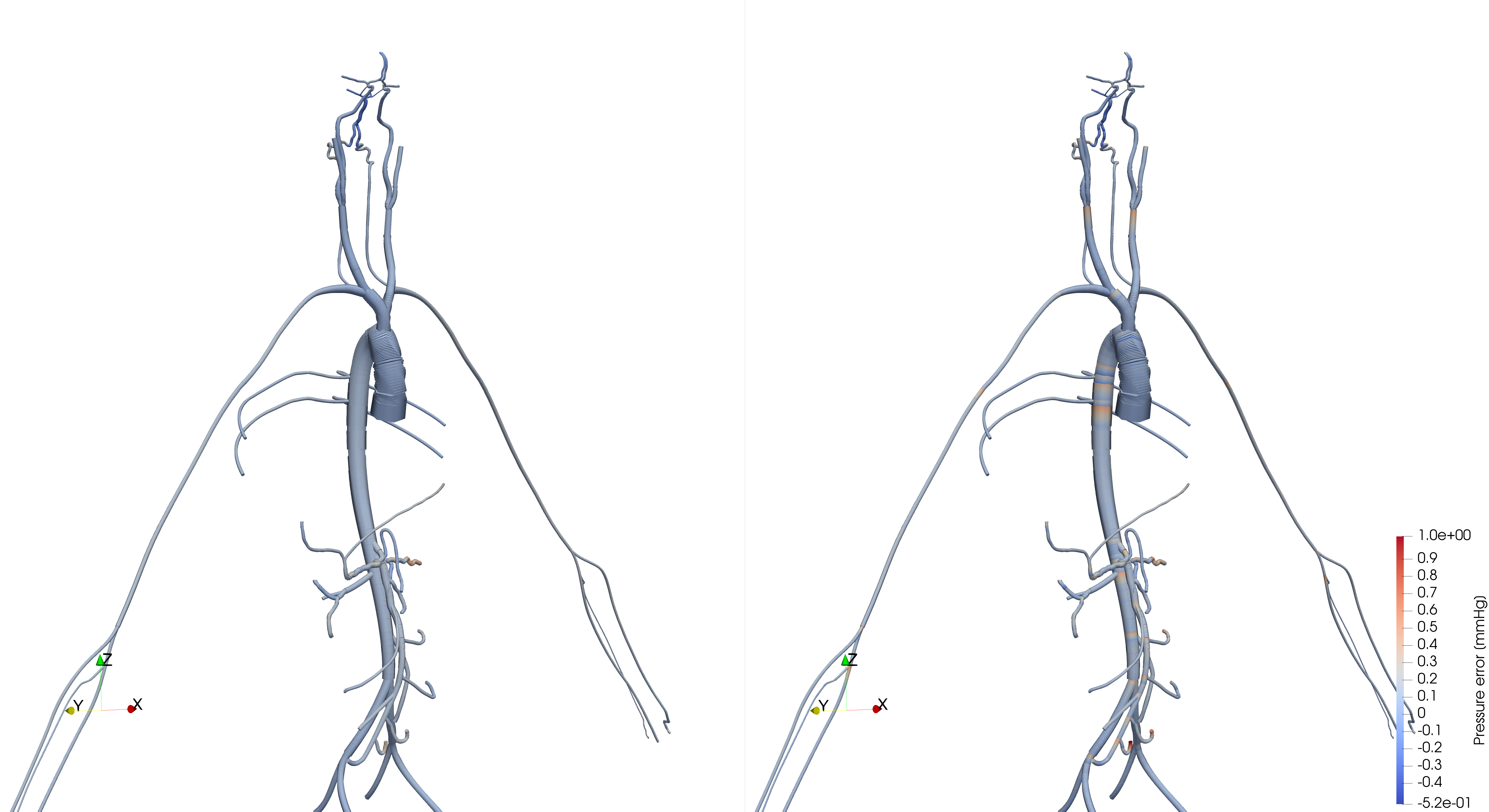} \\
\end{center}
\small\textbf{(b)}
\end{minipage}
\begin{minipage}[c]{0.9\textwidth}
\begin{center}
\includegraphics[width=0.65\textwidth]{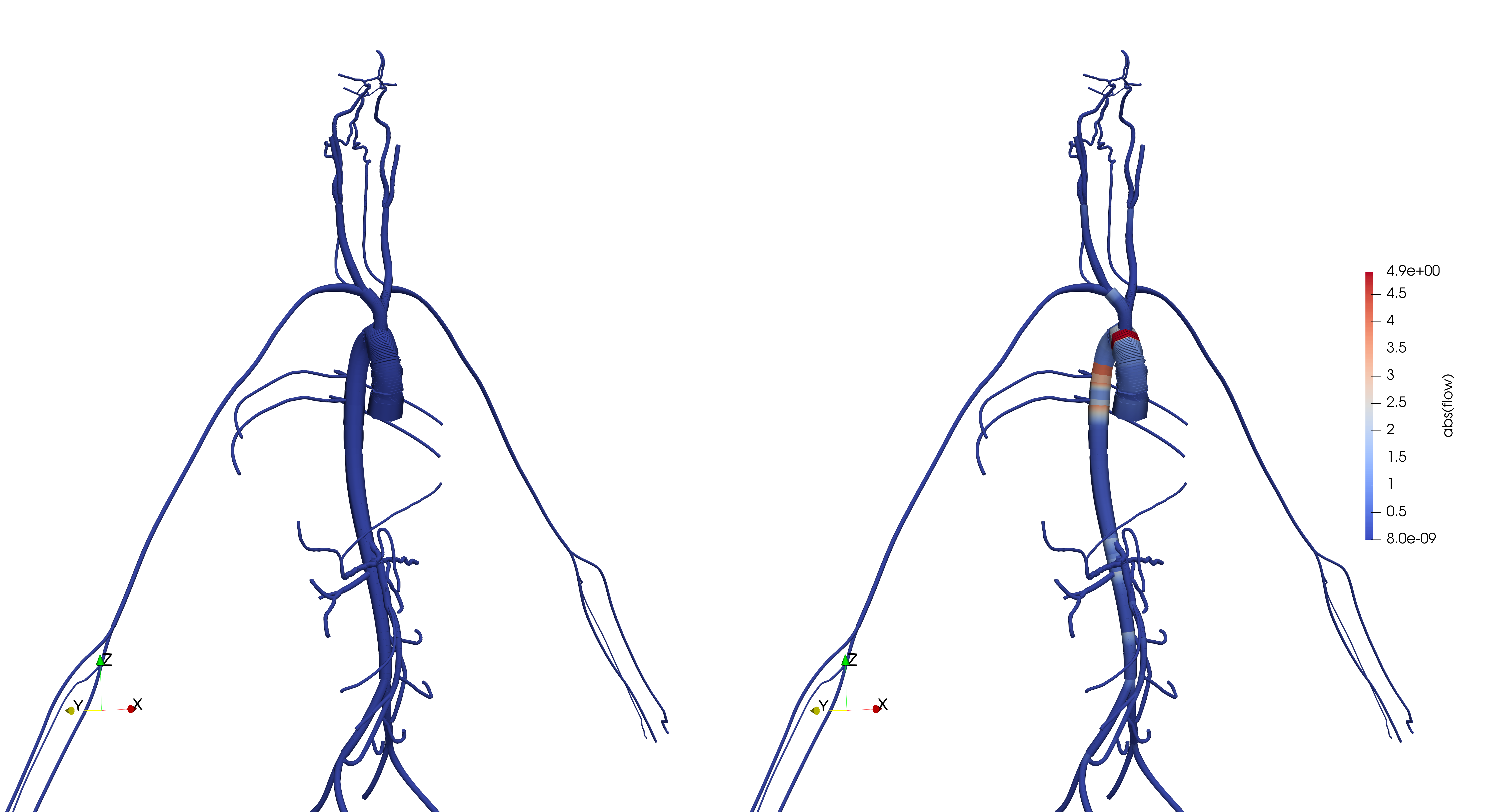} \\
\end{center}
\small\textbf{(c)}
\end{minipage}
\caption{ADAN86 model. Deadman test with gravity. Pressure field (a), error in pressure field with respect to hydrostatic pressure (b) and absolute value of flow rate (c). Left and right columns show results obtained with the second order WB and NWB schemes, respectively.}\label{fig:deadman}
\end{figure}

\begin{figure}
\begin{minipage}[c]{0.45\textwidth}
\begin{center}
\includegraphics[width=0.9\textwidth]{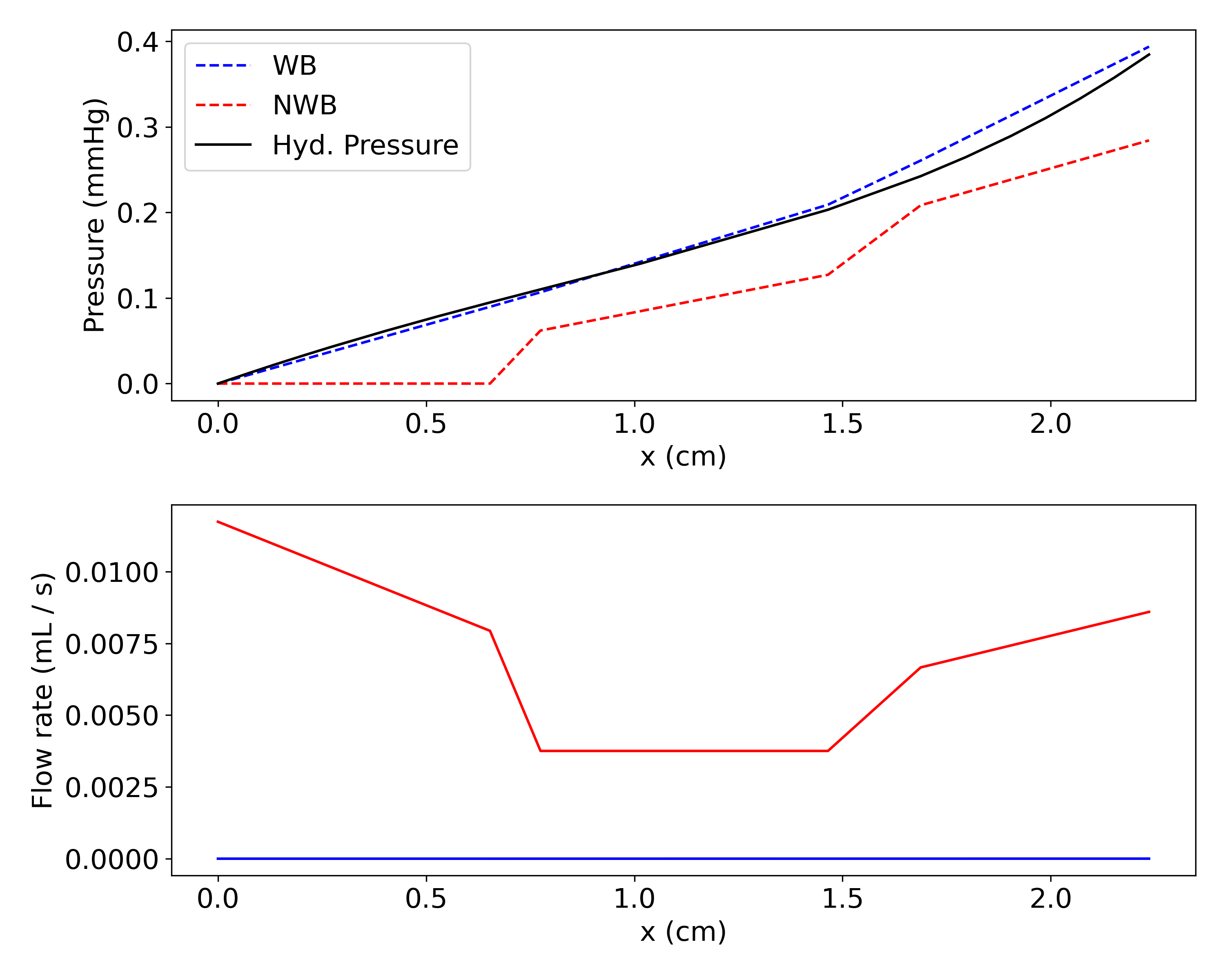}
\end{center}
\small\textbf{(a) Renal artery (7)}
\end{minipage}
\begin{minipage}[c]{0.45\textwidth}
\begin{center}
\includegraphics[width=0.9\textwidth]{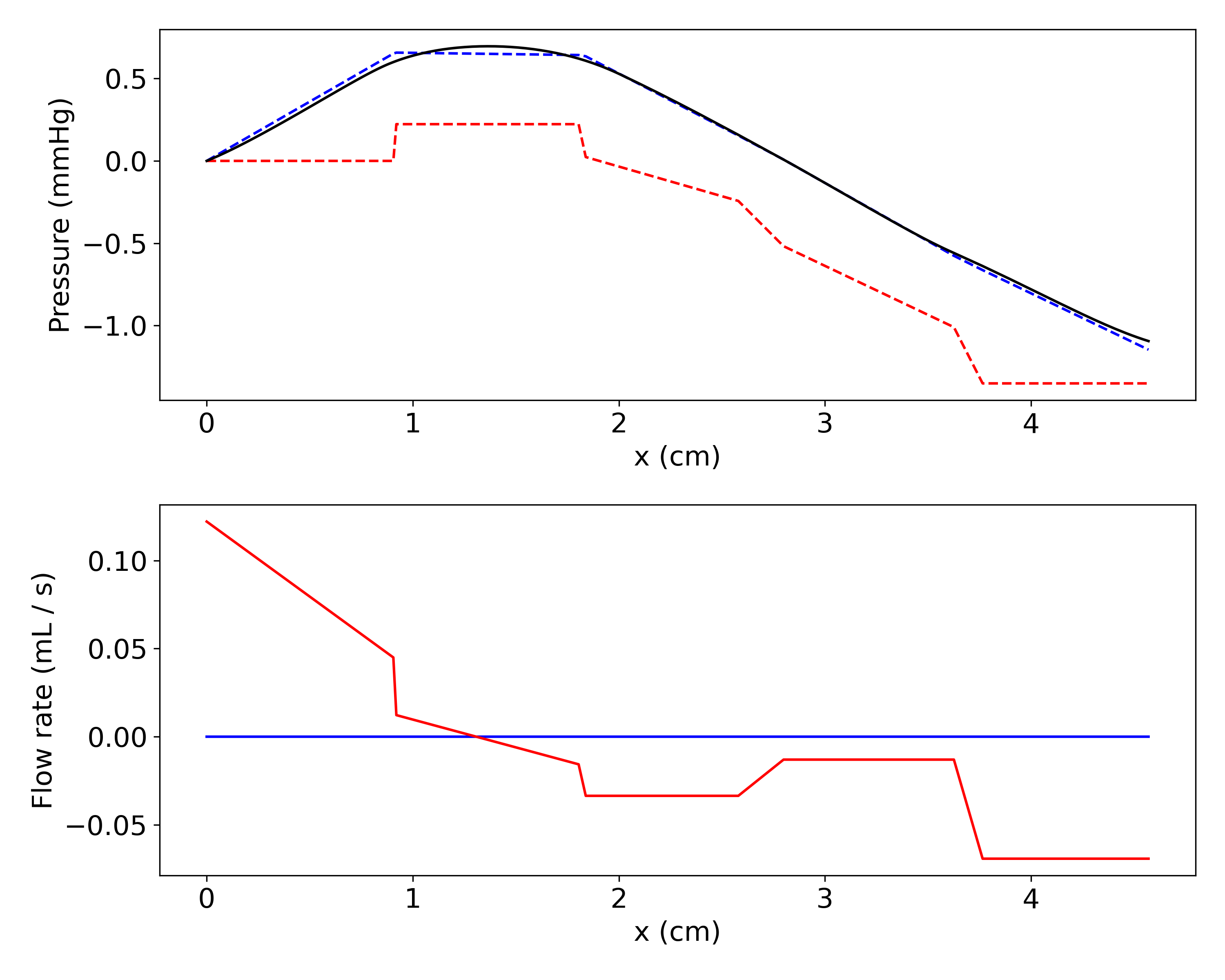}
\end{center}
\small\textbf{(b) Ileal artery (14)}
\end{minipage}
\begin{minipage}[c]{0.45\textwidth}
\begin{center}
\includegraphics[width=0.9\textwidth]{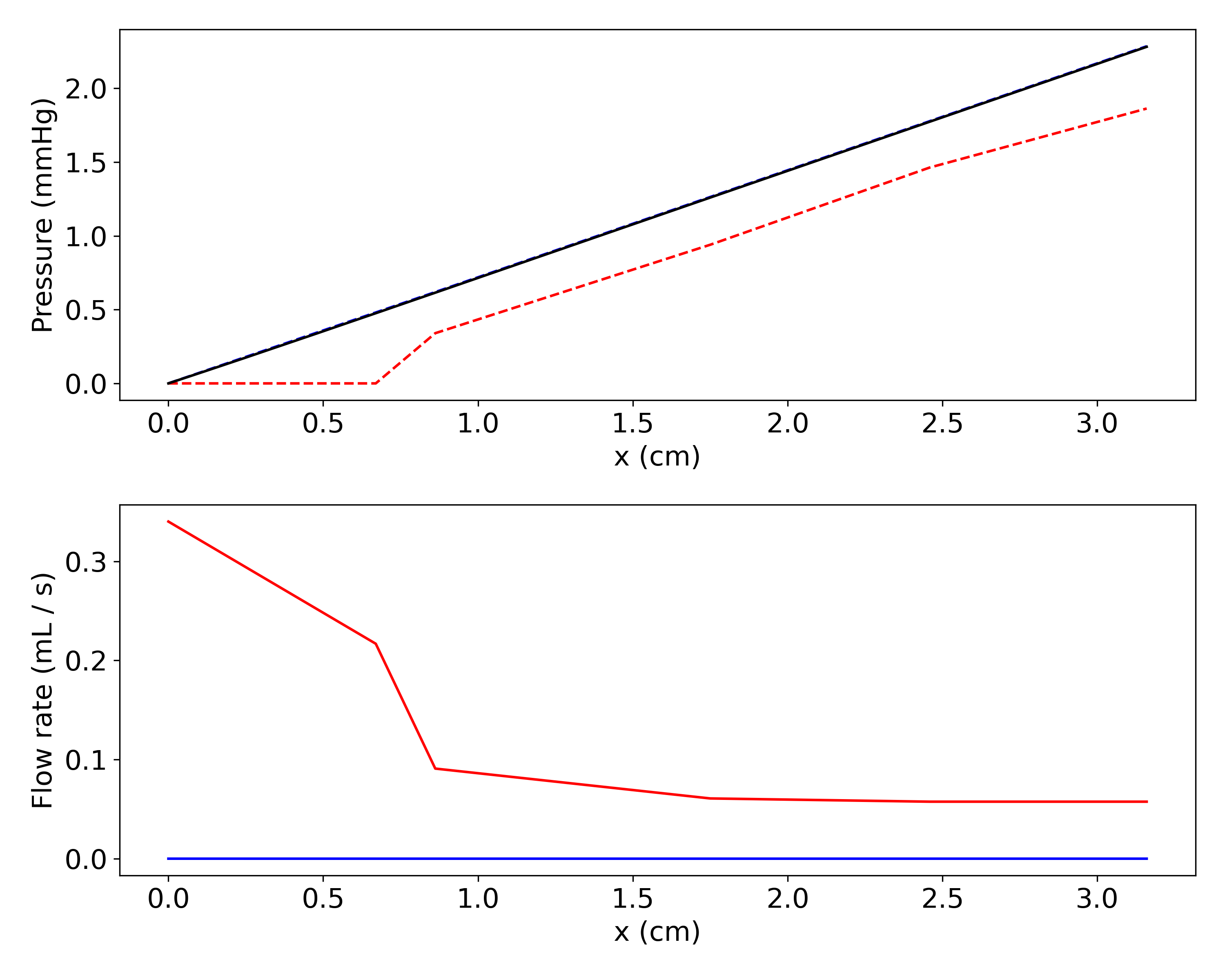}
\end{center}
\small\textbf{(c) Femoral artery (44)}
\end{minipage}
\begin{minipage}[c]{0.45\textwidth}
\begin{center}
\includegraphics[width=0.9\textwidth]{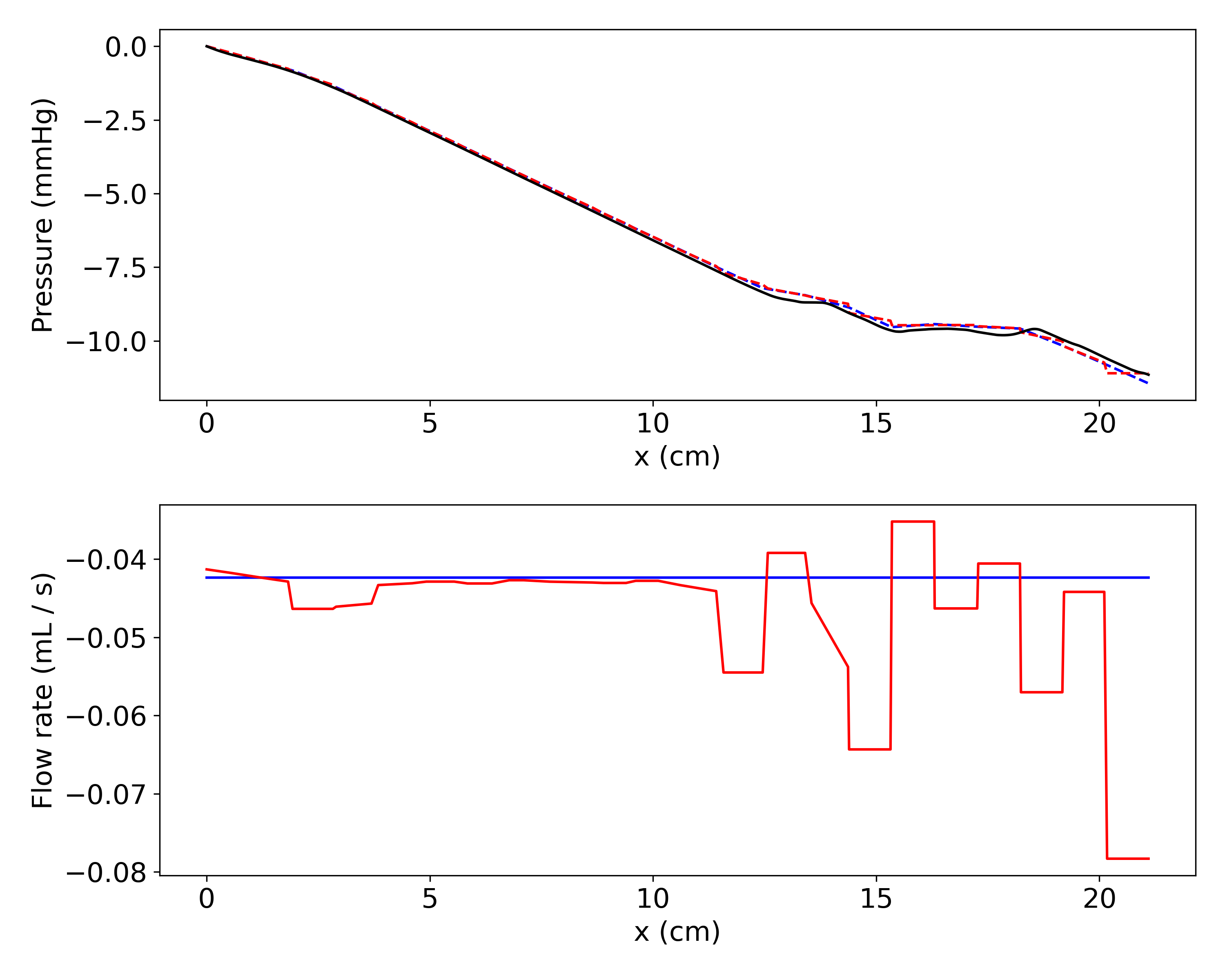}
\end{center}
\small\textbf{(d) Left vertebral artery (70)}
\end{minipage}
\begin{minipage}[c]{0.45\textwidth}
\begin{center}
\includegraphics[width=0.9\textwidth]{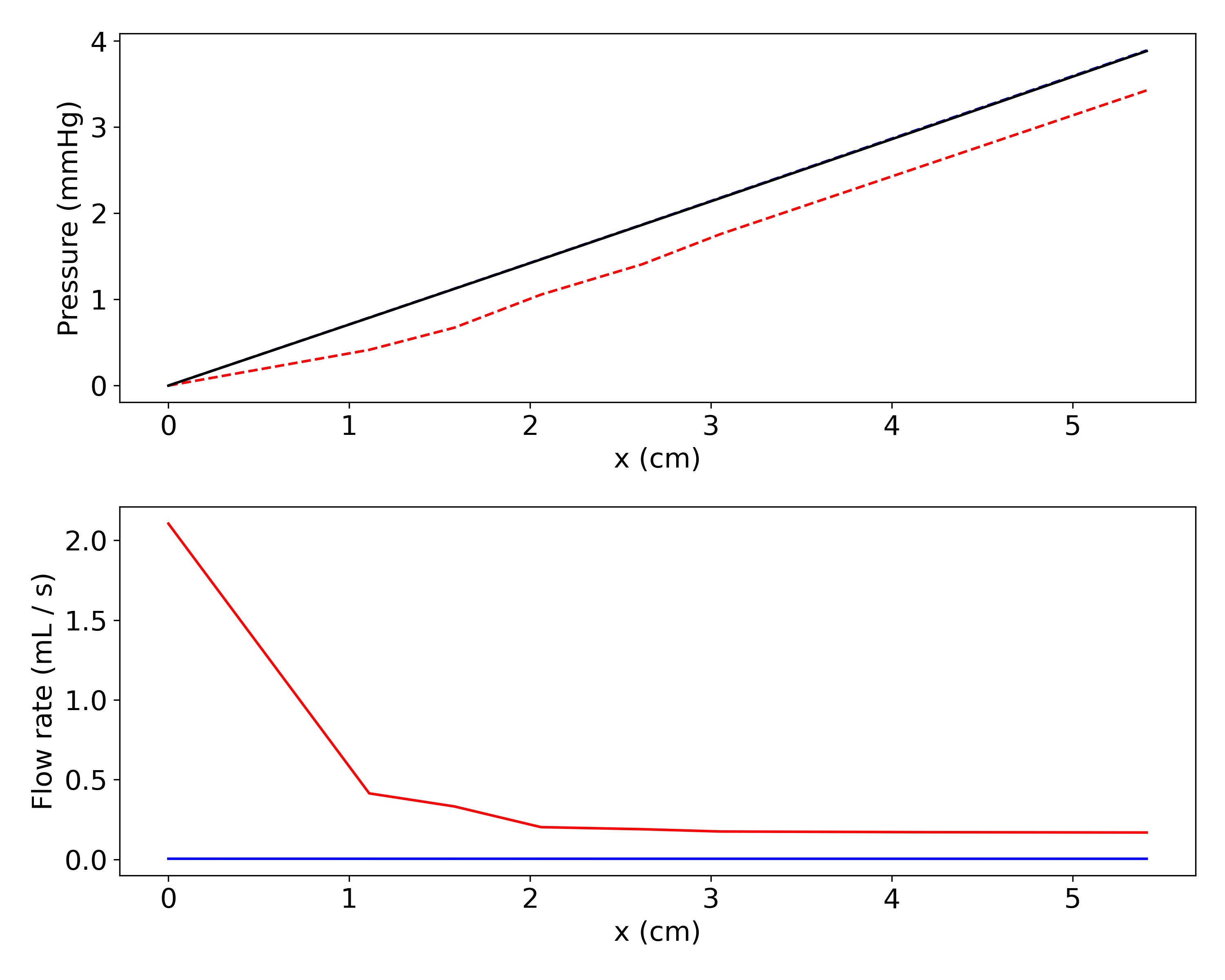}
\end{center}
\small\textbf{(e) Abdominal aorta (96)}
\end{minipage}
\begin{minipage}[c]{0.45\textwidth}
\begin{center}
\includegraphics[width=0.9\textwidth]{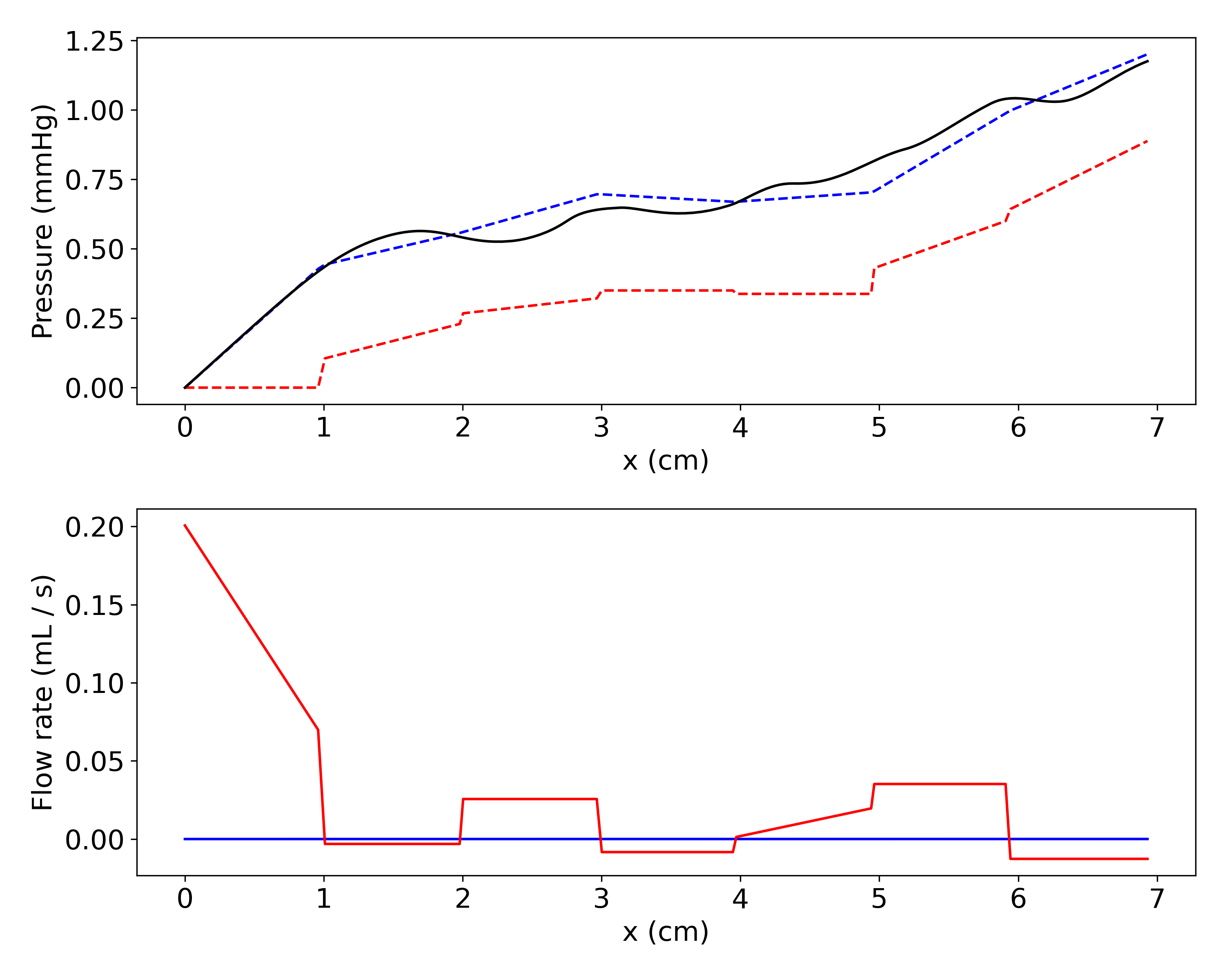}
\end{center}
\small\textbf{(f) Common hepatic artery (118)}
\end{minipage}
\caption{ADAN86 model. Deadman test with gravity. Pressure and flow rate. Pressure is shifted by initial pressure value. Hyd. pressure corresponds to hydrostatic pressure computed from the vessel inlet. Flow rate in this test is expected to be zero everywhere. Numbers in brackets represent the vessel number in Table \ref{tab:adannetwork}.}
\label{fig:deadmanxy}
\end{figure}

\paragraph{Physiological simulation in supine and upright positions}

We proceed with a test in which we show how physiologically realistic flows can be computed with the presented scheme. In particular, we perform a simulation with the second order WB version of our numerical method, in which a flow rate curve is prescribed at the inlet of the aorta (see Figure \ref{fig:adan86inflow}), while lumped-parameter models are used at terminal vessels. We compute results for two cases: a zero-gravity condition and a full-gravity one. In this last case, we set the gravitational field to have a reference acceleration equal to $\bar{g}=981 cm / s^2$, which is then projected along vessels according to their local orientation. Importantly, we set distal venous pressure $P_\mathrm{ven}=0$ for the zero-gravity case, while when gravity is acting on the network, it is computed as
\begin{equation}
P_\mathrm{ven} = \rho \, \bar{g} \, \Delta h\;,
\end{equation}
where $\Delta h$ is the height difference between the end point of the vessel connected to the lumped-parameter model and the root of the aorta. In this way we mimic the fact that in the upright position there is a hydrostatic component acting also on the venous compartment. Clearly, $P_\mathrm{ven}$ will be different for each terminal vessel, with positive values for vessels located below the aortic root and negative ones when vessels are found above it.

Figure \ref{fig:adan86pulsatile} shows computational results in terms of pressure and flow rate for the zero- and full-gravity setups. Results are shown in the form of time dependent curves, sampled at the midpoint of each vessel during the last simulation's cardiac cycle, between 9 and 10 seconds, for which a periodic state was reached. We can observe that zero-gravity results, corresponding to the supine position, show typical signatures for pressure and flow rate. In particular, the pressure waveform becomes steeper as we move away from the heart and its amplitude increases. The cardiac cycle-averaged pressure values decrease as we move distally, being this decrease the one that guarantees flow towards the periphery. Flow rate curves show also physiologically sound features, with a highly systolic aortic flow, which decreases as we move distally in terms of cardiac-cycle averaged value and amplitude, since blood is diverted to different parts of the body as we move towards the periphery. 

The highest impact of gravity on arterial haemodynamics is found in pressure waveforms. In fact, in this case pressure does not necessarily decrease as one moves distally, since there is also a gravity component driving flow. We can observe how, for the selected vessels, which are all ordered from highest to lowest in terms of elevation in the upright position, blood pressure actually increases as we move distally. Interestingly, pressure at the root of the aorta is lower in the full-gravity scenario compared to the zero-gravity one. This can be explained by the fact that the hydrostatic component in this case is negative, since this vessel is located above the aortic root. Another remarkable aspect is that since most arteries are subjected to considerably higher pressure values, the amplitude of pressure waveforms is larger. This is certainly true for vessels located at the lowest quotes, such as the femoral artery or the posterior tibial artery. The larger amplitude can be explained by the fact that the adopted tube law is nonlinear, with increasingly stiffer behaviour for larger pressure values. Compared to the impact of gravity on pressure waveforms, flow is considerably less affected in terms of waveform signatures and average values. This last aspect, i.e. that gravity does not affect blood flow distribution significantly, is consistent with the fact that we are not altering peripheral resistances of lumped-parameter models, and that we accurately account for the fact that venous beds are also subjected to a hydrostatic pressure component in the upright position.

We conclude this section by noting that the full-gravity scenario presented here represents a simplified version of the effect of orthostatic stress on the cardiovascular system. In a real human being, orthostatic stress would trigger a control mechanism called baroreflex that ensures that arterial pressure at the level of carotid arteries and aorta remains within narrow ranges. In order to do so, the central nervous system would act on properties of the cardiovascular system, such as heart rate and contractility, venous tone and peripheral vasoconstriction \cite{Herring:2018b}. While the goal of this test in the context of our work was to show that the proposed numerical methods are able to deal with realistic haemodynamic conditions, much more is needed to obtain fully physiologically sound simulations of the orthostatic stress' impact on the cardiovascular system. See \cite{foisCardiovascularResponsePosture2022b} for the only available computational study including one-dimensional blood flow models, control mechanisms and the orthostatic stress.

\begin{figure}
\centering
\includegraphics[width=0.4\textwidth]{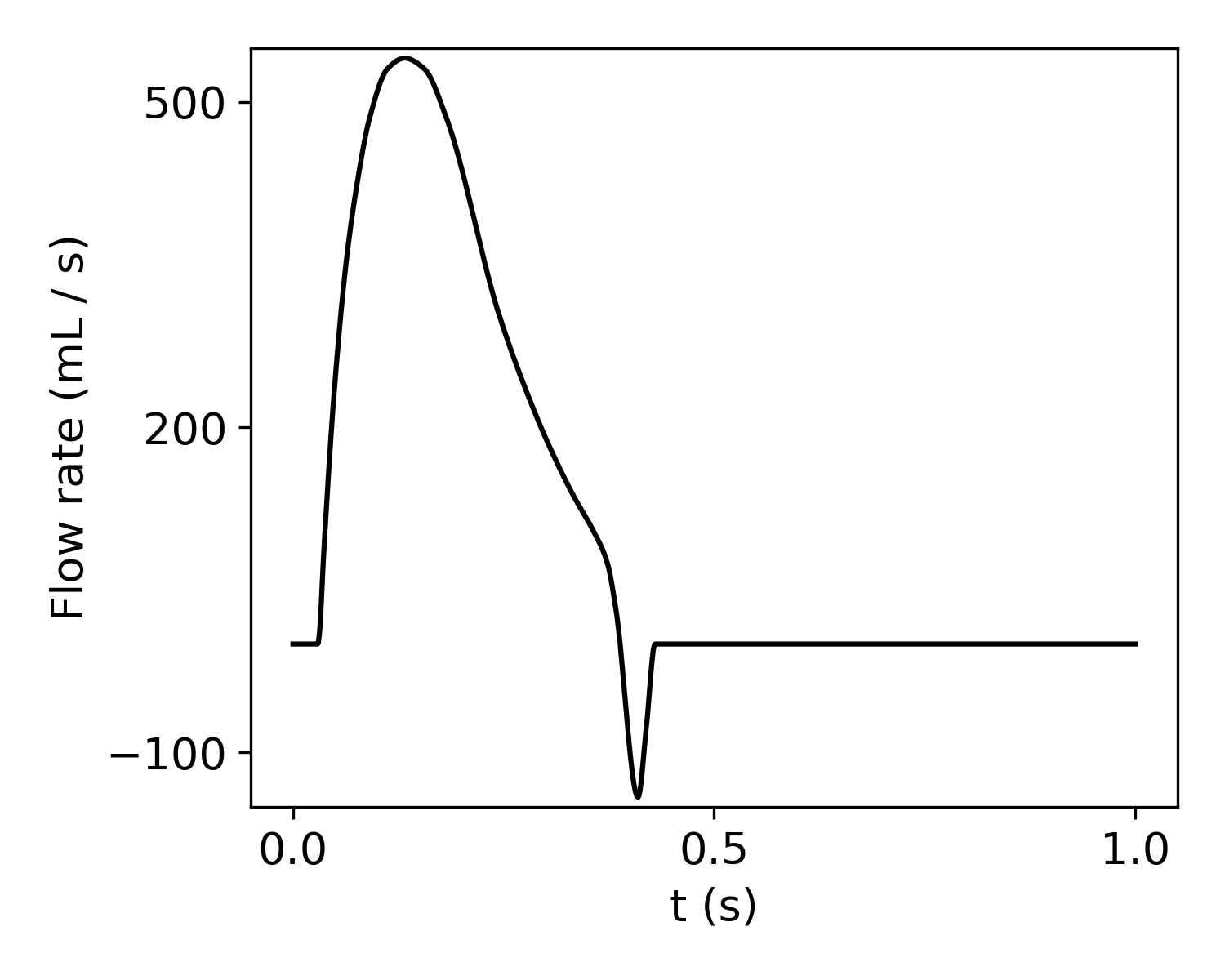} 
\caption{ADAN86 model. Supine and upright position pulsatile simulations. Prescribed flow rate curve at the aortic root.}
\label{fig:adan86inflow}
\end{figure}

\begin{figure}
\centering
\includegraphics[width=0.9\textwidth]{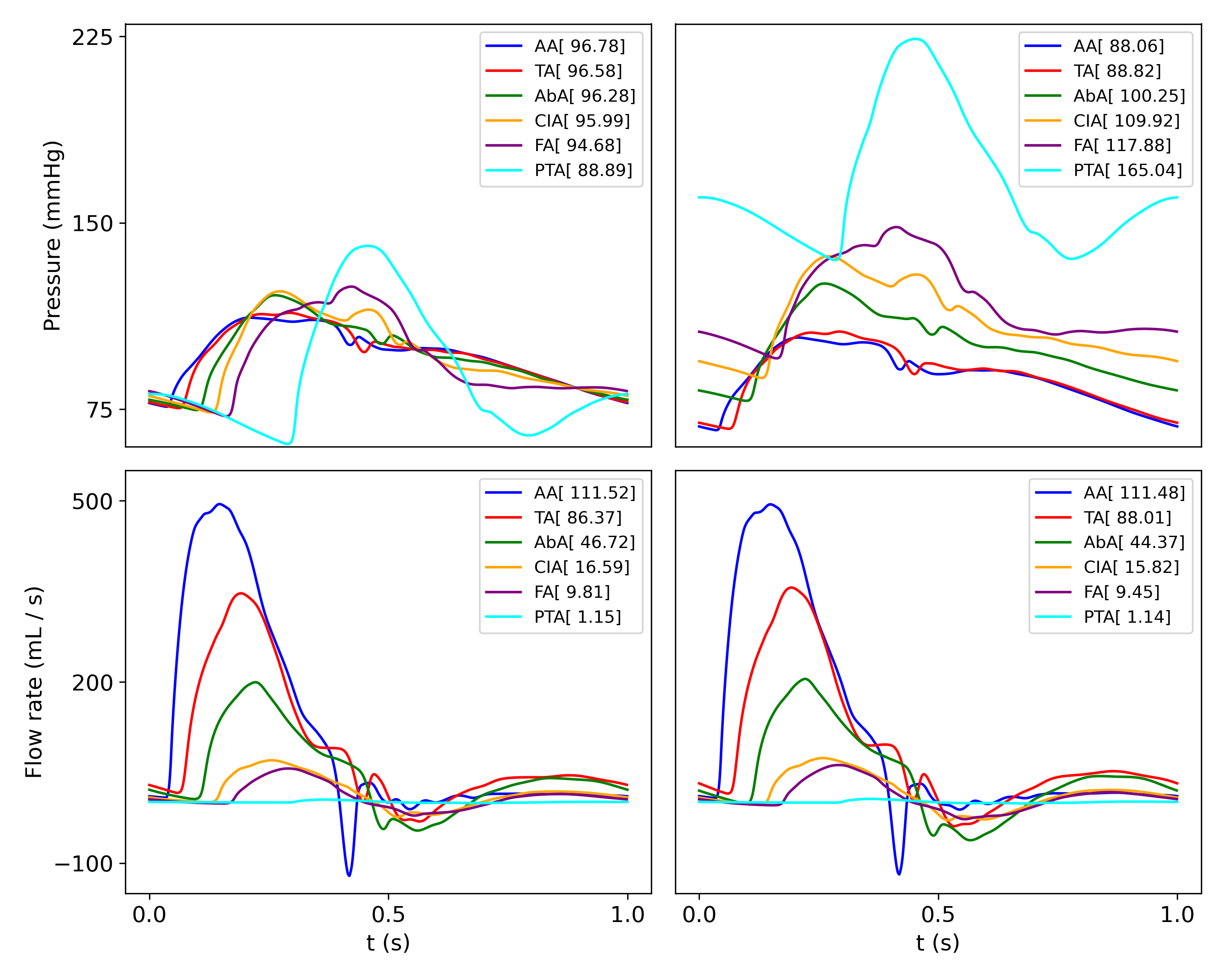} 
\caption{ADAN86 model. Supine and upright position pulsatile simulations. Pressure (top row) and flow rate (bottom row) sampled at the midpoint of selected vessels. The left column corresponds to the zero-gravity case while the right column shows results for the full-gravity setup. Numbers in square brackets represent cardiac cycle-averaged quantities. AA: ascending aorta (98); TA: thoracic aorta (104); AbA: abdominal aorta (95); CIA: common femoral artery (71); FA: femoral artery (35); PTA: posterior tibial artery (38). Numbers in brackets correspond to vessel number reported in Table \ref{tab:adannetwork}.}
\label{fig:adan86pulsatile}
\end{figure}

}

\section{Conclusions} \label{section--conclusions}
 
{\blue In this work we addressed the construction of well-balanced numerical schemes for blood flow in vessels with discontinuous properties and algebraic source terms. Since steady state solutions were not available in explicit or implicit form, we used numerically constructed ones, and designed schemes that preserve such solutions accurately. The proposed methodology was then extended in order to model blood flow in networks of vessels, requiring that coupling and boundary conditions are properly specified. We assessed the capacity of the presented methods to preserve steady state solutions, as well as to compute transient ones. Notably, we did so with academic tests, which allow for a detailed and control testing environment, as well as with realistic conditions. This second type of tests allowed us to assess the impact of well-balancing on real-life scenarios, and to explore the capacity of the proposed methods to cope with haemodynamic conditions taking place in such situations. Future work will regard the incorporation of viscoelastic vessel wall mechanics and the exploration of implicit-explicit time integration strategies. The first topic is motivated by the need for a tube law model that better reflects the mechanical properties of blood vessels, while the second one is a requirement dictated by the highly heterogeneous character of spatial scales involved in complex networks of blood vessels.}

\section*{Acknowledgments}
Ernesto Pimentel-Garc\'ia and Carlos Par\'es have been supported by project PID2022-137637NB-C21 funded by MCIN/AEI/10.13039/501100011033/
and ERDF A way of making Europe. Ernesto Pimentel-Garc\'ia was also financed by the European Union (NextGenerationEU).
Lucas O. Müller is a member of the ”Gruppo Nazionale per il Calcolo Scientifico dell’ Istituto Nazionale di Alta Matematica” (INdAM-GNCS, Italy). He also acknowledges funding by the European Union under NextGenerationEU, Mission 4, Component 2 - PRIN 2022 (D.D. 104/22), project title: Immersed methods for multIscale and multiphysics problems, CUP: E53D2300592
0006.

\bibliographystyle{abbrv}
\bibliography{references}

\appendix

\section{Appendix A}

   {\tiny
  	\begin{longtable}{|c|c|c|c|c|c|c|c|}

  		\hline 
  	Name & Inlet node	&	Outlet node	&	$L$	[$cm$] &	$R_0$ [$cm$] &	$R^\mathrm{prox}$ [$dyn \; s / cm^5$]&	$R^\mathrm{dist}$	[$dyn \; s / cm^5$] &	$C$	[$cm^5 / dyn$]\\
  	\hline
T\_hepatic\_artery\_proper\_left\_branch\_UBC00\_T7\_C\_0	&	0	&	1	&	16.42	&	0.12	&	8520.85	&	34083.39	&	6.03E-06	\\
T\_hepatic\_artery\_proper\_right\_branch\_UBC00\_T6\_C\_1	&	0	&	2	&	8.01	&	0.14	&	4714.38	&	18857.54	&	1.09E-05	\\
A\_hepatic\_artery\_proper\_C\_2	&	3	&	0	&	1.68	&	0.18	&	-	&	-	&	-	\\
T\_dorsal\_pancreatic\_C\_UBC00\_T1\_C\_3	&	4	&	5	&	3.35	&	0.06	&	166313.37	&	665253.50	&	1.51E-06	\\
T\_superior\_segmental\_R\_UBC00\_T4\_R\_4	&	6	&	7	&	2.97	&	0.19	&	5278.23	&	21112.93	&	9.73E-06	\\
A\_renal\_anterior\_branch\_R\_5	&	6	&	8	&	1.09	&	0.25	&	-	&	-	&	-	\\
T\_renal\_posterior\_branch\_R\_UBC00\_T3\_R\_6	&	8	&	9	&	2.24	&	0.16	&	9344.31	&	37377.24	&	5.49E-06	\\
T\_inferior\_segmental\_R\_UBC00\_T5\_R\_7	&	6	&	10	&	4.09	&	0.19	&	5278.23	&	21112.93	&	9.73E-06	\\
A\_renal\_anterior\_branch\_L\_8	&	11	&	12	&	1.09	&	0.25	&	-	&	-	&	-	\\
T\_renal\_posterior\_branch\_L\_UBC00\_T3\_L\_9	&	12	&	13	&	2.24	&	0.16	&	9313.04	&	37252.17	&	5.51E-06	\\
T\_inferior\_segmental\_L\_UBC00\_T5\_L\_10	&	11	&	14	&	4.09	&	0.19	&	5260.57	&	21042.28	&	9.76E-06	\\
T\_superior\_segmental\_L\_UBC00\_T4\_L\_11	&	11	&	15	&	2.97	&	0.19	&	5260.57	&	21042.28	&	9.76E-06	\\
T\_jejunal\_6\_C\_UBC00\_T11\_C\_12	&	16	&	17	&	6.39	&	0.16	&	19695.21	&	78780.84	&	2.61E-06	\\
T\_ileal\_4\_C\_UBC00\_T12\_C\_13	&	18	&	19	&	4.57	&	0.18	&	13296.36	&	53185.44	&	3.86E-06	\\
T\_ileal\_6\_C\_UBC00\_T13\_C\_14	&	20	&	21	&	2.92	&	0.18	&	13296.36	&	53185.44	&	3.86E-06	\\
T\_jejunal\_3\_C\_UBC00\_T10\_C\_15	&	22	&	23	&	4.77	&	0.16	&	19695.21	&	78780.84	&	2.61E-06	\\
T\_middle\_colic\_C\_UBC00\_T8\_C\_16	&	24	&	25	&	11.63	&	0.14	&	26865.03	&	107460.10	&	1.91E-06	\\
T\_ileocolic\_C\_UBC00\_T9\_C\_17	&	26	&	27	&	4.68	&	0.20	&	9717.22	&	38868.86	&	5.28E-06	\\
T\_anterior\_cerebral\_R\_UAR00\_T1\_R\_18	&	28	&	29	&	1.07	&	0.10	&	-	&	-	&	-	\\
T\_anterior\_cerebral\_R\_UAR00\_T1\_R\_19	&	29	&	30	&	2.93	&	0.10	&	12196.03	&	48784.12	&	2.79E-06	\\
A\_posterior\_communicating\_R\_20	&	31	&	32	&	1.67	&	0.05	&	-	&	-	&	-	\\
T\_middle\_cerebral\_R\_UAR00\_T3\_R\_21	&	28	&	33	&	3.03	&	0.10	&	10714.11	&	42856.44	&	3.40E-06	\\
A\_anterior\_communicating\_C\_22	&	29	&	34	&	0.57	&	0.07	&	-	&	-	&	-	\\
A\_posterior\_cerebral\_precommunicating\_part\_R\_23	&	35	&	32	&	0.79	&	0.08	&	-	&	-	&	-	\\
T\_posterior\_cerebral\_postcommunicating\_R\_UAR00\_T4\_R\_24	&	32	&	36	&	1.28	&	0.09	&	26297.04	&	105188.16	&	1.26E-06	\\
A\_basilar\_C\_25	&	35	&	37	&	2.27	&	0.17	&	-	&	-	&	-	\\
T\_anterior\_cerebral\_L\_UAR00\_T1\_L\_26	&	38	&	34	&	1.07	&	0.10	&	-	&	-	&	-	\\
T\_anterior\_cerebral\_L\_UAR00\_T1\_L\_27	&	34	&	39	&	2.93	&	0.10	&	12196.03	&	48784.12	&	2.79E-06	\\
A\_posterior\_communicating\_L\_28	&	40	&	41	&	1.67	&	0.05	&	-	&	-	&	-	\\
T\_middle\_cerebral\_L\_UAR00\_T3\_L\_29	&	38	&	42	&	3.03	&	0.10	&	10714.11	&	42856.44	&	3.40E-06	\\
A\_posterior\_cerebral\_precommunicating\_part\_L\_30	&	35	&	41	&	0.71	&	0.08	&	-	&	-	&	-	\\
T\_posterior\_cerebral\_postcommunicating\_L\_UAR00\_T4\_L\_31	&	41	&	43	&	1.28	&	0.09	&	26297.04	&	105188.16	&	1.26E-06	\\
A\_popliteal\_L\_32	&	44	&	45	&	13.21	&	0.27	&	-	&	-	&	-	\\
A\_popliteal\_L\_33	&	45	&	46	&	0.88	&	0.24	&	-	&	-	&	-	\\
A\_femoral\_L\_34	&	47	&	48	&	3.16	&	0.32	&	-	&	-	&	-	\\
A\_femoral\_L\_35	&	48	&	44	&	31.93	&	0.31	&	-	&	-	&	-	\\
T\_profunda\_femoris\_L\_UDL00\_T2\_L\_36	&	48	&	49	&	23.84	&	0.21	&	3107.78	&	12431.10	&	1.65E-05	\\
T\_posterior\_tibial\_L\_UDL00\_T4\_L\_37	&	50	&	51	&	38.30	&	0.12	&	19184.48	&	76737.94	&	2.68E-06	\\
A\_tibiofibular\_trunk\_L\_38	&	46	&	50	&	3.62	&	0.23	&	-	&	-	&	-	\\
T\_anterior\_tibial\_L\_UDL00\_T3\_L\_39	&	45	&	52	&	38.64	&	0.12	&	22236.54	&	88946.18	&	2.31E-06	\\
A\_popliteal\_R\_40	&	53	&	54	&	13.21	&	0.27	&	-	&	-	&	-	\\
A\_popliteal\_R\_41	&	54	&	55	&	0.88	&	0.24	&	-	&	-	&	-	\\
T\_profunda\_femoris\_R\_UDR00\_T2\_R\_42	&	56	&	57	&	23.84	&	0.21	&	3106.12	&	12424.49	&	1.65E-05	\\
A\_femoral\_R\_43	&	58	&	56	&	3.16	&	0.32	&	-	&	-	&	-	\\
A\_femoral\_R\_44	&	56	&	53	&	31.93	&	0.31	&	-	&	-	&	-	\\
T\_posterior\_tibial\_R\_UDR00\_T4\_R\_45	&	59	&	60	&	38.30	&	0.12	&	19176.12	&	76704.50	&	2.68E-06	\\
A\_tibiofibular\_trunk\_R\_46	&	55	&	59	&	3.62	&	0.23	&	-	&	-	&	-	\\
T\_anterior\_tibial\_R\_UDR00\_T3\_R\_47	&	54	&	61	&	38.64	&	0.12	&	22229.57	&	88918.28	&	2.31E-06	\\
A\_brachial\_L\_48	&	62	&	63	&	22.31	&	0.21	&	-	&	-	&	-	\\
A\_axillary\_L\_49	&	64	&	62	&	12.00	&	0.23	&	-	&	-	&	-	\\
T\_posterior\_interosseuous\_L\_UCL00\_T3\_L\_50	&	65	&	66	&	23.17	&	0.07	&	43494.79	&	173979.17	&	1.18E-06	\\
T\_ulnar\_L\_UCL00\_T2\_L\_51	&	63	&	67	&	2.98	&	0.14	&	-	&	-	&	-	\\
T\_ulnar\_L\_UCL00\_T2\_L\_52	&	67	&	68	&	23.93	&	0.14	&	10855.29	&	43421.17	&	4.73E-06	\\
T\_radial\_L\_UCL00\_T1\_L\_53	&	63	&	69	&	30.22	&	0.14	&	10271.50	&	41085.98	&	5.00E-06	\\
A\_common\_interosseous\_L\_54	&	67	&	65	&	1.63	&	0.10	&	-	&	-	&	-	\\
A\_brachial\_R\_55	&	70	&	71	&	22.31	&	0.21	&	-	&	-	&	-	\\
A\_axillary\_R\_56	&	72	&	70	&	12.00	&	0.23	&	-	&	-	&	-	\\
T\_posterior\_interosseuous\_R\_UCR00\_T3\_R\_57	&	73	&	74	&	23.17	&	0.07	&	43337.80	&	173351.20	&	1.18E-06	\\
T\_ulnar\_R\_UCR00\_T2\_R\_58	&	71	&	75	&	2.98	&	0.14	&	-	&	-	&	-	\\
T\_ulnar\_R\_UCR00\_T2\_R\_59	&	75	&	76	&	23.93	&	0.14	&	10648.96	&	42595.86	&	4.82E-06	\\
T\_radial\_R\_UCR00\_T1\_R\_60	&	71	&	77	&	30.22	&	0.14	&	10458.77	&	41835.10	&	4.91E-06	\\
A\_common\_interosseous\_R\_61	&	75	&	73	&	1.63	&	0.10	&	-	&	-	&	-	\\
A\_common\_carotida\_L\_62	&	78	&	79	&	12.14	&	0.33	&	-	&	-	&	-	\\
A\_common\_carotida\_R\_63	&	80	&	81	&	8.13	&	0.33	&	-	&	-	&	-	\\
A\_vertebral\_L\_64	&	82	&	37	&	21.08	&	0.13	&	-	&	-	&	-	\\
A\_internal\_carotid\_L\_65	&	38	&	40	&	0.20	&	0.13	&	-	&	-	&	-	\\
A\_internal\_carotid\_L\_66	&	40	&	78	&	13.51	&	0.13	&	-	&	-	&	-	\\
A\_internal\_carotid\_R\_67	&	28	&	31	&	0.20	&	0.13	&	-	&	-	&	-	\\
A\_internal\_carotid\_R\_68	&	31	&	80	&	13.51	&	0.13	&	-	&	-	&	-	\\
A\_vertebral\_R\_69	&	83	&	37	&	21.10	&	0.13	&	-	&	-	&	-	\\
A\_common\_iliac\_L\_70	&	84	&	85	&	7.41	&	0.45	&	-	&	-	&	-	\\
T\_internal\_iliac\_L\_UDL00\_T1\_L\_71	&	85	&	86	&	7.25	&	0.28	&	3768.79	&	15075.17	&	1.36E-05	\\
A\_external\_iliac\_L\_72	&	85	&	47	&	10.24	&	0.34	&	-	&	-	&	-	\\
A\_common\_iliac\_R\_73	&	84	&	87	&	7.64	&	0.45	&	-	&	-	&	-	\\
T\_internal\_iliac\_R\_UDR00\_T1\_R\_74	&	87	&	88	&	7.25	&	0.28	&	3757.54	&	15030.17	&	1.37E-05	\\
A\_external\_iliac\_R\_75	&	87	&	58	&	10.24	&	0.34	&	-	&	-	&	-	\\
T\_external\_carotid\_L\_UAL00\_T2\_L\_76	&	78	&	89	&	6.10	&	0.23	&	8541.83	&	34167.32	&	6.01E-06	\\
T\_external\_carotid\_R\_UAR00\_T2\_R\_77	&	80	&	90	&	6.10	&	0.23	&	8511.83	&	34047.31	&	6.03E-06	\\
T\_inferior\_mesenteric\_C\_UBC00\_T5\_C\_78	&	91	&	92	&	9.03	&	0.21	&	21674.64	&	86698.57	&	2.37E-06	\\
T\_superior\_mesenteric\_C\_UBC00\_T4\_C\_79	&	93	&	24	&	4.95	&	0.39	&	-	&	-	&	-	\\
T\_superior\_mesenteric\_C\_UBC00\_T4\_C\_80	&	24	&	23	&	3.52	&	0.35	&	-	&	-	&	-	\\
T\_superior\_mesenteric\_C\_UBC00\_T4\_C\_81	&	23	&	17	&	3.22	&	0.32	&	-	&	-	&	-	\\
T\_superior\_mesenteric\_C\_UBC00\_T4\_C\_82	&	17	&	27	&	1.67	&	0.29	&	-	&	-	&	-	\\
T\_superior\_mesenteric\_C\_UBC00\_T4\_C\_83	&	27	&	19	&	2.31	&	0.28	&	-	&	-	&	-	\\
T\_superior\_mesenteric\_C\_UBC00\_T4\_C\_84	&	19	&	21	&	2.05	&	0.26	&	-	&	-	&	-	\\
T\_superior\_mesenteric\_C\_UBC00\_T4\_C\_85	&	21	&	94	&	3.95	&	0.24	&	8807.45	&	35229.78	&	5.83E-06	\\
A\_renal\_L\_86	&	95	&	12	&	2.20	&	0.28	&	-	&	-	&	-	\\
A\_renal\_R\_87	&	96	&	8	&	3.77	&	0.31	&	-	&	-	&	-	\\
T\_splenic\_C\_UBC00\_T2\_C\_88	&	97	&	98	&	0.40	&	0.22	&	-	&	-	&	-	\\
T\_splenic\_C\_UBC00\_T2\_C\_89	&	98	&	4	&	0.28	&	0.22	&	-	&	-	&	-	\\
T\_splenic\_C\_UBC00\_T2\_C\_90	&	4	&	99	&	6.33	&	0.22	&	4290.28	&	17161.11	&	1.20E-05	\\
A\_abdominal\_aorta\_C\_91	&	100	&	101	&	0.32	&	0.75	&	-	&	-	&	-	\\
A\_abdominal\_aorta\_C\_92	&	101	&	93	&	1.40	&	0.75	&	-	&	-	&	-	\\
A\_abdominal\_aorta\_C\_93	&	93	&	95	&	0.43	&	0.73	&	-	&	-	&	-	\\
A\_abdominal\_aorta\_C\_94	&	95	&	96	&	1.20	&	0.73	&	-	&	-	&	-	\\
A\_abdominal\_aorta\_C\_95	&	96	&	91	&	5.41	&	0.71	&	-	&	-	&	-	\\
A\_abdominal\_aorta\_C\_96	&	91	&	84	&	4.22	&	0.64	&	-	&	-	&	-	\\
A\_aortic\_arch\_C\_97	&	102	&	103	&	7.45	&	1.60	&	-	&	-	&	-	\\
A\_aortic\_arch\_C\_98	&	103	&	79	&	0.96	&	1.30	&	-	&	-	&	-	\\
A\_aortic\_arch\_C\_99	&	79	&	104	&	0.70	&	1.26	&	-	&	-	&	-	\\
A\_aortic\_arch\_C\_100	&	104	&	105	&	4.32	&	1.23	&	-	&	-	&	-	\\
A\_thoracic\_aorta\_C\_101	&	105	&	106	&	0.99	&	1.06	&	-	&	-	&	-	\\
A\_thoracic\_aorta\_C\_102	&	106	&	107	&	0.79	&	1.04	&	-	&	-	&	-	\\
A\_thoracic\_aorta\_C\_103	&	107	&	108	&	1.56	&	1.02	&	-	&	-	&	-	\\
A\_thoracic\_aorta\_C\_104	&	108	&	109	&	0.53	&	0.99	&	-	&	-	&	-	\\
A\_thoracic\_aorta\_C\_105	&	109	&	100	&	12.16	&	0.98	&	-	&	-	&	-	\\
A\_brachiocephalic\_trunk\_C\_106	&	81	&	103	&	4.74	&	0.62	&	-	&	-	&	-	\\
A\_subclavian\_R\_107	&	81	&	83	&	1.57	&	0.49	&	-	&	-	&	-	\\
A\_subclavian\_R\_108	&	83	&	72	&	4.11	&	0.42	&	-	&	-	&	-	\\
A\_subclavian\_L\_109	&	104	&	82	&	4.95	&	0.49	&	-	&	-	&	-	\\
A\_subclavian\_L\_110	&	82	&	64	&	4.11	&	0.35	&	-	&	-	&	-	\\
A\_celiac\_trunk\_C\_111	&	101	&	97	&	1.69	&	0.33	&	-	&	-	&	-	\\
T\_posterior\_intercostal\_T6\_L\_UBC00\_T1\_L\_112	&	110	&	107	&	17.85	&	0.14	&	231660.66	&	926642.64	&	2.22E-07	\\
T\_posterior\_intercostal\_T7\_L\_UBC00\_T2\_L\_113	&	111	&	109	&	18.55	&	0.16	&	212483.26	&	849933.04	&	2.42E-07	\\
T\_posterior\_intercostal\_T6\_R\_UBC00\_T1\_R\_114	&	112	&	106	&	19.72	&	0.14	&	225808.66	&	903234.64	&	2.27E-07	\\
T\_posterior\_intercostal\_T7\_R\_UBC00\_T2\_R\_115	&	113	&	108	&	20.19	&	0.16	&	210677.94	&	842711.76	&	2.44E-07	\\
T\_left\_gastric\_C\_UBC00\_T3\_C\_116	&	98	&	114	&	9.48	&	0.15	&	311250.72	&	1245002.88	&	1.65E-07	\\
A\_common\_hepatic\_C\_117	&	97	&	3	&	6.93	&	0.27	&	-	&	-	&	-	\\
\hline
 \caption{Network connectivity and vessel geometry for ADAN86 model.} \label{tab:adannetwork}

  	\end{longtable} 
	 }

\end{document}